\documentclass[10pt]{article}
\usepackage[T1]{fontenc}
\usepackage[utf8]{inputenc}
\usepackage[table]{xcolor}
\usepackage{amsmath,amsthm,graphicx,mathtools,amssymb}
\usepackage{wrapfig,subcaption,multicol}
\usepackage{float}
\input{insbox}
\usepackage{tikz}
\usetikzlibrary{patterns,patterns.meta,arrows.meta,calc,3d,perspective,fadings,shadows,cd}
\usepackage[a4paper, width=160mm, top=25mm, bottom=25mm]{geometry}
\numberwithin{equation}{section}
\usepackage[colorlinks, allcolors=black]{hyperref}

\theoremstyle{plain}
	\newtheorem{thm}{Theorem}[section]
	\newtheorem{cor}[thm]{Corollary}
	\newtheorem{lem}[thm]{Lemma}
	\newtheorem{prop}[thm]{Proposition}
	
	\newtheorem*{intro1}{Proposition~\ref{prop:dimension_real_locus}}
	\newtheorem*{intro2}{Proposition~\ref{prop:compactness_real_locus}}
	\newtheorem*{intro3}{Theorem~\ref{thm:fibration_affine_real_loci}}
	\newtheorem*{intro4}{Proposition~\ref{prop:line_dec_affine}}
	\newtheorem*{intro5}{Theorem~\ref{thm:can_fib}}
	\newtheorem*{intro6}{Proposition~\ref{prop:unwinding_quot}}
	\newtheorem*{intro7}{Corollary~\ref{cor:untwisted_compact_connected}}
	\newtheorem*{intro8}{Proposition~\ref{prop:blow_up_W}}
	\newtheorem*{intro9}{Proposition~\ref{prop:virt_bett_numb}}
	\newtheorem*{intro10}{Theorem~\ref{thm:alg_coho} and Corollary~\ref{cor:alg_coho}}
	\newtheorem*{intro11}{Theorem~\ref{thm:orientability}}
	\newtheorem*{intro12}{Proposition~\ref{prop:topo_121}}
	\newtheorem*{intro13}{Theorem~\ref{thm:homeo_real_loc}}	
	
\theoremstyle{definition}
    
	\newtheorem{dfn}[thm]{Definition}

	\newtheorem{rem}[thm]{Remark}
	
	\newtheorem{ex}[thm]{Example}
	
	\newtheorem{exs}[thm]{Examples}

\newcommand{\bigslant}[2]{{\raisebox{.2em}{$#1$}\left/\raisebox{-.2em}{$#2$}\right.}}
\newcommand{\smallslant}[2]{{\raisebox{.1em}{$#1$}\left/\raisebox{-.1em}{$#2$}\right.}}
\newcommand{\C}{\mathbb{C}}
\newcommand{\R}{\mathbb{R}}

\newcommand{\Z}{\mathbb{Z}}
\newcommand{\N}{\mathbb{N}}
\newcommand{\F}{\mathbb{F}}
\newcommand{\Proj}{\mathbb{P}}

\newcommand{\pt}{\textnormal{pt}}
\newcommand{\Hom}{\textnormal{Hom}}
\newcommand{\Ext}{\textnormal{Ext}}
\newcommand{\End}{\textnormal{End}}
\newcommand{\Aut}{\textnormal{Aut}}
\newcommand{\id}{\textnormal{id}}

\renewcommand{\O}{\mathcal{O}}
\newcommand{\df}{\textnormal{d}}
\newcommand{\x}{\textnormal{x}}
\renewcommand{\mod}{\textnormal{mod}\,}
\newcommand{\Sph}{\textnormal{S}}
\newcommand{\im}{\textnormal{im}}
\newcommand{\rk}{\textnormal{rk}\,}
\newcommand{\GL}{\textnormal{GL}}
\newcommand{\Spec}{\textnormal{Spec}}
\newcommand{\A}{\mathbb{A}}
\newcommand{\Bl}{\textnormal{Bl}\,}
\newcommand{\cl}{c\ell}
\newcommand{\VP}{\beta}
\newcommand{\G}[1]{\mathbb{G}_{\textnormal{m},#1}}
\newcommand{\Res}{\textnormal{Res}_{\C/\R}}
\newcommand{\Ch}{\textit{CH}\,}
\newcommand{\SO}{\textnormal{SO}_{2,\R}}
\newcommand{\te}{equivariant torus embedding}
\newcommand{\cf}{canonical fibre}
\newcommand{\tc}{topological core}
\newcommand{\dl}{orbital lattice}
\newcommand{\tor}[1]{\hat{#1}}
\newcommand{\Vab}[2]{\textnormal{V}_{#1,#2}}

\title{{\bf Real Toric Varieties} \\ Interactions between \\their Geometry and their Topology}
\author{Jules Chenal\footnote{Université de Lille, France, jules.chenal@univ-lille.fr}\hspace{.25cm}and Matilde Manzaroli}

\begin{document}

\maketitle

\begin{abstract}
	In the present article, we investigate the topology of real toric varieties, especially those whose torus is not split over the field of real numbers. We describe some canonical fibrations associated to their real loci. Then, we establish various properties of their cohomology provided that their real loci are compact and smooth. For instance, we compute their Betti numbers, show that their cohomology is totally algebraic, and extend a criterion of orientability. In addition, we provide the topological classification of equivariant embeddings of non-split tridimensional tori.  
\end{abstract}

\renewcommand{\abstractname}{Acknowledgements}
\begin{abstract}
	This research was funded, in whole or in part, by l’Agence Nationale de la Recherche (ANR), project ANR-22-CE40-0014. For the purpose of open access, the author has applied a CC-BY public copyright licence to any Author Accepted Manuscript (AAM) version arising from this submission.
\end{abstract}

\tableofcontents

\newpage 
\section*{Introduction}

The aim of this article is to describe the topology of the real loci of real toric varieties and real \te s. We can recall quickly that an \te\;is a toric variety whose principal homogeneous open subspace is a trivial torsor, together with a choice of a rational point in this principal homogeneous open subspace. The real loci of such objects have been well understood when the acting torus $T$ is \emph{split}. It means that $T$ is isomorphic to a product of real multiplicative groups. For instance, if $X$ is a real toric variety under the action of $\G{\R}^n$, then its real locus is obtained by gluing together $2^n$ copies of a “cell”. When $X$ is complete, the cell is an actual cellular structure on a closed ball of dimension $n$. This cellular construction has been generalised to the notion of \emph{small covers}. We can cite \cite{Davis:1991aa} for the study of the topological properties of such objects. Here, we allow different forms of tori. We refer to \cite{Huruguen:2011aa}, \cite{Elizondo:2014aa} for an algebraic study of toric varieties under the action of (non-necessarily split) tori. The first topological study of the real loci of real toric varieties under the action of general forms of real tori was conducted by C. Delaunay in her thesis \cite{delau2004}. The philosophy of this text is to be at the interface between the geometry and the topology of real toric varieties. As such, we will always try to provide statements that are both geometrically and topologically significant. The article is split into five sections. In the first section, we give a recollection of the definitions and first properties of toric varieties. In the second, we investigate fibration properties of affine varieties. Then, we move on to general varieties. The fourth section treats of cycles and cohomology. In the last section, we deal with topological types of the real loci in low dimensions.

\paragraph{Toric Real Structures.} We start by giving a thorough exposition of the properties of real tori. In particular, we provide some technical results that will only become useful later in the text. We introduce the relevant notions of cocharacter lattice $N$ and fan $C$ of a general real toric variety $T\curvearrowright X$. In the general case, the free abelian group $N$ is endowed with an involution $\tau$ which permutes the cones of $C$. When $T$ is split, the involution is the identity and the notions correspond to the usual cocharacter lattice and fan. We also consider the \emph{twist class} $[\varepsilon]$ of $X$. It is the equivalence class of the principal homogeneous open subspace of $X$ in the category of $T$-torsors. As such, it is a Galois cohomology class of $T(\C)$. These algebraic objects can be used to determine whether $X$ has a real point, and if its real locus is compact. We define the \emph{cellular dimension} of $X$ as the maximum of the integers $m$ such that there is a topological embedding of $\R^m$ into $X(\R)$. We find the following property. 
\begin{intro1}
	Let $T\curvearrowright X$ be a real toric variety with twist class $[\varepsilon]$. The real locus of $X$ is non-empty if and only if there is a toric subvariety $Z\subset X$ with isotropy $T_Z\subset T$ such that $[\varepsilon]$ belongs to the image of $H^1\big(\Z/2;T_Z(\C)\big)\rightarrow H^1\big(\Z/2;T(\C)\big)$. In this case, the cellular dimension of $X(\R)$ is given by the following expression:
	\begin{equation*}
		\max\Big\{ \dim Z\,\big|\, [\varepsilon]\in\im\, H^1\big(\Z/2;T_Z(\C)\big)\rightarrow H^1\big(\Z/2;T(\C)\big)\Big\},
	\end{equation*}
	where $Z$ ranges among the toric subvarieties of $X$.
\end{intro1}
The compacity of the real locus can be read on the fan like completeness.
\begin{intro2}
	Let $T\hookrightarrow X$ be a real \te. The real locus of $X$ is compact if and only if the group $\ker(1-\tau)\subset N$ is contained in the support of the fan of $X$.
\end{intro2}

\paragraph{Structure of Affine Varieties.} When the acting torus is split, it is well known, \textit{cf.} \cite[\S2.1]{Fulton:1993aa}, that an affine toric variety can be written as the product of a torus and a toric variety that admits a fixed point. The situation is slightly more complicated when one drops the splitness hypothesis. Let $T\curvearrowright X$ be an affine real toric variety defined by a cone $c$. We denote by $T_c$ the isotropy group of the smallest $T$-invariant subvariety of $X$, and by $T(c)$ the quotient torus $T/T_c$. 
\begin{equation*}
	1 \rightarrow T_c \longrightarrow T \overset{\pi}{\longrightarrow} T(c) \rightarrow 1.
\end{equation*}
The quotient $X/T_c$ is a principal homogeneous variety $T(c)\curvearrowright X(c)$. If $X$ admits a real point, one can associate a real toric variety $T_c\curvearrowright X^\omega_c$ to every $T(\R)$-orbit $\omega$ of the real locus of $X(c)$. The varieties $(X^\omega_c)_\omega$ are real forms of the same complex variety. These objects allow to describe $X(\R)$ as a disjoint union of locally trivial fibrations.
\begin{intro3}
	Let $T\curvearrowright X$ be an affine real toric variety with a real point. The real part of the projection $\pi:X\rightarrow X{(c)}$ splits as the disjoint union of the following locally trivial fibrations:
	\begin{equation*}
		X_{c}^{\omega}(\R) \rightarrow \pi^{-1}(\omega) \rightarrow \omega,
	\end{equation*}
	for all $T(\R)$-orbits $\omega$ of the real locus of $X{(c)}$. Furthermore, the structure group of every such fibration is $T_c(\R)$, and the associated principal bundle is given by the following exact sequence of Lie groups:
	\begin{equation*}
		1\rightarrow T_c(\R) \rightarrow T(\R) \rightarrow \pi\big(T(\R)\big)\rightarrow 1.
	\end{equation*}
	If we further assume that $\pi:T\rightarrow T{(c)}$ induces a surjection between the real loci, then:
	\begin{equation*}
		X_c\rightarrow X\rightarrow X{(c)},
	\end{equation*}
	is an algebraic fibre bundle of structure group $T_c$ and principal bundle:
	\begin{equation*}
		1\rightarrow T_c \rightarrow T \rightarrow T{(c)}\rightarrow 1.
	\end{equation*}
\end{intro3}

When $X$ is smooth and admits a real point, the morphism $\pi:T\rightarrow T(c)$ is surjective and $X(c)$ is isomorphic to $T(c)$. Furthermore, the fibres are always affine spaces of the dimension of the cone $c$. It allows us to further describe the real locus.

\begin{intro4}
	Let $T\hookrightarrow X$ be smooth affine real \te\;defined by a cone $c$. The fibre bundle $X_c\rightarrow X \rightarrow T{(c)}$ is a vector bundle. Every toric subvariety $Y$ induces a sub-vector bundle $Y\rightarrow T{(c)}$. If $Y<X$ is maximal among the toric subvarieties of $X$, then either $Y$ is a divisor and $X/Y\rightarrow T{(c)}$ is a real line bundle, or $Y$ has codimension 2 and $X/Y\rightarrow T{(c)}$ is a complex line bundle. Furthermore, the sum of the projections:
	\begin{equation*}\tag{\ref{eq:iso_line_bundle}}
		X\longrightarrow \bigoplus_{\substack{Y\textnormal{ maximal}\\ \textnormal{toric subvariety}}} \bigslant{X}{Y},
	\end{equation*}
	is an isomorphism of real vector bundles.
\end{intro4}

This proposition enables us to provide simple models of every smooth affine real toric variety, \textit{cf.} Proposition~\ref{prop:equiv_ngbhd}. 

\paragraph{Canonical Fibration and Isogeny.} Then, we move toward a topological description of real \te s as fibrations over product of $\SO$ whose fibres are equivariant embeddings of split tori. We note that every torus $T$ is endowed with a canonical exact sequence:
\begin{equation*}\tag{\ref{eq:can_fib_tori}}
 1 \longrightarrow \G{\R}^p \longrightarrow T \overset{\pi}{\longrightarrow}\SO^q \longrightarrow 1.
\end{equation*}
Our goal is to somewhat extend this sequence to real \te s. Let $T\hookrightarrow X$ be a real \te. We define the \emph{canonical fibre} $F$ of $X$ as the closure of $\G{\R}^p$ in $X$. We also introduce the \emph{topological core} $U$ of $X$ as the smallest $T$-invariant open set of $X$ that contains every real point. Using fans, we characterise \emph{properly wound} toric varieties, those varieties that satisfy $\dim U/\G{\R}^p=\dim U -p=q$. It leads us to the following theorem:

\begin{intro5}
	Let $T\hookrightarrow X$ be a properly wound real \te, $\G{\R}^p\hookrightarrow F$ be its \cf, and $U$ be its \tc. The quotient $U/\G{\R}^p$ is isomorphic to $\SO^q$, and $U$ is a fibre bundle:
	\begin{equation*}
		F \rightarrow U \rightarrow \SO^q,
	\end{equation*}
	with structure group $\G{\R}^p$, and associated principal bundle: $1\rightarrow\G{\R}^p \rightarrow T \rightarrow \SO^q \rightarrow 1.$
\end{intro5}

We call the fibration of the topological core of a properly wound \te, its \emph{canonical fibration}. In addition, every torus $T$ is endowed with a canonical isogeny:
\begin{equation*}\tag{\ref{eq:torus_can_isog}}
	1\rightarrow \Gamma_\R \longrightarrow \tilde{T} \overset{w}{\longrightarrow} T \rightarrow 1.
\end{equation*}
The torus $\tilde{T}$ is isomorphic to $\G{\R}^p\times_\R\SO^q$, and $\Gamma_\R$ is a constant finite $2$-torsion group. We can always extend the isogeny to $X$ into a finite map. We define a canonical \te\;$\tilde{T}\curvearrowright \tilde{X}$ and a morphism of \te s $w:\tilde{X}\rightarrow X$  called the \emph{unwinding} of $X$. It satisfies the following proposition:

\begin{intro6}
	Let $T\hookrightarrow X$ be a real \te. Its unwinding $w:\tilde{X}\rightarrow X$ satisfies the following properties:\begin{enumerate}
	\item[\textnormal{(i)}] $w:\tilde{X}\rightarrow X$ is the geometric quotient of $\tilde{X}$ by $\Gamma_\R$;
	\item[\textnormal{(ii)}] $w:\tilde{X}(\R)\rightarrow X(\R)$ is the topological quotient of $\tilde{X}(\R)$ by $\Gamma$;
	\item[\textnormal{(iii)}] $w$ is \emph{totally real} i.e. the set $\big\{x\in \tilde{X}(\C)~|~w(x)\in X(\R)\big\}$ equals $\tilde{X}(\R)$.
	\end{enumerate}
\end{intro6}

This proposition allows to easily describe the real loci of properly wound \te s. The unwinding of such an \te\;“\,trivialises\,” its canonical fibration. In particular, it provides the following homeomorphism of the real locus:
\begin{equation*}\tag{\ref{eq:joint_map_tor}}
	X(\R)\approx F(\R)\times^\Gamma (\Sph^1)^q.
\end{equation*}
It is a twisted product, which means that we perform the quotient by a diagonal action. The action on the second factor is free. We finish this section by introducing \emph{resolutions of winding}. Given an \te\;$T\hookrightarrow X$, we define such a resolution of its winding as a properly wound \te\;$T\hookrightarrow X'$ together with a surjective morphism of \te s $X'\rightarrow X$. We note that, since the acting tori are the same, such a morphism is necessarily birational. We show that such a resolution always exists. This allows to apply Theorem~\ref{thm:can_fib} to every \te. It implies the following proposition:

\begin{intro7}
	The real locus of a real \te\;that has compact real locus is path connected.
\end{intro7}

When $X$ is smooth but improperly wound, its unwinding will have quotient singularities. In this case, resolving the winding of $X$ amounts to resolve the singularities of the unwinding in advance. In particular, when $X$ is smooth, there is a well defined closed subscheme $W$ of codimension $2$ whose blow-up always resolve the winding of $X$.

\begin{intro8}
	Let $T\hookrightarrow X$ be a real \te. The variety $\Bl_W X\rightarrow X$ is a resolution of the winding of $X$. Moreover, $\Bl_W X\rightarrow X$ restricts to an isomorphism of the canonical fibres.
\end{intro8}

\paragraph{Cycles and Cohomology} In this section, we start by computing the \emph{virtual Betti numbers} of every real \te. These numbers, introduced in \cite{McCroryClint2003VBno}, coincide with the Betti numbers of the real locus whenever the variety is smooth and have compact real locus. We express them using a bivariate polynomial, $e[X]$. This polynomial counts the number of toric subvarieties of $X$ whose torus is in a given isogeny class.

\begin{intro9}
	Let $T\hookrightarrow X$ be a real torus embedding. The virtual Poincaré polynomial of $X$ is given by the following formula:
	\begin{equation*}
		\beta[X]=e[X](t-1;t+1).
	\end{equation*}
	Hence, whenever the topological core of $X$ is smooth and have compact real locus, the Poincaré polynomial of $X(\R)$ is given by $b[X(\R)]=e[X](t-1;t+1)$.
\end{intro9}

This proposition implies that the only spheres occurring as real loci of toric varieties are $\Sph^1$ and $\Sph^2$. This computation shows that the Leray-Serre spectral sequences of the canonical fibrations of properly wound smooth \te s with compact real loci degenerate at the first page. Further, we show that the cohomology of every smooth real \te\;with compact real locus is totally algebraic:

\begin{intro10}
	Let $T\hookrightarrow X$ be a smooth real \te, if its real locus is compact then its cohomology is totally algebraic.
\end{intro10}

When one has a projective equivariant embedding of a split torus, this theorem is a simple consequence of an algebraic cellular decomposition defined by a shelling of its fan, \textit{cf.} \cite[\S10]{Danilov:1978aa} or \cite[\S5.2]{Fulton:1993aa}. The surjectivity can be extended to the complete case by  the techniques of V. Danilov. However, even when the variety is projective, if the torus is not split, a shelling of the fan does not define an algebraic cellular decomposition in general. We should note that, contrary to the split case, not every cohomology class is necessarily dual to a toric cycle, \textit{cf.} Proposition~\ref{prop:non_tor_classes}. The presentation of the subgroup of the first cohomology group spanned by classes of toric divisors enables us to derive the following orientability criterion:

\begin{intro11}
	Let $T\hookrightarrow X$ be a real \te with smooth topological core and compact real locus. Its real locus is orientable if and only if there exists a linear map:
	\begin{equation*}
		j:\ker(1-\tau)\otimes\F_2 \rightarrow \F_2,
	\end{equation*}
	that vanishes on $\Gamma$ and whose value is one on every primitive generator of the invariant rays of $X$.
\end{intro11}

It generalises \cite[Theorem~3.2]{Soprunova:2013aa}. 
\paragraph{Topological Types in Low Dimension.} In this last section, we begin by reformulating the results of \cite{delau2004} about the topological types of real toric curves and surfaces in our formalism. Then, we determine the prime decomposition of every smooth real equivariant embedding of non-split tori with compact real locus. We can remark that this decomposition was provided in \cite[Theorem~3.12]{ERO22} for smooth real equivariant embeddings of split tori with compact and orientable real loci\footnote{They even provide the JSJ-decomposition of such threefolds.}. Even if $\Sph^3$ never occurs as the real locus of a toric threefold, every lens space with fundamental group of even order can be constructed:

\begin{intro12}
	Let $T\hookrightarrow X$ be a real \te\;of type\footnote{Given the classification of real tori, every torus is isomorphic to a unique product $\G{R}^{p-r}\times_\R\SO^{q-r}\times\Res\G{\C}^r$. We say that its type is $(p;q)_r$, \textit{cf.} Corollary~\ref{cor:real_loci_real_torus}.} $(1;2)_1$ that has compact real locus and smooth topological core. The real locus of $X$ is homeomorphic to either $\Proj^2(\R)\times\Sph^1$, $(2\cdot\Proj^2(\R))\times\Sph^1$, or a lens space $L(2p;q)$ with $2p$ and $q$ coprime. All these threefolds occur as the real locus of such a variety.
\end{intro12}

Further, we refine the polynomial $e[X]$ into a trivariate polynomial $e^*[X]$ that counts the toric orbits every isomorphism type. We show that it defines an almost complete homeomorphism invariant of smooth and compact real \te s of type $(2;1)_1$.

\begin{intro13}
	Let $T\hookrightarrow X,Y$ be two real \te s of type $(2;1)_1$ with compact real loci and smooth topological cores. If $e^*[X]=e^*[Y]$, then $X(\R)$ is homeomorphic to $Y(\R)$. If $X(\R)$ is homeomorphic to $Y(\R)$, then $e^*[X]=e^*[Y]$ except when their real loci are homeomorphic to $\Proj^2(\R)\times\Sph^1$, in which case, their $e^*$-polynomials can either be $xz+2z+xy+3y$ or $xz+2z+xy+x+2y+2$.
\end{intro13}

\section*{General Notations and Conventions}

\paragraph{Group Cohomology.} Throughout this text, we will consider abelian groups endowed with involutions i.e. module over the group algebra $\Z[\Z/2]$. We will denote the latter algebra by $\Z[\tau]$ where $\tau^2=1$. Moreover, we will denote by $\Z[1]$, respectively $\Z[-1]$, the module $\Z$ over $\Z[\tau]$ on which $\tau$ acts as the multiplication by $1$, respectively $-1$. When $N$ is a module over $\Z[\tau]$, the cohomology of the group $\Z/2$ with coefficients in $N$ will always be assumed to be computed with the Quillen resolution of $\Z[1]$:
\begin{equation*}
	\begin{tikzcd}
	0\ar[r] & \Z[1] \ar[r] & \Z[\tau] \ar[r,"1-\tau"] & \Z[\tau] \ar[r,"1+\tau"] & \Z[\tau] \ar[r,"1-\tau"] & \cdots
	\end{tikzcd}
\end{equation*}
Hence, $(H^k(\Z/2;N) )_{k\geq 0}$ is the cohomology of the following complex:
\begin{equation*}
	\begin{tikzcd}
	 N \ar[r,"1-\tau"] & N \ar[r,"1+\tau"] & N \ar[r,"1-\tau"] & \cdots.
	\end{tikzcd}
\end{equation*}

\paragraph{Monoid Algebra.} When $R$ is a commutative ring and $M$ is a commutative monoid, we denote the associated algebra by $R[M]$. For all $m\in M$, the symbol $\x^m$ denotes the corresponding element in $R[M]$ (in the previous notations $\tau=\x^1$). Moreover, if $x$ is a ring morphism from $R[M]$ to $S$ we denote by $x^m\in S$ the value of $x$ at $\x^m$.

\paragraph{Homeomorphism.} We denote homeomorphisms by the symbol $\approx$.

\paragraph{Varieties.} Let $k$ be a field. A variety $X$ over $k$ means a separated integral scheme of finite type over $k$. When $k$ is the field of real or complex numbers, the set of $k$-points of $X$ is always endowed with its Euclidean topology. Out of simplicity, we may define morphisms of schemes using formulæ involving fake variables. For instance, let $G$ be a $k$-group acting on a $k$-scheme $X$ via $\alpha:G\times_kX\rightarrow X$, $x_0$ be a $k$-point of $X$, and $f:Y\rightarrow G$ be a morphism of $k$-schemes. The expression $y \mapsto f(y)\cdot x_0$ denotes the following the morphism: 
\begin{equation*}
	\alpha\circ (f; x_0\circ s) : Y \longrightarrow G\times_k X \longrightarrow X,
\end{equation*}
where $s:Y\rightarrow \Spec\,k$ is the structure morphism of $Y$.

\paragraph{Cycle Class Map.}
	Let $X$ be a real variety. We denote the cycle class map of Borel-Haefliger by:
	\begin{equation*}
		\cl^X:\Ch_k(X)\rightarrow H_k^\textnormal{BM}(X(\R);\F_2),
	\end{equation*}
	for all non-negative integers $k$, \textit{cf.} \cite[\S5.12]{Borel:1961aa}. It commutes with proper push-forward, \textit{cf.} \cite[Lemma~19.1.2]{Fulton:1998aa} in the complex case. Moreover, when $X(\R)$ is non-empty and has a smooth open neighbourhood in $X$, we denote by $[X(\R)]$ its $\F_2$-fundamental class. It allows to define the morphism:
	\begin{equation*}
		\cl_X:\Ch^k(X)\rightarrow H^k(X(\R);\F_2),
	\end{equation*}
	that satisfies $\cl_X(Z)\cap[X(\R)]=\cl^X(Z)$, for all $k$-codimensional classes $Z$. It commutes with pull-backs, \textit{cf.}  \cite[Corollary~19.2]{Fulton:1998aa} adapted to real varieties, and proper push-forward. We recall that the cohomological push-forward between smooth manifolds is defined via the Poincaré duality. If $f:M\rightarrow N$ is a proper map, the push-forward is denoted by $f_!$. 

\paragraph{Affine Geometry.} Let $N$ be a free abelian group of finite rank.\begin{enumerate}
	\item[(i)] A subgroup $N'$ of $N$ is said to be \emph{primitive} if the quotient $N/N'$ is torsion free. Likewise, a vector $v\in N$ is said to be \emph{primitive} when $\Z v$ is a primitive subgroup of $N$;
	\item[(ii)] a \emph{polyhedral cone} (or simply a \emph{cone}) $c$ of $N\otimes \R$ is a subset of the following form:
	\begin{equation*}
		c=\{v\in N\otimes \R\;|\; \alpha_i(v)\geq 0 \textnormal{ for all }1\leq i\leq k\},
	\end{equation*}
	 where $\alpha_1,...,\alpha_k$ are linear forms. A \emph{face} of $c$ is a cone of the form $c\cap\ker(\beta)$ where $\beta$ is a linear form that is non-negative over $c$. The cone $c$ is said to be \emph{strongly convex} when the origin is the only linear subspace it contains. If the forms $\alpha_1,...,\alpha_k$ can be taken integral then $c$ is said to be \emph{rational};
	\item[(iii)] A \emph{fan} $C$ is a finite collection of cones that contains all the faces of its cones, and in which the intersection of two cones is a common face of both of them;
	\item[(iv)] A $k$-dimensional cone $c$ is said to be \emph{simplicial} if it consists of non-negative linear combinations of $k$ independent vectors. If the vectors can further be taken as part of a basis of the lattice $N$ then $c$ is said to be \emph{smooth}. By extension a fan is said to be \emph{simplicial} (resp. \emph{smooth}) when it is entirely made of simplicial (resp. smooth) cones;
	\item[(v)] The \emph{support} of a fan $C$ is the set formed by the union of its cones. If the support of $C$ covers $N\otimes\R$ then we say that $C$ is \emph{complete};
	\item[(vi)] A pair $(N;C)$ where $N$ is a a free abelian group and $C$ is a fan of strongly convex rational polyhedral cones of $N\otimes \R$ will be called an \emph{\dl};
	\item[(vii)] A morphism between two such objects $f:(N_1;C_1)\rightarrow (N_2;C_2)$ is a morphism $f:N_1\rightarrow N_2$ such that for all cones $c_1\in C_1$ there is a cone $c_2\in C_2$ that contains $f(c_1)$.
	\item[(viii)] If $c$ is a cone of $N\otimes \R$ and $M$ denotes $\Hom(N;\Z)$, then $c^+$ is defined to be the cone of non-negative forms $\{\alpha\in M\otimes \R\;|\; \alpha(v)\geq 0,\,\forall v\in c\}$, and $c^\perp$ the vector subspace $\{\alpha\in M\otimes \R\;|\; \alpha(v)= 0,\,\forall v\in c\}$.
	\end{enumerate}

 We will use Fulton's notations. In particular, if $c$ is a rational polyhedral cone of $N\otimes\R$, where $N$ is a free abelian group, $N_c$ denotes the group of lattice points contained in the subspace spanned by $c$, and $N(c)$ denotes the quotient $N/N_c$. If $M$ is the dual of $N$, then $M(c)$ denotes $c^\perp \cap M$, and $M_c$ the quotient $M/M(c)$. 

\section{Toric Real Structures}

\subsection{Real Tori}

Let $T$ be a complex torus of dimension $n$. We denote its cocharacter lattice by $N$. It is the group of morphisms of algebraic groups from the complex multiplicative group to $T$. It is a free abelian group of rank $n$. The character lattice of $T$, denoted by $M$, is the group of morphisms of algebraic groups from $T$ to the complex multiplicative group. It is in natural duality with $N$. The coordinate ring of $T$ is naturally isomorphic to the group algebra $\C[M]$ which is a ring of Laurent polynomials in $n$ indeterminates. The group of complex points of $T$ is naturally isomorphic to $N\otimes\C^\times$ and $\Hom(M;\C^\times)$.

\begin{dfn}
	A \emph{real algebraic torus} is an algebraic group $T$ defined over $\R$ whose complexification $T_\C$ is isomorphic to a product of complex multiplicative groups. The real torus $T$ is said to be \emph{split} when it is isomorphic to a product of real multiplicative groups.
\end{dfn}

\begin{dfn}
	Let $T$ be a complex torus. A \emph{torus real structure} of $T$ is an anti-regular involutive morphism of complex groups $\tau:T\rightarrow T$.
\end{dfn}

Since tori are affine, it is equivalent to specify a real torus or a complex torus endowed with a torus real structure, \textit{cf.} \cite[\S2.12]{Borel:1964aa}. Let $T$ be a real torus. The cocharacter lattice $N$ of its complexification is endowed with an involution induced by the torus real structure $\tau\coloneqq \id\times\textnormal{conj}:T_\C\rightarrow T_\C$. We will also denote it by $\tau$. If $\tau_{\mathbb{G}}$ stands for the canonical real structure of $\G{\C}$, then the involution of $N$ is given by the following formula:
\begin{equation}\label{eq:invol_N}
	\tau v \coloneqq  \tau\circ v \circ \tau_\mathbb{G},\;\forall v\in N=\Hom(\G{\C};T_\C).
\end{equation}
We denote by $\tau^*$ the adjoint involution of the character lattice $M$.

\begin{dfn}[Character and Cocharacter Lattices]
	Let $T$ be a real torus. The \emph{cocharacter lattice} of $T$ is the $\Z[\tau]$-module formed by the cocharacter lattice of $T_\C$ endowed with the involution given by Formula (\ref{eq:invol_N}). Likewise, its \emph{character lattice} is the $\Z[\tau]$-module formed by the character lattice of $T_\C$ endowed with the adjoint involution.
\end{dfn}

A torus real structure $\tau$ on a complex torus $T$ is fully determined by its action on the character and cocharacter lattices of $T$. In particular, if $t$ belongs to $T(\C)$ and $\alpha$ is a character of $T$, then $\tau(t)^\alpha$ is the complex conjugate of $t^{\tau^*\alpha}$. Accordingly, two torus real structures are isomorphic if and only if the corresponding involutions of $N$ are similar. The functor that sends a real torus to its cocharacter lattice is fully faithful, cf \cite[\S8.12, Proposition]{Borel:1969}. 

\begin{exs}
	The first and obvious example of real tori is the real multiplicative group $\G{\R}$. Its cocharacter lattice is $\Z[1]$. The only other real torus of dimension one is the group of planar rotations $\SO$ which is isomorphic to $\Spec\,\R[x;y]/(x^2+y^2-1)$. Its cocharacter lattice is $\Z[-1]$. A third example would be the Weil restriction of the complex multiplicative group $\Res\G{\C}$ whose coordinate ring is $\R[x;y;1/(x^2+y^2)]$, and whose cocharacter lattice is $\Z[\tau]$.
\end{exs}

\begin{prop}[Theorem~2 in \cite{cass08}]\label{prop:str_Z_tau_modules}
	Let $N$ be a lattice endowed with an involution $\tau$. It splits, as a module over $\Z[\tau]$, into a direct sum of the three factors $\Z[1]$, $\Z[-1]$, and $\Z[\tau]$.
\end{prop}

The proposition directly implies the following corollary.

\begin{cor}[Theorem~2 in \cite{cass08}]\label{cor:real_loci_real_torus}
	Let $T$ be a $n$-dimensional real torus. There exists three non-negative integers $u,v,w$ satisfying $u+v+2w=n$, and such that:
	\begin{equation*}
		T\cong \G{\R}^u\times_\R \SO^v\times_\R \Res\G{\C}^w. 
	\end{equation*}
	In this case, $T(\R)$ is isomorphic to $(\R^\times)^u\times(\Sph^1)^v\times(\C^\times)^w$.
\end{cor}

\begin{dfn}\label{dfn:torus_invariants}
	Let $T$ be a real torus with cocharacter lattice $(N;\tau)$. \begin{itemize}
		\item[(i)] The \emph{isogeneous type} of $T$ is the couple of integers $(p;q)$ where $p$ denotes the rank of $\ker(1-\tau)$ and $q$ denotes the rank of $\ker(1+\tau)$;
		\item[(ii)] The \emph{winding group} $\Gamma$ of $T$ is the quotient of $N$ by the sum $\tilde{N}\coloneqq \ker(1-\tau)\oplus \ker(1+\tau)$. This sum contains $2N$, thus $\Gamma$ is of $2$-torsion. The \emph{winding number} $r$ of $T$ is $\dim_{\F_2}\Gamma$;
		\item[(iii)] The \emph{type} $(p;q)_r$ of $T$ is the unique triplet of integers such that $N$ is isomorphic to the direct sum $\Z[1]^{p-r}\oplus\Z[-1]^{q-r}\oplus\Z[\tau]^r$.
	\end{itemize}
	The type is the combination of the isogeneous type and of the winding number. Two tori have the same isogeneous type if and only if they are isogeneous. A real torus is \emph{unwound} when its winding number is zero.
\end{dfn}

\paragraph{Closed Subgroups.} Following  \cite[\S8.12]{Borel:1969}, the category of \emph{real diagonalisable groups}\footnote{A real affine group $G$ is \emph{diagonalisable} if $\Hom(G_\C;\G{\C})$ spans $\mathcal{O}(G_\C)$. In this case, $\mathcal{O}(G_\C)$ is isomorphic to $\C[\Hom(G_\C;\G{\C})]$. This is the definition of diagonalisability given by A. Borel. \cite[Definition~1.1]{Grothendieck:1970aa}, which is used in \cite{Sumihiro:1975aa}, is more restrictive. It requires $\mathcal{O}(G)$ to be isomorphic to $\R[\Hom(G;\G{\R})]$.} is equivalent to the category of finitely generated $\Z[\tau]$-modules. The equivalence sends a diagonalisable group $G$ to its character group $\Hom(G_\C;\G{\C})$ endowed with an involution defined by a formula similar to (\ref{eq:invol_N}). Since equivalences of Abelian categories are always additive and exact, \textit{cf.} \cite[Proposition 16.2.4]{Schubert:1972aa}, the closed subgroups of a real torus $T$ correspond to the quotients, as $\Z[\tau]$-modules, of its character lattice $M$. Thus, if $G$ is a closed subgroup of $T$ corresponding to a quotient $M\rightarrow Q$, the quotient $T/G$ is a real torus for the kernel of $M\rightarrow Q$ is torsion free.

\paragraph{2-Torsion.} Let $T$ be a real torus and $M$ be its character lattice. Its 2-torsion has character lattice $M/2M$. Let $N$ be its cocharacter lattice. We have a natural isomorphism of $\Z[\tau]$-modules:
\begin{equation}\label{eq:2tors}
	\begin{array}{rcl}
		N/2N & \longrightarrow & T[2](\C)=\Hom_\Z(M/2M;\C^\times) \\ v & \longmapsto & [\alpha \mapsto (-1)^{\alpha(v)}].
	\end{array}
\end{equation}
Therefore, $T[2](\R)$ is naturally isomorphic to $H^0(\Z/2;N/2N)$.

\paragraph{Fundamental Group.} The cocharacter functor $T\mapsto \Hom(\G{\C};T_\C)$ is naturally isomorphic to $T\mapsto \pi_1(T(\C);1)$. The isomorphism sends $v:\G{\C}\rightarrow T_\C$ to $v_*[\Sph^1]$, where $[\Sph^1]$ is the unit circle of $\C^\times$ endowed with its trigonometric parametrisation. The real structure endows $\pi_1(T(\C);1)$ with an involution for which the natural isomorphism is anti-equivariant. Hence, we have a natural isomorphism between $\ker(1+\tau)$ and the subgroup of invariant classes of loops. One can check easily that, for real tori, the subgroup of invariant classes of loops is naturally isomorphic to $\pi_1(T(\R);1)$. Hence, if $G$ denotes the identity  component of $T(\R)$, this observation yields a natural surjection:
\begin{equation}\label{eq:nat_h}
	h:H_1(G;\Z/2)\rightarrow H_1(\Z/2;N).
\end{equation}

\paragraph{Group Cohomology.} The exponential exact sequence allows for an easy computation of the group cohomology of real tori.

\begin{lem}\label{lem:group_cohomo_tori}
	Let $T$ be a real torus with cocharacter lattice $N$. For all integers $k\geq1$, there is a natural isomorphism $H^k(\Z/2;T(\C))\rightarrow H^k(\Z/2;N)$.
\end{lem}

\begin{proof}
	Let us consider the exponential exact sequence of $\Z[\tau]$-modules:
	\begin{equation}\label{eq:exponential_seq}
		0 \rightarrow \Z[-1] \overset{2i\pi\,}{\longrightarrow} \C \overset{\exp\,}{\longrightarrow} \C^\times \rightarrow 0.
	\end{equation}
	Since $N$ is a free Abelian group, (\ref{eq:exponential_seq}) remains exact after tensorisation by $N$. Every group of the tensorisation have a natural structure of $\Z[\tau]$-module where $\tau$ acts as: $\tau\cdot a\otimes b\coloneqq \tau(a)\otimes\tau(b)$. The module $N\otimes_\Z\C$ is acyclic, hence we have an isomorphism:
	\begin{equation*}
		H^k\big(\Z/2;T(\C)\big)\longrightarrow H^{k+1}\big(\Z/2;N\otimes_\Z\Z[-1]\big),
	\end{equation*} 
	for every integer $k\geq 1$. The lemma follows from the canonical isomorphism between $H^k(\Z/2;N)$ and $H^{k+1}(\Z/2;N\otimes_\Z\Z[-1])$ obtained by cup product with the generator of $H^1(\Z/2;\Z[-1])$.
\end{proof}

Following Lemma~\ref{lem:group_cohomo_tori}, we may identify the cohomologies of the complex points of a real torus and of its cocharacter lattice. When we do so, it will always be through this natural isomorphism. It might also be worth noting that, given a real torus with cocharacter lattice $N$ and character lattice $M$, the duality pairing $M\otimes N\rightarrow \Z[1]$ is $\Z[\tau]$-linear. It induces, together with the cup product, a natural duality:
\begin{equation}\label{eq:nat_duality}
	H^1(\Z/2;M)\otimes H^1(\Z/2;N) \rightarrow H^2(\Z/2;\Z[1])=\Z/2.
\end{equation}

\paragraph{Canonical Isogeny.} Every real torus is canonically isogeneous to an unwound torus. Following Definition~\ref{dfn:torus_invariants}, we have an exact sequence of $\Z[\tau]$-modules:
\begin{equation}\label{eq:torus_can_isog_coch}
	0\rightarrow \tilde{N} \overset{w}{\longrightarrow} N \longrightarrow \Gamma \rightarrow 0.
\end{equation}
The induced involution of $\Gamma$ is trivial. Let us define the following real diagonalisable group:
\begin{equation}
	\Gamma_\R\coloneqq \Spec\,\R[\Ext_\Z(\Gamma;\Z)].
\end{equation}
We note that the functor $\Gamma\mapsto\Ext_\Z(\Gamma;\Z)$ is naturally isomorphic to $\Gamma\mapsto\Hom_\Z(\Gamma;\F_2)$ over the subcategory of $\F_2$-vector spaces. In addition, $\Gamma_\R(\R)$ is naturally isomorphic to $\Gamma$. We also note that $\Gamma_\R$ is the real constant group associated to $\Gamma$. Hence, according to the equivalence of categories between real diagonalisable groups and finitely generated $\Z[\tau]$-module, (\ref{eq:torus_can_isog_coch}) yields an exact sequence of real groups:
\begin{equation}\label{eq:torus_can_isog}
	1\rightarrow \Gamma_\R \longrightarrow \tilde{T} \overset{w}{\longrightarrow} T \rightarrow 1.
\end{equation}
The torus $\tilde{T}$ is unwound by construction, and (\ref{eq:torus_can_isog}) is an isogeny for its kernel is finite. The sequence (\ref{eq:torus_can_isog_coch}) can be used to compute the cohomology of $N$. It is not absolutely necessary but the computation provides useful morphisms. Since $\Gamma$ is purely of 2-torsion and its involution is trivial, all its cohomology groups are isomorphic to itself. Moreover, the morphisms induced by $w$ in group cohomology are surjective. This is a simple consequence of Lemma~\ref{lem:group_cohomo_tori}. Hence, we find these two short exact sequences: 
\begin{equation}\label{eq:emb_gamma}
	\left\{\begin{array}{l} 0\rightarrow \Gamma \overset{\df_0}{\longrightarrow} \ker(1+\tau)\otimes\F_2\longrightarrow H^1(\Z/2;N) \rightarrow 0 \\
	0\rightarrow \Gamma \overset{\df_1}{\longrightarrow} \ker(1-\tau)\otimes\F_2\longrightarrow H^2(\Z/2;N) \rightarrow 0,
	\end{array}\right.
\end{equation}
where the inclusions are given by the connecting morphisms. Lastly, we note that, since $\tilde{T}$ is unwound, the 2-torsion of its complex locus is real. Thus, we have three natural embeddings of $\Gamma$ in $\tilde{N}/2\tilde{N}$: via inclusion of 2-torsions of real loci, via the tensorisation of (\ref{eq:torus_can_isog_coch}) by $\Z/2$, and via the diagonal map $(\df_0;\df_1)$ given by (\ref{eq:emb_gamma}). One can show that they are the same inclusion.

\paragraph{Exact Sequences of Real Tori.} For technical purposes, we want to decompose exact sequences of real tori into elementary pieces. This leads us to a characterisation of \emph{locally split exact sequences}, i.e. exact sequences that defines principal bundles. To do so, it will useful to understand extensions of $\Z[\tau]$-modules whose underlying Abelian groups are free and finitely generated. Let $M$ and $N$ be two such modules. The Abelian group $\Ext^1_{\Z[\tau]}(N;M)$ parametrises equivalence classes of extensions of $N$ by $M$. The Grothendieck Spectral Sequence, \textit{cf.} \cite[Théorème~2.4.1]{tohoku}, yields an isomorphism:
\begin{equation}\label{eq:ext_Z(tau)0}
	\Ext^1_{\Z[\tau]}(N;M) \longrightarrow H^1\big(\Z/2;\Hom_\Z(N;M)\big),
\end{equation}
where $\tau$ acts as $\tau\cdot f\coloneqq \tau_M \circ f\circ \tau_N$ on $f\in \Hom_\Z(N;M)$. There is a more down to earth way to understand this isomorphism. An extension of $N$ by $M$ can always be given by an involution of $N\oplus M$ that respects the extension. Hence, it has the following shape:
\begin{equation*}
	\left( \begin{array}{cc} \tau_N & 0 \\ d & \tau_M \end{array}\right) \in \left( \begin{array}{cc} \End_\Z(N) & \Hom_\Z(M;N) \\ \Hom_\Z(N;M) & \End_\Z(M) \end{array}\right).
\end{equation*}
The involution condition is equivalent to the requirement that $d$ is anti-equivariant, i.e. satisfies $\tau\cdot d=-d$. Two such extensions are isomorphic if and only if their matrices are conjugated by a matrix of the form:
\begin{equation*}
	\left( \begin{array}{cc} \id_N & 0 \\ m & \id_M \end{array}\right). 
\end{equation*}
This is equivalent to their $d$\,\textsuperscript{th} coordinate differing by an element of the form $\tau\cdot m-m$. The isomorphism (\ref{eq:ext_Z(tau)0}) is now obvious. Given an extension $0\rightarrow M\rightarrow E\rightarrow N\rightarrow 0$, a practical way to compute its equivalence class is to consider a $\Z$-linear section $s:N\rightarrow E$ of the projection. The morphism $\tau_E \circ s-s\circ\tau_N$ takes its values in $M$ and is anti-equivariant. Its cohomology class represents the equivalence class of $E$. To go a little further, we can note that the extension $E$ is fully characterised by its associated cohomological long exact sequence. The classes in $H^1(\Z/2;\Hom_\Z(N;M))$ are represented by anti-equivariant morphisms. They induce two morphisms $H^1(\Z/2;N)\rightarrow H^2(\Z/2;M)$ and $H^2(\Z/2;N)\rightarrow H^3(\Z/2;M)$. Given $E$, these two morphisms are, by construction, the two connecting morphisms of the cohomological long exact sequence. If one decomposes $N$ and $M$ into sums of $\Z[1]$, $\Z[-1]$, and $\Z[\tau]$, one finds that the morphism:
\begin{equation}\label{eq:ext_Z(tau)}
	\begin{array}{rcl}
		H^1\big(\Z/2;\Hom_\Z(N;M)\big) & \longrightarrow & \Hom_\Z\big( H^1(N);H^2(M) \big)\times \Hom_\Z\big( H^2(N);H^3(M) \big) \\
		~[E] & \longmapsto & (\df_1;\df_2),
	\end{array}
\end{equation} 
is an isomorphism.

\begin{lem}\label{lem:exact_seq_real_tori}
	A short exact sequence of real tori is a product of a split short exact sequence, and of some copies of the following non-split short exact sequences:
	\begin{equation}\label{eq:non-split-exact-sequeneces}
	\left\{\begin{aligned}
		1 &\rightarrow \G{\R} \rightarrow \Res\G{\C} \rightarrow \SO \rightarrow 1 \\
		1 &\rightarrow \SO \rightarrow \Res\G{\C} \rightarrow \G{\R} \rightarrow 1\,.
	\end{aligned}\right. 
	\end{equation} 
\end{lem}

\begin{proof}
	The lemma is equivalent to the following statement: \emph{Every short exact sequence of modules over $\Z[\tau]$, whose underlying Abelian groups are free and finitely generated, is a direct sum of a split exact sequence, and of copies of the following non-split short exact sequences:}
	\begin{equation}\label{eq:elementary_non_split_ext}
	\left\{\begin{aligned}
		0 &\rightarrow \Z[1] \longrightarrow \Z[\tau] \rightarrow \Z[-1] \rightarrow 0 \\
		0 &\rightarrow \Z[-1] \rightarrow \Z[\tau] \longrightarrow \Z[1] \rightarrow 0\,.
	\end{aligned}\right. 
	\end{equation} 
	Let us consider such an exact sequence of $\Z[\tau]$-modules: $0\rightarrow M \rightarrow E \rightarrow N \rightarrow 0$. We denote its connecting morphisms in the cohomological long exact sequence by $\df_1$ and $\df_2$. Let us now use two isomorphisms $u_N:N\cong \Z[1]^{k_N}\oplus \Z[-1]^{l_N}\oplus \Z[\tau]^{m_N}$ and $u_M:M\cong \Z[1]^{k_M}\oplus \Z[-1]^{l_M}\oplus \Z[\tau]^{m_M}$. For all $P\in\{M;N\}$, our isomorphisms yield:
	\begin{equation*}
		u^1_P:H^1(\Z/2;P)\cong \big(\Z/2\big)^{l_P},\; u^2_P:H^2(\Z/2;P)\cong \big(\Z/2\big)^{k_P},\;\textnormal{and } u_P^3:H^3(\Z/2;P)\cong \big(\Z/2\big)^{l_P}.
	\end{equation*}
	The reduction modulo 2 morphism $\GL_n(\Z)\rightarrow\GL_n(\F_2)$ is surjective\footnote{n.b. $\GL_n(\F_2)$ equals $\textnormal{SL}_n(\F_2)$ and thus is generated by transvections. They can be lifted.} for all non-negative integers $n$. Thus, we can find, for all $P\in\{M;N\}$, $\phi_P\in\GL_{l_P}(\Z)$ and $\psi_P\in\GL_{k_P}(\Z)$ such that:
	\begin{equation*}
		\psi_M\,u^2_N\,\df_1(u^1_N)^{-1}\phi^{-1}_N\textnormal{ and } \phi_M\,u^3_N\,\df_1(u^2_N)^{-1}\psi^{-1}_N,
	\end{equation*}
	are diagonal matrices. We can note that:
	\begin{equation*}
	\left( \begin{array}{ccc} \psi_P & 0 & 0 \\ 0 & \phi_P & 0 \\ 0 & 0 & I_{2m_P} \end{array}\right),
\end{equation*}
	is a $\Z[\tau]$-linear automorphism of $ \Z[1]^{k_P}\oplus \Z[-1]^{l_P}\oplus \Z[\tau]^{m_P}$ for all $P\in\{M;N\}$. Therefore, we can assume that the matrices of $\df_1$ and $\df_2$ are diagonal in the bases given by $u_M$ and $u_N$. These bases allow to see $E$ as an extension of $\Z[1]^{k_N}\oplus \Z[-1]^{l_N}\oplus \Z[\tau]^{m_N}$ by $\Z[1]^{k_M}\oplus \Z[-1]^{l_M}\oplus \Z[\tau]^{m_M}$. By diagonality and the isomorphisms (\ref{eq:ext_Z(tau)0}) and (\ref{eq:ext_Z(tau)}), we can construct an isomorphic extension by summing split exact sequences and the two elementary non-split exact sequences of (\ref{eq:elementary_non_split_ext}). There will be $\rk \df_1$ summands of the first kind, and $\rk\df_2$ summands of the second kind.
\end{proof}

\begin{prop}\label{prop:local_section_tori}
	A short exact sequence of real tori $1\rightarrow T_1 \rightarrow T_2 \overset{\pi}{\rightarrow} T_3\rightarrow 1$ defines an algebraic fibre bundle if and only if $\pi:T_2(\R)\rightarrow T_3(\R)$ is surjective. In this case, we can find an open subscheme $U\subset T_3$ that admits a section $s:U\rightarrow T_2$ of $\pi$, and such that the family $(\pi(t)\cdot U )_{t\in T_2(\R)}$ is an open cover of $T_3$.
\end{prop}

\begin{proof}
	If $1\rightarrow T_1 \rightarrow T_2 \rightarrow T_3\rightarrow 1$ defines an algebraic fibre bundle, $\pi:T_2(\R)\rightarrow T_3(\R)$ is surjective for $T_1(\R)$, the real locus of the fibre, cannot be empty. Conversely, let us assume that $\pi:T_2(\R)\rightarrow T_3(\R)$ is surjective. Hence, the summand $1\rightarrow \SO \rightarrow \Res\G{\C} \rightarrow \G{\R} \rightarrow 1$ does not appear in the decomposition of $1\rightarrow T_1 \rightarrow T_2 \rightarrow T_3\rightarrow 1$ given by Lemma~\ref{lem:exact_seq_real_tori}. Indeed, for this short exact sequeunce, the projection of $\C^\times$ is the group of positive real numbers. To prove the statement it will be enough to exhibit a section $\pi:\Res\G{\C}\rightarrow \SO$ over an open subscheme $U$ whose translates by $t\in\SO(\R)$ cover $\SO$. The morphism $\pi$ is given by the following morphism of $\R$-algebras:
	\begin{equation*}
		\begin{array}{rcl}
			\pi^*: \bigslant{\R[u;v]}{(u^2+v^2-1)} & \longrightarrow & \R\big[x;y;\tfrac{1}{x^2+y^2}\big]\\[5pt]
			u & \longmapsto & \tfrac{x^2-y^2}{x^2+y^2} \\[5pt]
			v & \longmapsto & \tfrac{2xy}{x^2+y^2}.
		\end{array}
	\end{equation*}
	Let us consider the open subscheme $U=\{u\neq-1\}$ and the map $s:U\rightarrow \Res\G{\C}$ given by the following morphism of $\R$-algebras:
	\begin{equation*}
		\begin{array}{rcl}
			s^*: \R\big[x;y;\tfrac{1}{x^2+y^2}\big] & \longrightarrow & \bigslant{\R\big[u;v;\tfrac{1}{u+1}\big]}{(u^2+v^2-1)}\\[5pt]
			x & \longmapsto & \tfrac{v}{u+1} \\[5pt]
			y & \longmapsto & 1.
		\end{array}
	\end{equation*}
	A direct computation shows that $s^*\circ\pi^*$ is the localisation: \begin{equation*}
		\R[u;v]/(u^2+v^2-1)\rightarrow \R\big[u;v;\tfrac{1}{u+1}\big]/(u^2+v^2-1).
\end{equation*}
Thus, $s$ is a section of $\pi$ over $U$. Now let us consider $i=(0;1)\in\SO(\R)$. The open set $i\cdot U$ is given by $\{v\neq 1\}$, and $U\cup i\cdot U$ equals $\SO$ for the ideal $(1+u;1-v)$ is the full coordinate ring. Indeed, we have $1=(1-u)(1+u)+(1+v)(1-v)$.
\end{proof}

\paragraph{Divisors.}

We compute the first Chow group of real tori in terms of the group cohomology of their character lattices.

\begin{prop}\label{prop:chow_tori}
	Let $T$ be a real torus and $M$ be its character lattice. There is a natural isomorphism:
	\begin{equation}\label{eq:can_iso_chow}
		f:\Ch^1(T)\rightarrow H^1(\Z/2;M).
	\end{equation}
\end{prop}

\begin{proof}
	We note that $\Ch^1(T_\C)$ vanishes for $T_\C$ is an open subscheme of an affine space. Thus, we have an exact sequence of $\Z[\tau]$-modules:
	\begin{equation*}
		0 \rightarrow \C[M]^\times \rightarrow \C(M)^\times\rightarrow Z^1(T_\C) \rightarrow 0,
	\end{equation*}
	Following Hilbert's Theorem 90 \cite{hilb}, the group $H^1(\Z/2;\C(M)^\times)$ vanishes. Thus, we have the following exact sequence:
	\begin{equation}\label{seq:chow_T_R}
		0 \rightarrow \O(T)^\times \rightarrow \R(T)^\times\rightarrow Z^1(T) \rightarrow H^1(\Z/2;\C[M]^\times) \rightarrow 0.
	\end{equation}
	Therefore, (\ref{seq:chow_T_R}) provides a natural isomorphism:
	\begin{equation}\label{eq:can_iso_1}
		\Ch^1(T)\rightarrow H^1(\Z/2;\C[M]^\times).
	\end{equation}
	Using degrees and valuations, we see that $\C[M]^\times$ is the group of monomials $\C^\times\times M$. Consequently, the natural inclusion $M\rightarrow \C[M]^\times$ yields a natural isomorphism:
	\begin{equation}\label{eq:can_iso_2}
		H^1(\Z/2;M)\rightarrow H^1(\Z/2;\C[M]^\times).
	\end{equation}
	We find (\ref{eq:can_iso_chow}) by combining (\ref{eq:can_iso_1}) and (\ref{eq:can_iso_2}).
\end{proof}

\begin{prop}\label{prop:relation_of_groups}
	Let $T$ be a real torus, $M$ be its character lattice, and $G$ be the identity component of $T(\R)$. We have the following  commutative diagram:
	\begin{equation}\label{diag:naturality_alg_top}
		\begin{tikzcd}
			& H^1(\Z/2;M) \ar[rd,"h^*"] & \\
			\Ch^1(T) \ar[ru,"(\ref{eq:can_iso_chow})"]\ar[rd,"\cl_T" below ] & & H^1(G;\Z/2) \\
			& H^1\big(T(\R);\Z/2\big)\ar[ru,"\hspace{.4cm}\textnormal{rest.}" below]
		\end{tikzcd}
	\end{equation}
	where $h^*$ denotes the adjoint of the natural morphism (\ref{eq:nat_h}).
\end{prop}

\begin{proof}
	Let us first prove the commutativity for $\SO$. We note that all groups are lines over $\F_2$ and that $h^*$, $(\ref{eq:can_iso_chow})$ and $\textnormal{rest.}$ are all isomorphisms. Hence, in this case, the commutativity of the diagram only amounts to the surjectivity of $\cl_{\SO}$. If we consider the coordinates $\R[x;y]/(x^2+y^2-1)$ on $\SO$, the class of the point $P=\{x=1;y=0\}$ is the Poincaré dual of a point of the circle, i.e. the generator of $H^1(\Sph^1;\Z/2)$. Thus, $\cl_{\SO}$ is onto and the square is commutative. We can also note that its complexification is the divisor of the function $z-1$ in the usual isomorphism:
	\begin{equation*}
		\begin{array}{rcl}
			\C[z;z^{-1}] & \longrightarrow & \smallslant{\C[x;y]}{(x^2+y^2-1)} \\
			z & \longmapsto & x+iy.
		\end{array}
	\end{equation*}
	Thus, its image by (\ref{eq:can_iso_chow}) is given by the class of $(z-1)/\tau(z-1)=(z-1)/(z^{-1}-1)=-z$ which is the generator. Now, let us consider a more general torus $T$, and $\alpha$ in $M$ such that $\tau^*\alpha=-\alpha$. It defines a group morphism $\alpha:T\rightarrow \SO$. By functoriality and the commutativity for $\SO$, all outer squares of the following diagram are commutative:
	\begin{equation*}
		\begin{tikzcd}
			& & H^1(\Z/2;\Z[-1]) \ar[rrdd,"h^*"] \ar[d,"\alpha^*"] & &\\
			& & H^1(\Z/2;M) \ar[rd,"h^*"] & &\\
			\Ch^1(\SO) \ar[rruu,"(\ref{eq:can_iso_chow})"]\ar[rrdd,"\cl_{\SO}" below ]\ar[r,"\alpha^*"] & \Ch^1(T) \ar[ru,"(\ref{eq:can_iso_chow})"]\ar[rd,"\cl_T" below ] & & H^1(G;\Z/2) & H^1(\Sph^1;\Z/2) \ar[l,"\alpha^*"] \\
			& & H^1\big(T(\R);\Z/2\big)\ar[ru,"\hspace{.4cm}\textnormal{rest.}" below] & &\\
			& & H^1\big(\Sph^1;\Z/2\big)\ar[rruu,"\hspace{.4cm}\textnormal{rest.}" below]\ar[u,"\alpha^*"] & &
		\end{tikzcd}
	\end{equation*}
	With this remark, the commutativity of (\ref{diag:naturality_alg_top}) is a consequence of the isomorphisms (\ref{eq:can_iso_chow}) and the identity $[\alpha]=\alpha^*1$ in $H^1(\Z/2;M)$ (every class is obtained in such a way).
\end{proof}
The injectivity of $h^*$ implies the following corollary. 
\begin{cor}\label{cor:divisors_of_SO2}
	Let $T$ be a real torus, the cycle class map:
	\begin{equation*}
		\cl_{T}:\Ch^1(T)\rightarrow H^1\big(T(\R);\Z/2\big),
	\end{equation*}
	is injective. It is an isomorphism if and only if $T$ is a power of $\SO$.
\end{cor}

\subsection{Real Toric Varieties}

\paragraph{Toric Varieties and Equivariant Torus Embeddings.} Let $k$ be a field and $K$ be an algebraic closure of $k$.
\begin{dfn}[Toric Variety] 
	A toric variety over a field $k$ is a couple $T\curvearrowright X$ where $T$ is a torus over $k$ and $X$ is a normal, geometrically irreducible, $T$-variety such that $T_K$ has an open orbit in $X_K$ on which it acts without isotropy. A morphism of toric varieties $f:(T_1\curvearrowright X_1) \rightarrow (T_2\curvearrowright X_2)$ is the data of a group morphism $ \tor{f}:T_1\rightarrow T_2$, and of a morphism $f:X_1\rightarrow X_2$ for which the following diagram commutes:
	\begin{equation*}
		\begin{tikzcd}
			T_1\times_K X_1 \ar[d,"\textnormal{action}" left] \ar[rr,"\tor{f}\times f"] && T_2\times_K X_2 \ar[d,"\textnormal{action}" right] \\ X_1 \ar[rr,"f" below] && X_2
		\end{tikzcd}
	\end{equation*}
	If a torus $T$ is fixed, a morphism $f:(T\curvearrowright X_1) \rightarrow (T\curvearrowright X_2)$ that reduces to the identity on $T$ is just an equivariant morphism.
\end{dfn}

Every toric variety contains a unique open principal homogeneous space. Its base change to $K$ is the open orbit.

\begin{dfn}[Equivariant Torus Embedding]
	An \emph{equivariant torus embedding} $T\hookrightarrow X$ is a toric variety $T\curvearrowright X$ together with an equivariant open embedding of $T$ in $X$ (where $T$ acts on itself by translations). A morphism of \te s $f:(T_1\hookrightarrow X_1) \rightarrow (T_2\hookrightarrow X_2)$ is a morphism of the underlying toric varieties $f:(T_1\curvearrowright X_1) \rightarrow (T_2\curvearrowright X_2)$ such that the following diagram commutes:
	\begin{equation*}\begin{tikzcd}
		X_1 \ar[r,"f"] & X_2 \\
		T_1 \ar[hookrightarrow,u] \ar[r,"\tor{f}" below]& T_2 \ar[hookrightarrow,u]
	\end{tikzcd}\end{equation*}
	In the case of morphisms of \te s, we will not distinguish $f$ from $\tor{f}$. 
\end{dfn}

\begin{dfn}
	Let $T\curvearrowright X$ be a toric variety. A \emph{toric subvariety} of $T\curvearrowright X$ is a $T$-invariant closed subvariety $Y\subset X$ for which the induced action of $T$ modulo its isotropy is a toric variety. 
\end{dfn}

If $K$ is algebraically closed, one can always realise the closed immersion of a toric subvariety $Y\hookrightarrow X$ as a morphism of toric varieties by means of a section of the exact sequence:
\begin{equation*}
	0 \rightarrow \big\{\textnormal{Isotropy }Y\big\} \rightarrow T \rightarrow \;\bigslant{T}{\big\{\textnormal{Isotropy }Y\big\}} \rightarrow 0.
\end{equation*}
This is not always possible over non-closed field. In the case of $\R$, it is illustrated by (\ref{eq:non-split-exact-sequeneces}). 

\paragraph{Complex Toric Varieties and Equivariant Torus Embeddings.} We refer to the three books \cite{Kem-Knu-Mum-Sai_tor_emb}, \cite{Fulton:1993aa}, and \cite{Cox-Lit-Sch_tor_var} for a thorough treatment of complex toric varieties/toric varieties under the action of split tori. Let $T\hookrightarrow X$ be a complex \te, $N$ be the cocharacter lattice of $T$, and $M$ be its character lattice. Following \cite[Chapter 1]{Kem-Knu-Mum-Sai_tor_emb}, $T\hookrightarrow X$ defines a rational strongly convex polyhedral fan $C$ of $N\otimes \R$. The pair $(N;C)$ is an \dl\;in the terminology of the paragraph on affine geometry in the section on notations. Reciprocally, an \dl\;$(N;C)$ defines a complex \te. These two constructions are even functorial and yield an equivalence of categories. Example~\ref{ex:segre_embedding} interprets Segre's embedding in this context.
\begin{ex}[Segre's embedding]\label{ex:segre_embedding}
	The fan of $\Proj^1_\C\times_\C\Proj^1_\C$ is spanned by the four cones $\langle \pm\partial x; \pm\partial y \rangle_+$ of $\Z^2$. Its orbital lattice is represented in Figure~\ref{fig:orb_fan_P1P1}. The fan of $\Proj^3_\C$ is spanned by the four cones $\langle \partial x; \partial y; \partial z \rangle_+$, $\langle \partial x; \partial y; -\partial x-\partial y- \partial z \rangle_+$, $\langle \partial y; \partial z; -\partial x-\partial y- \partial z \rangle_+$, and $\langle \partial x; \partial z; -\partial x-\partial y- \partial z \rangle_+$ of $\Z^3$. Segre's Embedding $\Proj^1_\C\times_\C\Proj^1_\C\rightarrow \Proj^3_\C$ is a morphism of complex \te s. It is induced by the lattice morphism $\Z^2\rightarrow \Z^3$ that sends $(u;v)$ to $(v;u;u+v)$.
	\begin{figure}[!ht]
			\centering	
			\begin{tikzpicture}[scale=1]
				\draw[ thick] (-1,0) -- (1,0);
				\draw[ thick] (0,-1) -- (0,1);
				\fill[pattern={Lines[angle=-45, yshift=-.6pt , line width=.25pt]}] (0,0) -- (1,0) -- (1,1) -- (0,1);
				\fill[pattern={Lines[angle=45, yshift=-.6pt , line width=.25pt]}] (0,0) -- (-1,0) -- (-1,1) -- (0,1);
				\fill[pattern={Lines[angle=-45, yshift=-.6pt , line width=.25pt]}] (0,0) -- (-1,0) -- (-1,-1) -- (0,-1);
				\fill[pattern={Lines[angle=45, yshift=-.6pt , line width=.25pt]}] (0,0) -- (1,0) -- (1,-1) -- (0,-1);
				\fill (0,0) circle (.05);
				\fill (0.75,0) circle (.05);	
				\fill (-.75,0) circle (.05);
				\fill (0,.75) circle (.05);
				\fill (0,-.75) circle (.05);
				\fill (.75,.75) circle (.05);
				\fill (.75,-.75) circle (.05);
				\fill (-.75,.75) circle (.05);
				\fill (-.75,-.75) circle (.05);
			\end{tikzpicture}
			\caption{The orbital lattice of $\Proj^1_\C\times_\C\Proj^1_\C$.}
			\label{fig:orb_fan_P1P1}
	\end{figure}
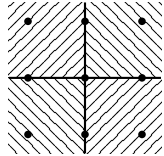
\end{ex}
Given a complex toric variety $T\curvearrowright X$, choosing an \te\;$T\hookrightarrow X$ is equivalent to the choice of a complex point in the open orbit of $X$. Two such choices yield the same fan in $N$. In this context, the fan has a more abstract flavour. It records the combinatorics of $T$-invariant affine open subvarieties of $X$, their coordinate algebras, and their embeddings into one another. There is a canonical increasing bijection:
\begin{equation}\label{eq:bij_aff_cone_C}
	(C;\leq)  \longrightarrow  \left\{ \begin{array}{c} T\textnormal{-Invariant Affine}\\ \textnormal{Open Subvarieties of }X\end{array}\right\}. 
\end{equation}
The coordinate algebra of the open subvariety $U_c$, associated to a cone $c$, is isomorphic to $\C[c^+\cap M]$. An inclusion $U_{c_1}\subset U_{c_2}$ is provided by the inclusion $\C[c_2^+\cap M]\subset \C[c_1^+\cap M]$. This bijection respects the intersection, in that $U_{c_1}\cap U_{c_2}$ equals $U_{c_1\cap c_2}$. In addition, we also have a decreasing bijection:
\begin{equation}\label{eq:bij_orb_cone_C}
	(C;\leq)  \longrightarrow  \left\{ \textnormal{Orbits of Complex Points of }X\right\}, 
\end{equation}
where an orbit $O_1$ is smaller than an orbit $O_2$ if it is strictly contained in the closure of $O_2$. The isotropy of the orbit $O(c)$, associated to the cone $c$, is the torus $\Spec\,\C[M_c]$. Hence, the coordinate ring of $O(c)$ is isomorphic to $\C[M(c)]$. See the paragraph on affine geometry at the beginning for the relevant notations. The composition of (\ref{eq:bij_orb_cone_C}) with the closure map yields the following decreasing bijection:
\begin{equation}\label{eq:bij_subvar_cone_C}
	(C;\leq)  \longrightarrow  \left\{ \textnormal{Toric Subvarieties of }X\right\}. 
\end{equation}
The torus acting on the toric subvariety $V(c)$, associated to the cone $c$, is $\Spec\,\C[M(c)]$. The fan of its cocharacter lattice $N(c)$ is the collection of the projections of the cones of $C$ that contain $c$. Two subvarieties $V(c_1)$ and $V(c_2)$ have non-empty intersection if and only if $c_1+c_2$ is a cone of $C$. In this case, the intersection is $V(c_1+c_2)$. We also want to recall that $X$ is smooth if and only if $C$ is smooth and that $X$ is complete if and only if $C$ is complete. A last useful feature is that the closed $T$-invariant subschemes:
\begin{equation}\label{eq:orb_filt}
	X_k\coloneqq \bigcup_{\substack{c\in C\\\dim c\geq \dim X-k}} V(c) \subset X,
\end{equation}
for all $0\leq k\leq \dim X$, define an increasing exhaustive closed filtration of $X$. The graded pieces of the filtration (\ref{eq:orb_filt}) are precisely the orbits of the complex points of $X$.

\paragraph{Real Structures and Real Forms.} Since complex toric varieties have been well understood we want to make use of this knowledge to describe real toric varieties. In particular, we want to be able to use the more flexible notion of \emph{real structure} on a complex toric variety to represent a real toric variety. Let us remind the definition and some elementary facts.

\begin{dfn}[Toric Real Structure]
	Let $T\curvearrowright X$ be a complex toric variety. A \emph{toric real structure} on $T\curvearrowright X$ is a torus real structure $\tau$ of $T$ and an antiregular involution $\sigma$ of $X$ such that for all $x\in X$ and all $t\in T$:
	\begin{equation*}
		\sigma(t\cdot x)=\tau(t)\cdot\sigma(x).
	\end{equation*}
	An equivalence between two toric real structures is a toric automorphism of $T\curvearrowright X$ that sends one onto the other.
\end{dfn}

\begin{dfn}
	 A \emph{real form} of a complex toric variety ${T\curvearrowright X}$ is a real toric variety ${T'\curvearrowright X'}$ together with an isomorphism of complex toric varieties $\phi:(T'_\C\curvearrowright X'_\C)\rightarrow (T\curvearrowright X)$. An equivalence between real forms of $T\curvearrowright X$ is a toric isomorphism between the two real toric varieties whose complexification is compatible with the toric isomorphisms with $T\curvearrowright X$.
\end{dfn}

The complexification of a real toric variety is naturally endowed with a toric real structure. It leads to a functor from the groupoid of real forms of $T\curvearrowright X$ to its groupoid of toric real structures. It is an equivalence of categories. This follows from the fact that the union two $T$-invariant affine open set of the complex toric variety $T\curvearrowright X$ is always quasi-projective. More details about this argument are given by R. Terpereau in the second paragraph of the section 4.1 of his survey \cite{Ronan_survey}. 

\paragraph{Fan of a Real Toric Variety.}
Let $T\curvearrowright X$ be a real toric variety. The real structure $\sigma$ permutes the toric orbits of $X_\C$ just as $\tau$ permutes the cones of the fan of $X_\C$, \textit{cf.} \cite[Proposition 1.19]{Huruguen:2011aa}. It motivates the following definition.

\begin{dfn}
	Let $T\curvearrowright X$ be a real toric variety. Its \emph{fan} is the fan $C$ of the complexification $T_\C\curvearrowright X_\C$ endowed with the action of the Galois group. As such, the couple $(N;C)$ is an \dl\;endowed with an involution $\tau$ that permutes the cones of $C$. A cone is said to be \emph{invariant} whenever it is fixed, not necessarily point-wise, by the involution. The set of invariant cones is denoted by $C^\tau$.
\end{dfn}

As in the complex case, the fan of $T\curvearrowright X$ retains a lot of information about the action. However, in this case, the $T$-invariant objects are not parametrised by the cones of $C$ but rather by $\Z/2$-invariant cones. The real versions of the bijections (\ref{eq:bij_aff_cone_C}), (\ref{eq:bij_orb_cone_C}), and (\ref{eq:bij_subvar_cone_C}) take the forms of an increasing bijection:
\begin{equation}
	(C^\tau;\leq)  \longrightarrow  \left\{ \begin{array}{c} T\textnormal{-Invariant Affine}\\ \textnormal{Open Subvarieties of }X\end{array}\right\}, 
\end{equation}
a decreasing bijection:
\begin{equation}\label{eq:bij_orb_cone_R}
	(C^\tau;\leq)  \longrightarrow  \left\{ \begin{array}{c}\textnormal{Principal Homogenous Toric}\\ \textnormal{Varieties Immersed in }X\end{array}\right\}, 
\end{equation}
and another decreasing bijection:
\begin{equation}\label{eq:bij_subvar_cone_R}
	(C^\tau;\leq)  \longrightarrow  \left\{ \textnormal{Toric Subvarieties of }X\right\}. 
\end{equation}
Nonetheless, these bijections do not exploit all of the information contained in the fan. Let us consider the quotient set $C/\tau$. If $c$ is a cone of $C$, we denote its class in the quotient by $[c]$. One can check that the order relation of $C$ induces an order relation on the quotient. It parametrises the $T$-invariant closed subvarieties of $X$ via a decreasing bijection:
\begin{equation}\label{eq:bij_invariant_subvar}
	\left(C/\tau;\,\leq\right) \longrightarrow \big\{\textnormal{Closed }T\textnormal{-invariant Subvarieties of }X\big\}.
\end{equation}
Let us denote the closed subvariety associated to the class of a cone $c$ by $V[c]$. Its complexification is given by $V(c)\cup V(\tau(c))$, \textit{cf.} $(\ref{eq:bij_subvar_cone_C})$. Either the cone is invariant and it is a real toric variety under the action of the torus modulo isotropy, or it is not geometrically irreducible. If $c$ is not invariant there are two different cases in understanding the real locus of $V[c]$. Either $c+\tau(c)$ belongs to $C$ and then the real locus equals the real locus of the toric subvariety $V[c+\tau(c)]$, or it does not and the real locus is empty. We note that the filtration (\ref{eq:orb_filt}) is real. It is the complexification of the following filtration:
\begin{equation}\label{eq:orb_filt_R}
	X_k\coloneqq \bigcup_{\substack{[c]\in C/\tau\\\dim [c]\geq \dim X-k}} V[c] \subset X,
\end{equation}
for all $0\leq k\leq \dim X$. The graded pieces of (\ref{eq:orb_filt_R}) are decreasingly parametrised by $(C/\tau;\,\leq)$. The complexification of the piece associated to the class $[c]$ is the union of $T_\C$-orbits $O(c)\cup O(\tau(c))$ with induced real structure. The graded pieces of (\ref{eq:orb_filt_R}) observe following alternative:
\begin{enumerate}
	\item[(i)] The cone $c$ is invariant: the graded piece of its class is a principal homogeneous real toric variety under the action of the torus modulo isotropy;
	\item[(ii)] The cone $c$ is not invariant: the graded piece of its class is isomorphic to $\G{\C}^k\rightarrow \Spec\,\R$. In particular, the graded piece has no real point.
\end{enumerate}

\begin{exs} There is only one complete complex toric curve: the projective line. It has three toric real forms. The first is the real projective line on which the real multiplicative group acts by homotheties, the second is the plane conic $\{x^2+y^2=z^2\}$ with its natural action of $\SO$, the third is the “\,empty\,” conic $\{x^2+y^2+z^2=0\}$ with the natural action of the same group. They all have the same fan: a line cut in half. In the first case, the real structure acts trivially on the lattice. In the two other cases, it acts as the multiplication by $-1$. The two conics will be discriminated by their twist class, \textit{cf.} Definition~\ref{dfn:iso_type_real}. Figure~\ref{fig:real_form_p1} depicts the graded pieces of the filtration (\ref{eq:orb_filt_R}).
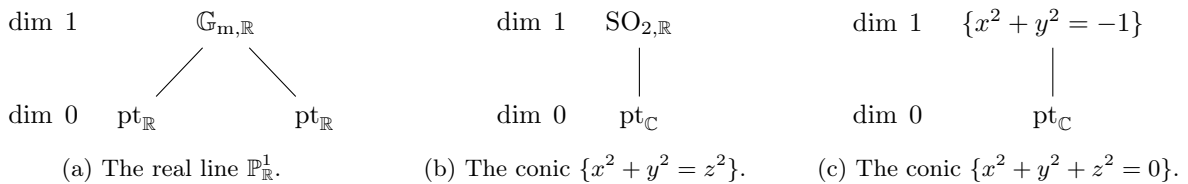
\begin{figure}[!ht]
		\centering
		\begin{subfigure}[b]{.32\textwidth}
			\centering
			\begin{tikzcd}[column sep=tiny]
				\dim\;1 & & \G{\R} \ar[no head,dr]\ar[no head,dl]& \\
				\dim\;0 & \pt_\R &  & \pt_\R
			\end{tikzcd}
			\caption{The real line $\Proj^1_\R$.}
		\end{subfigure}
		\hfill
		\begin{subfigure}[b]{.32\textwidth}
			\centering
			\begin{tikzcd}[column sep=tiny]
				\dim\;1 & \SO \ar[no head,d] \\
				\dim\;0 & \pt_\C 
			\end{tikzcd}
			\caption{The conic $\{x^2+y^2=z^2\}$.}
		\end{subfigure}
		\hfill
		\begin{subfigure}[b]{.32\textwidth}
			\centering
			\begin{tikzcd}[column sep=tiny]
				\dim\;1 & \{x^2+y^2=-1\}\ar[no head,d] \\
				\dim\;0 & \pt_\C
			\end{tikzcd}
			\caption{The conic $\{x^2+y^2+z^2=0\}$.}
		\end{subfigure}
		\caption{The graded pieces of the toric real forms of $\Proj^1_\C$.}
		\label{fig:real_form_p1}
	\end{figure}
	
The product of two lines $\Proj^1_\C\times_\C\Proj^1_\C$ admits seven toric forms, six of which are obtained as products of the toric real forms of $\Proj^1_\C$. The seventh is the Weil restriction $\Res\Proj^1_\C$. Its cocharacter lattice is $\Z[\tau]$. Its fan and graded pieces are depicted in Figure~\ref{fig:ResP1}.
	\begin{figure}[!ht]
		\centering
		\begin{subfigure}[b]{.45\textwidth}
			\centering
			\begin{tikzpicture}[scale=2]
				\draw[ thick] (-1,0) -- (1,0);
				\draw[ thick] (0,-1) -- (0,1);
				\fill (0,0) circle (.025cm);
				\fill (.45,0) circle (.025cm);
				\fill (-.45,0) circle (.025cm);
				\fill (0,.45) circle (.025cm);
				\fill (.45,.45) circle (.025cm);
				\fill (-.45,.45) circle (.025cm);
				\fill (0,-.45) circle (.025cm);
				\fill (.45,-.45) circle (.025cm);
				\fill (-.45,-.45) circle (.025cm);
				\fill (-.9,0) circle (.025cm);
				\fill (.9,0) circle (.025cm);
				\fill (0,.9) circle (.025cm);
				\fill (0,-.9) circle (.025cm);
				\fill[pattern={Lines[angle=-45, yshift=-.6pt , line width=.25pt]}] (0,0) -- (1,0) -- (0,1);
				\fill[pattern={Lines[angle=45, yshift=-.6pt , line width=.25pt]}] (0,0) -- (-1,0) -- (0,1);
				\fill[pattern={Lines[angle=-45, yshift=-.6pt , line width=.25pt]}] (0,0) -- (-1,0) -- (0,-1);
				\fill[pattern={Lines[angle=45, yshift=-.6pt , line width=.25pt]}] (0,0) -- (1,0) -- (0,-1);
				\draw[<->,dashed] (.2,1,0) .. controls (.7,1,0) and (1,.7,0) .. (1,.2,0);
				\draw (.65,.65) node{$\tau$};
			\end{tikzpicture}
			\caption{The fan with action of the real structure.}
		\end{subfigure}
		\hfill
		\begin{subfigure}[b]{.45\textwidth}
			\centering
				\begin{tikzcd}[column sep=tiny]
					\dim\;2 & & \Res\G{\C} \ar[no head,dr]\ar[no head,dl] &\\
					\dim\;1 & \G{\C} \ar[no head,dr]\ar[no head,d]& & \G{\C} \ar[no head,d]\ar[no head,dl]\\
					\dim\;0 & \pt_\R & \pt_\C & \pt_\R
				\end{tikzcd}
			\caption{The graded pieces.}
		\end{subfigure}
		\caption{The fan and graded pieces of $\Res\Proj^1_\C$.}
		\label{fig:ResP1}
	\end{figure}
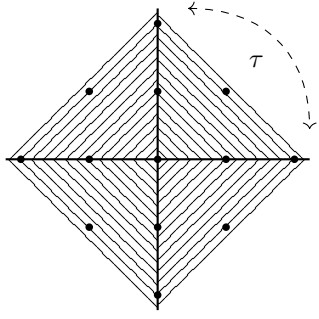
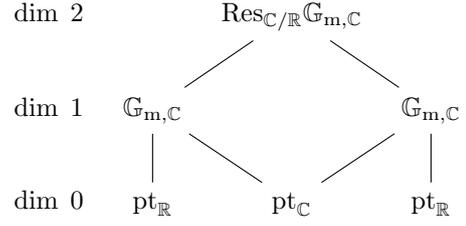
\end{exs}

\paragraph{Principal Homogeneous Toric Varieties.} 
Let us consider a real principal homogeneous toric variety $T\curvearrowright X$. In this case, we are given two real structures on a complex torus. A first one $\tau$ that is “\,linear\,”, and another one $\sigma$ that is “\,affine\,” and whose “\,linear part\,” is $\tau$.

\begin{dfn}[Twist Class]
	Let $T\curvearrowright X$ be a principal homogenous real toric variety, $x$ be a complex point of $X$ and $\varepsilon$ be the element of $T(\C)$ that satisfies $\sigma(x)=\varepsilon x$. The element $\varepsilon$ satisfies the following equation:
	\begin{equation*}
		\varepsilon\cdot\tau(\varepsilon)=1.
	\end{equation*}
We say that $\varepsilon$ is a \emph{twist cocycle} of $X$. Choosing the element $t\cdot x$ instead of $x$ yields the twist cocycle $\varepsilon t \tau(t)^{-1}$. Hence, the class of $\varepsilon$ in $H^1(\Z/2;T(\C))$ is independent of such choice. We call it the \emph{twist class} of $X$. Following Lemma~\ref{lem:group_cohomo_tori}, we will often consider the twist class as an element of the group cohomology of the cocharacter lattice.
\end{dfn}

The twist class is a complete invariant of principal homogenous real toric varieties under the action of a given torus. Let $T\curvearrowright X_1,X_2$ be two such varieties, they are equivariantly isomorphic if and only if they have the same twist class, \textit{cf.} \cite[Remark~1.18]{Huruguen:2011aa} or \cite[Lemma~2.11]{real_horo}. Therefore, the twist class vanishes if and only if the principal homogenous real toric variety admits a real point, \textit{cf.} \cite[Remark~4.1]{Ronan_survey}.   

\begin{rem}
	Let $T\curvearrowright X$ be a $n$-dimensional principal homogeneous complex toric variety. We have the following “\,fibration\,” of groupoids:
	\begin{equation*}
		\begin{tikzcd}
			\left\{\begin{array}{c} \textnormal{Real group} \\ \textnormal{structures} \\ \tau \textnormal{ on }T \end{array}\right\} & \left\{\begin{array}{c} \textnormal{Toric real} \\ \textnormal{structures} \\ (\tau;\sigma) \textnormal{ on }(T;X) \end{array}\right\} \ar[l] & \left\{\begin{array}{c} \textnormal{Real structures } \sigma \\ \textnormal{on }X \textnormal{ compatible} \\ \textnormal{with }T \text{ and }\tau \end{array}\right\} \ar[l]
		\end{tikzcd}
	\end{equation*}
	Modulo equivalence, it has the following description:
	\begin{equation*}
		\begin{tikzcd}[row sep=.05cm]
			\left\{ \begin{array}{c} (p;q)_r\in\N^3 \\ p+q=n \\ r\leq\min(p;q) \end{array}\right\} & \bigslant{\left\{\begin{array}{c} \textnormal{Toric real str.} \\ \textnormal{on } (T;X) \end{array}\right\}}{\textnormal{ Eq.}} \ar[l] & \Big(\bigslant{\Z}{2}\Big)^{q-r} \ar[l] \\
			(\textit{Type}) & (\textit{Equivalence Class})~~~ & (\textit{Twist}\,)\quad
		\end{tikzcd}
	\end{equation*}
	Let $f$ denote the polynomial $\sum_{k=0}^n(\lfloor\frac{n-k}{2}\rfloor+1)t^k$. A simple computation yields $f(2)$ non-equivalent principal homogeneous real toric varieties among which only $f(1)$ admit a real point. 
\end{rem}

\paragraph{The Category of Equivariant Torus Embeddings.} 

\begin{prop}\label{prop:functor_fan}
	Let $f:(T_1\hookrightarrow X_1)\rightarrow (T_2\hookrightarrow X_2)$ be a morphism of real \te s, and $N_i$ be the cocharacter lattice of $T_i$ for all $i\in\{1;2\}$. The group morphism $N_1\rightarrow N_2$ induced by $f$ (that we denote by the same symbol) is $\Z[\tau]$-linear. It maps every cone of $C_1$ in a cone of $C_2$.
\end{prop}

\begin{proof}
	The $\Z[\tau]$-linearity is a consequence of the realness of $f$. The proof of the statement about the cones can be found in \cite[Theorem~3.4.4]{Cox-Lit-Sch_tor_var}.
\end{proof}

\begin{dfn}	Let $\mathcal{C}_\R$ denote the category whose objects are made of couples $(N;C)$ where:
\begin{itemize}
	\item[(i)] $N$ is a $\Z[\tau]$-module whose underlying Abelian group is free of finite rank;
	\item[(ii)] $C$ is a fan of $N\otimes_\Z\R$ whose cones are permuted by $\tau$.
\end{itemize}
A morphism $f$ from $(N_1;C_1)$ to $(N_2;C_2)$ is a $\Z[\tau]$-linear morphism $f:N_1\rightarrow N_2$ that sends every cone of $C_1$ in a cone of $C_2$. Proposition~\ref{prop:functor_fan} ensures that the map that sends a real \te\;$T\hookrightarrow X$ to the couple $(N;C)$ where $N$ is the cocharacter of $T$, and $C$ is the fan of $X$, defines a functor $F$.
\end{dfn}

\begin{prop}
	Let $\textnormal{Eq.Tor.Emb.}_\R$ be the category of real \te s. The functor $F:\textnormal{Eq.Tor.Emb.}_\R \rightarrow \mathcal{C}_\R$ is an equivalence of categories.
\end{prop}

\begin{proof}
	This is a simple elaboration on the similar fact concerning complex \te s, see \cite[Theorem~3.4.4]{Cox-Lit-Sch_tor_var} for instance. The only addition is to keep the equivariance.
\end{proof}

\paragraph{Invariants of Real Toric Varieties.}

\begin{dfn}\label{dfn:iso_type_real}
	Let $T\curvearrowright X$ be a real toric variety. We extend and introduce some terminology:
	\begin{itemize}
		\item[(i)] The \emph{type}, \emph{isogeneous type}, \emph{winding number}, and \emph{winding group} of the real toric variety refer to the corresponding concepts of its torus. In particular, we say that the variety is \emph{unwound} when its torus is;
		\item[(ii)] The \emph{twist class} of the real toric variety refers to the twist class of its principal orbit.
	\end{itemize}
\end{dfn} 

\begin{prop}[Proposition~1.19 and Theorem~1.22 in \cite{Huruguen:2011aa}]\label{prop:real_structure}
	Let $T_\C\curvearrowright X_\C$ be a complex toric variety. For all involutions $\tau$ of its cocharacter lattice $N$ that permute the cones of its fan, and all classes $[\varepsilon]$ in the corresponding cohomology group $H^1(\Z/2;N)$, there is a unique, up to isomorphism, real form of $T_\C\curvearrowright X_\C$ whose cocharacter lattice is given by $(N;\tau)$ and whose twist class is $[\varepsilon]$.
\end{prop}

\paragraph{Quotients.} We will consider several quotients of toric varieties by closed subgroups of its torus. Our quotients will always be realised by toric varieties. Let $T\curvearrowright X$ be a real toric variety and $G$ be a closed subgroup of $T$. We denote by $M$ (resp. $Q$) the character lattice of $T$ (resp. $G$). Let us consider the exact sequence of character groups:
	\begin{equation*}
			0 \rightarrow K \overset{\pi^*}{\longrightarrow} M \longrightarrow Q \rightarrow 0,
	\end{equation*}
	associated to the exact sequence of real diagonalisable groups:
	\begin{equation}\label{eq:quotient_closed_subgroup}
			1 \rightarrow G \longrightarrow T \overset{\pi}{\longrightarrow} T/G \rightarrow 1.
	\end{equation}
	Since $G$ is closed, $T/G$ is a real torus whose cocharacter lattice is given by $\Hom(K;\Z)$. Let $C$ be the fan of $X$, and $[\varepsilon]\in H^1(Z/2;N)$ be its twist class.
	 
\begin{dfn}\label{dfn:quot}
	Let $T\curvearrowright X$ be a real toric variety and $G$ be a closed subgroup of $T$. If the set of images of the cones of $C$ by ${\pi_*:N\rightarrow \Hom(K;\Z)}$ is a fan of $\Hom(K;\Z)\otimes\R$, i.e. the image of every cone is strongly convex, then we denote it by $C/G$.
\end{dfn}

\begin{prop}\label{prop:quot}
	Let $T\curvearrowright X$ be a real toric variety and $G$ be a closed subgroup of $T$. If they meet the requirements of Definition~\ref{dfn:quot}, then there exists a real toric variety $T/G\curvearrowright X/G$ and a morphism of toric varieties $\pi:X\rightarrow X/G$, whose toric part is given by the projection $\pi$ of (\ref{eq:quotient_closed_subgroup}), that realises the quotient of $X$ by the action of $G$. The fan of $X/G$ is given by $C/G$ and its twist class by $\pi_*[\varepsilon]$.
\end{prop}

\begin{proof}
	Let us first note that the category of schemes over a base scheme $S$ admits all inductive limits whose gluing maps are open immersions (this is, in a way, how schemes are defined within locally ringed spaces). In this framework, one can write:
	\begin{equation*}
		X_\C=\varinjlim\,\{U_c : c\in C\},
	\end{equation*}
	where $U_c$ denotes the $T_\C$-invariant affine open subset of $X_\C$ defined by the cone $c$, \textit{cf.} (\ref{eq:bij_aff_cone_C}). Every open set $U_c$ admits a quotient. Since $G$ is linear, it is given by the subring of invariants, \textit{cf.} \cite[Definition~1.3]{Mumford:1965aa}:
	\begin{equation*}
		\C[c^+\cap M]^G\coloneqq \{f\in\C[c^+\cap M]\,|\,\alpha^*(f)=1\otimes f\},
	\end{equation*}
	where $\alpha^*:\C[c^+\cap M]\rightarrow \C[Q]\otimes_\C\C[c^+\cap M]$ is provided by the action $\alpha:G_\C\times_\C U_c\rightarrow U_c$. One can easily show the following equality:
	\begin{equation}\label{eq:invariant_algebra}
		\C[c^+\cap M]^G=\C[c^+\cap \pi^*(K)].
	\end{equation} 
	Let $V$ denote the lattice in the maximal vector space contained in $\pi_*(c)$. The affine variety defined be the algebra described in formula (\ref{eq:invariant_algebra}) is the underlying variety of the affine toric variety associated to the image of the cone $\pi_*(c)$ of $\Hom(K;\Z)$ in the quotient $\Hom(K;\Z)/V$. Under the hypothesis of Definition~\ref{dfn:quot}, $V$ is trivial for $\pi_*(c)$ is strongly convex. Henceforth, the affine variety $U_c/(G_\C)$ is a toric variety under the action of $T_\C/(G_\C)$. Moreover, if $c$ is a face of $d$, the map $U_c/(G_\C)\rightarrow U_d/(G_\C)$ induced by the open immersion corresponds to the open immersion $U_{\pi_*(c)}\rightarrow U_{\pi_*(d)}$ for $\pi_*(c)$ is a face of $\pi_*(d)$. As a consequence, we can form the complex scheme:
	\begin{equation*}
		X_\C/G_\C\coloneqq \varinjlim\,\{U_c/(G_\C) : c\in C\}.
	\end{equation*}
	It is endowed with a canonical morphism $\pi_\C:X_\C\rightarrow X_\C/G_\C$. Using the universal properties of inductive limits and quotients, one can straightforwardly show that $\pi$ is the quotient of $X_\C$ by $G_\C$. Now, we can use (\ref{eq:invariant_algebra}) to further describe $X_\C/G_\C$. We have:
	\begin{equation}
		X_\C/G_\C=\varinjlim\,\{U_{\pi_*(c)} : c\in C\}.
	\end{equation}
	If $Y$ denotes the complex \te\;associated to the fan $C/G$, we have:
	\begin{equation}
		Y=\varinjlim\,\{U_{d} : d\in C/G\}.
	\end{equation}
	Thus, we have a canonical morphism $f:X_\C/G_\C\rightarrow Y$ obtained by gluing the natural inclusions of $U_{\pi_*(c)}$ in $Y$. Reciprocally, for all $d\in C/G$, let us denote by $g_*(d)$ the intersection of all cones $c$ of $C$ such that $\pi_*(c)=d$. The map $g_*$ is covariant and yields a morphism $g:Y\rightarrow X_\C/G_\C$ by gluing the inclusions $U_d= U_{\pi_*(g_*(d))}\subset X_\C/G_\C$. By construction, we have $fg=\id_Y$. To show that $gf=\id_X$ we only need to note that, whenever $c_1\leq c_2\in C$ satisfy $\pi_*(c_1)=\pi_*(c_2)$, the morphism $U_{\pi_*(c_1)}\rightarrow U_{\pi_*(c_2)}$ is the identity.
	
	\vspace{5pt}
	
	Now that we have established that $X_\C$ has a quotient under the action of $G_\C$ and that it is realised by a toric variety, we need to transport the toric real structure on the quotient. We note that, since everything was assumed to be real in the beginning, $\Hom(K;\Z)$ is naturally endowed with an involution $\tau$. This involution permutes the cones of the fan $C/G$ for $\pi_*$ is equivariant, and $C$ is a real fan. Thus, according to Proposition~\ref{prop:real_structure}, to define a toric real structure on $X_\C/G_\C$ we only need to further specify a twist cocyle. In order for $\pi_\C:X_\C\rightarrow  X_\C/G_\C$ to be real, the only choice is to set the quotient twist cocycle to be $\pi_\C(\varepsilon)$. Now, let $T/G\curvearrowright X/G$ be the real toric variety $X_\C/G_\C$ endowed with the toric real structure that we just defined. By construction, we have a morphism of real toric varieties $\pi : X\rightarrow X/G$. We need to show that it is the quotient of $X$ by $G$. Let $\alpha:G\times_\R X \rightarrow X$ be the action and $\textnormal{pr}_X:G\times_\R X \rightarrow X$ the projection onto $X$. Since $(\pi\circ\alpha)_\C=(\pi\circ\textnormal{pr}_X)_\C$, we find that $\pi$ is invariant for base change is a faithful functor. Moreover, let $f:X\rightarrow Z$ be an invariant real morphism. Thus, $f_\C:X_\C\rightarrow Z_\C$ is an invariant morphism. Hence, there is a unique morphism $g_\C:X_\C/G_\C\rightarrow Z_\C$ such that $f_\C=g_\C\circ\pi_\C$. Since $f$ and $\pi$ are real, we have that $f_\C=(\sigma\circ g_\C\circ\sigma)\circ\pi_\C$. Thus, $\sigma\circ g_\C\circ\sigma=g_\C$ by uniqueness. This implies that $g_\C$ is real, i.e. the complexification of some unique $g:X/G\rightarrow Z$ satisfying $f=g\circ\pi$.
\end{proof}

\begin{rem}
	If $G$ is any finite subgroup of $T$ then the requirements of Definition~\ref{dfn:quot} are always satisfied for $\pi_*$ is injective. Following \cite[Theorem~5.1]{zbMATH01361100}, it is a geometric quotient. 
\end{rem}
 
\subsection{Elementary Topological Properties}

\begin{prop}\label{prop:density_real_loc}
	Let $T\curvearrowright X$ be a real toric variety. If $X$ is untwisted then its principal orbit is dense in its real locus.
\end{prop}

\begin{proof}
	Let $x$ be a real point of $X$. Let $c$ be the invariant cone of the fan of $X$ that corresponds to the orbit of $x$ and $v\in N$ be an invariant element of $c$. Since $\tau(v)=v$, it defines a 1-parameter subgroup of $T$. Let $x_0$ be a real point of the principal orbit of $X$. Following \cite[\S2.3]{Fulton:1993aa}, $t\in\R^\times\mapsto t^v\cdot x_0$ converges toward a point of the orbit of $x$ as $t$ tends to $0$. This limit point is real by construction. Therefore, there is an element $u\in T(\R)$ such that $t\in\R^\times\mapsto ut^v\cdot x_0$ converges toward $x$ as $t$ tends to $0$.
\end{proof}

\begin{dfn}[Cellular Dimension]\label{dfn:cell_dim}
	Let $X$ be a real variety. We define the \emph{cellular dimension} of $X$ as the maximum of the integers $m$ such that there is a topological embedding of $\R^m$ into $X(\R)$ (if there is no such embedding we set it to $-1$).
\end{dfn}

\begin{prop}\label{prop:dimension_real_locus}
	Let $T\curvearrowright X$ be a real toric variety, and $[\varepsilon]$ be its twist class. The real locus of $X$ is non-empty if and only if there is an invariant cone $c$ in the fan of $X$ such that $[\varepsilon]$ belongs to the image of $H^1(\Z/2;N_c)\rightarrow H^1(\Z/2;N)$. In this case, the cellular dimension of $X(\R)$ is given by the following expression:
	\begin{equation*}
		\dim X - \min\Big\{ \dim c\,\big|\, [\varepsilon]\in\im\, H^1\big(\Z/2;N_c\big)\rightarrow H^1\big(\Z/2;N\big)\Big\},
	\end{equation*}
	where $c$ ranges among the invariant cones of $X$.
\end{prop}

\begin{proof}
	The complex locus $X(\C)$ can be written as a disjoint union of toric orbits: one for each cone $c$ of the fan of $X$. One of these orbits is stabilised by the real structure if and only if the associated cone $c$ is invariant. In this case, it is isomorphic to a principal homogeneous real toric variety of the real torus $T(c)$ associated with the cocharacter lattice $N(c)$. The twist class of this orbit is given by $\pi[\varepsilon]$, where $\pi$ denotes the projection $T\rightarrow T(c)$. Hence, the orbit of $c$ has a real point if and only if the class $\pi[\varepsilon]$ vanishes in $H^1(\Z/2;T(c)(\C))$, \textit{cf.} \cite[Remark~4.1]{Ronan_survey}. Since this group is the same as $H^1(\Z/2;(c))$, the latter condition is equivalent, using the group cohomology long exact sequence, to $[\varepsilon]$ belonging to the image of $H^1(\Z/2;N_c)\rightarrow H^1(\Z/2;N)$. With this observation the proposition follows from Definition~\ref{dfn:cell_dim}. 
\end{proof}

\begin{rem}\label{rem:fixed_point}
	A direct consequence of Proposition~\ref{prop:dimension_real_locus} is that toric fixed points of real affine toric varieties are always real.
\end{rem}

\begin{dfn}
	Let $T\curvearrowright X$ be a real toric variety. We define the \emph{\tc}~$U$ of $X$ to be the union of all affine open toric subvarieties of $X$, i.e. $\cup_{c\in C^\tau} U_c$ with induced real structure.
\end{dfn}

\begin{prop}\label{prop:real_pt_smooth}
	A real toric variety with smooth topological core has a real point if and only if it is untwisted.
\end{prop}

\begin{proof}
	Let $T\curvearrowright X$ be a smooth real toric variety. Let $c$ be an invariant cone of the fan of $X$. Since $X$ is smooth, $c$ is generated by part of a basis of $N$. The invariance implies that the real structure $\tau$ permutes these generators. Therefore, $N_c$ is isomorphic to $\Z[1]^k\oplus\Z[\tau]^l$. Its first cohomology group vanishes. Thus, by Proposition~\ref{prop:dimension_real_locus}, $X$ has a real point if and only if its twist class vanishes. We could also have argued that if $X$ has a real point then smoothness ensures that there is a real point in the principal orbit.
\end{proof}

We can note that C. Delaunay proved this statement for smooth real toric variety with compact real locus in her thesis, \textit{cf.} \cite[Theorem~4.1.1]{delau2004}.

\begin{lem}\label{lem:localisation_real_points}
	Let $T\curvearrowright X$ be a real toric variety. Its topological core contains all its real points. 
\end{lem}

\begin{proof}
	By definition, we have the inclusion of $U(\R)$ in $X(\R)$. Conversely, if $x\in U_c(\C)$ is a fixed point of the real structure, then it must belong to both $U_c(\C)$ and $\sigma(U_c(\C))$. The latter open set is $U_{\tau(c)}(\C)$. Hence, the intersection of these two sets is $U_{c\,\cap\tau(c)}(\C)$. The cone $c\cap\tau(c)$ is invariant by construction, thus $x$ must belong to $U(\R)$.
\end{proof}

\begin{rem}
	If $f:(T_1\curvearrowright X_1)\rightarrow (T_2\curvearrowright X_2)$ is a morphism of real toric varieties and $U_1$ (resp. $U_2$) denotes the \tc\;of $X_1$ (resp. $X_2$) then $f(U_1)\subset U_2$. Let $T\hookrightarrow X$ be a real \te. In this case, its topological core $U$ is the smallest equivariant open neighbourhood of $X(\R)$ in $X$.
\end{rem}

\begin{prop}\label{prop:compactness_real_locus}
	Let $T\hookrightarrow X$ be a real \te. The real locus of $X$ is compact if and only if the group $\ker(1-\tau)\subset N$ is contained in the support of the fan of $X$.
\end{prop}

\begin{proof}
	We denote the fan of $X$ by $C$. Following \cite[\S2.3]{Fulton:1993aa}, the complement of the support of $C$ in $N$ is characterised by the following property: \emph{A 1-parameter subgroup $v:t\in\C^\times\mapsto t^v$ has a subsequential limit at $0$ if and only if it belongs to the support of $C$}. If $X(\R)$ is compact then every real 1-parameter subgroup $v\in N$ has a subsequential limit at $0$ for the image of $\R^\times$ by $v$ is contained in $X(\R)$. Thus, $\ker(1-\tau)$ is contained in the support of $C$.
	
	\vspace{5pt}
	
	Conversely, let us assume that $\ker(1-\tau)$ is contained in the support of $C$. We consider $Y$, an equivariant completion of $X$, \textit{cf.} \cite[Theorem~4.13]{Sumihiro:1975aa}. It is a real \te\; of $T$. We denote its fan by $D$. It contains $C$ as a sub-fan. If we can show that every invariant cone of $D$ belongs to $C$ then Lemma~\ref{lem:localisation_real_points} implies that $X$ and $Y$ have the same real locus. Hence, the real locus of $X$ will be compact for $Y$ is complete. Let $d$ be an invariant cone of $D$, and $v$ be a lattice point of the relative interior of $d$. By invariance of $d$, the point $v+\tau(v)$ also belongs to the relative interior of $d$. However, this point $v+\tau(v)$ belongs to $\ker(1-\tau)$ and thus to support of $C$. Since the relative interiors of the cones of $D$ form a partition of $N\otimes \R$, it implies that $d$ belongs to $C$.
\end{proof}

\begin{prop}\label{prop:smooth_equiv_completion}
	Let $T\curvearrowright X$ be a real toric variety that has smooth topological core and compact real locus. There exists a smooth and complete real toric variety $T\curvearrowright X'$ and an equivariant birational map $X\dashrightarrow X'$ that induces an isomorphism between their topological cores. The map can be taken to be a morphism if $X$ is smooth.  
\end{prop}
	
\begin{proof}
	By H. Sumihiro's equivariant completion, \textit{cf.} \cite[Theorem~4.13]{Sumihiro:1975aa}, we can find a complete toric variety $T\curvearrowright \overline{X}$ and an equivariant open immersion $X\rightarrow \overline{X}$. Thus, the fan $C$ of $X$ is included in a complete equivariant fan $\overline{C}$, the fan of $\overline{X}$. Since $X$ has compact real locus, Proposition~\ref{prop:compactness_real_locus} ensures that $\ker(1-\tau)$ is contained in the support of $C$. Moreover, the proof of the same proposition establishes that every invariant cone of $\overline{C}$ is a cone of $C$. Therefore, $\overline{X}$ has smooth topological core. The completion $X\rightarrow \overline{X}$ is even an isomorphism between their topological cores. One can adapt the method described in \cite[\S2.6]{Fulton:1993aa} to resolve the singularities of $\overline{X}$ and obtain $T\curvearrowright X'$. This resolution $X'\rightarrow \overline{X}$ yields an isomorphism between the topological cores for $\overline{X}$ has smooth topological core. Hence, the resulting equivariant birational map $X\dashrightarrow X'$ yields an isomorphism between their topological cores.
\end{proof}

\section{Structure of Affine Toric Varieties}

\subsection{The Affine Fibration}

In this section, we investigate the structure of a real affine toric variety $T\curvearrowright X$. Let us denote by $(\tau;\sigma)$ the real structure of $X$, by $N$ the cocharacter lattice of $T$, and by $c$ the cone whose faces form the fan of $X$. We recall that $N_c$ denotes the subgroup of $N$ spanned by the integral elements of $c$, and $N(c)$ denotes the quotient $N/N_c$. The exact sequence of $\Z[\tau]$-modules:
\begin{equation*}
	0\rightarrow N_c\rightarrow N \rightarrow N(c)\rightarrow 0,
\end{equation*}
yields an exact sequence of real tori: 
\begin{equation}\label{eq:alg_loc_triv}
	1\rightarrow T_c\rightarrow T \rightarrow T{(c)}\rightarrow 1.
\end{equation}
We denote its projection by $\pi$. Let $X{(c)}$ denote the quotient of $X$ by $T_c$, \textit{cf.} Definition~\ref{dfn:quot}. It is a principal homogeneous real toric variety under the action of $T{(c)}$. We denote the quotient map by $\pi:X\rightarrow X{(c)}$.

\begin{prop}\label{prop:proj_base_affine_split}
	$X$ has a real point if and only if $X{(c)}$ has a real point.
\end{prop}

\begin{proof}
	If $X$ has a real point then so does $X{(c)}$. Conversely, if $X{(c)}$ has a real point then its twist class vanishes by \cite[Remark~4.1]{Ronan_survey}. The twist class of $X{(c)}$ is the image by $\pi$ of the twist class of $X$. Therefore, the twist class of $X$ lies in the image of the morphism $H^1(\Z/2;N_c)\rightarrow H^1(\Z/2;N)$. Thus Proposition~\ref{prop:dimension_real_locus} ensures that $X$ has a real point.
\end{proof}

From now on, we assume that $X$ and $X{(c)}$ have a real point $x_0$. We denote by $X_c$ the fibre of $\pi$ over $x_0$. It is an affine real toric variety under the action of $T_c$. It has the same fan as $X$ but seen in $N_c$. Following \cite[\S2.1]{Fulton:1993aa}:
\begin{equation}\label{eq:triv_fib_bun}
	(X_c)_\C\rightarrow X_\C \rightarrow X(c)_{\C},
\end{equation}
is a trivial fibre bundle. Let us note that the morphism $T{(c)}\rightarrow X{(c)}$ sending $t$ to $t\cdot x_0$ is an isomorphism in this case. To trivialise (\ref{eq:triv_fib_bun}), we consider a section $s$ of the complexification of $\pi:T\rightarrow T{(c)}$, n.b. it is equivalent to a section of $N\rightarrow N(c)$ in the category of Abelian groups. A trivialisation of the fibre bundle (\ref{eq:triv_fib_bun}) is then given as follows:
\begin{equation}\label{eq:triv_coordinates_C}
	\begin{array}{ccl}
	X(c)_{\C} \times_\C (X_{c})_\C & \longrightarrow & X_\C \\[5pt]
	~(x;y) & \longmapsto & s(t)\cdot y,
	\end{array}
\end{equation}
where $t$ is the unique element of $T(c)$ satisfying $x=t\cdot x_0$. We want to answer the question:
\begin{center}
	\emph{To what extent does the real analog of (\ref{eq:triv_fib_bun}) can be interpreted as a fiber bundle ?}
\end{center}
In general, it cannot be interpreted as a fibre bundle from the point of view of algebraic geometry. The only obstruction is the failure to construct local sections of $\pi:T\rightarrow T{(c)}$. Further, we will see that the bundle can be non-trivial even if such local sections can be constructed. From the topological point of view, (\ref{eq:triv_fib_bun}) can always be thought of as a disjoint union of fibre bundles, the fibres being real loci of varying real forms of $(X_{c})_\C$. Let us denote by $(\tau_c;\sigma_c)$ the real structure of $X_c$, by $[\varepsilon_c]$ its twist class, and by $(\tau_{(c)};\sigma_{(c)})$ the real structure of $X(c)$. The expression $t \mapsto \tau s (t)/s\,\tau_{(c)}(t)$ defines an anti-regular morphism of complex tori that we denote by $\delta:T{(c)}_{\,\C} \rightarrow  (T_{c})_\C$. It satisfies the identity:
\begin{equation}\label{eq:identity_phi}
	\tau_c\,\delta\cdot\delta\,\tau_{(c)}=1.
\end{equation}
Using the “\,coordinates\,” (\ref{eq:triv_coordinates_C}), the real structure $\sigma$ is given by the following formula:
\begin{equation}\label{eq:sigmaX_coordinates}
	\sigma(x;y)=\big(\sigma_{(c)}(x);\delta(t)\cdot\sigma_c(y)\big),
\end{equation}
where $x=t\cdot x_0$. The identity (\ref{eq:identity_phi}) and the expression (\ref{eq:sigmaX_coordinates}) imply that the fibre of $\pi$ over a real point $t\cdot x_0$ is a real form of $(X_{c})_\C$. Its real structure is given by $\delta(t)\cdot\sigma_c$ so its twist class is $[\delta(t)]+[\varepsilon_c]$.

\begin{prop}\label{prop:surj_real_loci}
	The projection $\pi:X\rightarrow X{(c)}$ is surjective over the real loci.
\end{prop}

\begin{proof}
	By construction, the defining cone of $(X_c)_\C$ has the same dimension as the cocharacter lattice $N_c$. Hence, $(X_c)_\C$ has a toric fixed point, which has to be real for all toric real forms. Since the real fibres of $\pi$ are real forms of $(X_c)_\C$, they all have a real point and $\pi$ is surjective over the real loci.
\end{proof}

We should note that, despite Proposition~\ref{prop:surj_real_loci}, $\pi:T\rightarrow T{(c)}$ is not necessarily surjective over the real loci. The lack of surjectivity is precisely assessed by group cohomology.

\begin{lem}\label{lem:cennecting_morphism}
	Let $\df:T{(c)}(\R)\rightarrow H^1(\Z/2;T_c(\C))$ denote the connecting morphism of the cohomological long exact sequence. For all real points $t$ of $T{(c)}$, $\df t$ equals $[\delta(t)]$.
\end{lem}

\begin{proof}
	A lift of $t\in T{(c)}(\R)$ in $T(\C)$ is given by $s(t)$. By definition, $\df t$ is the cohomology class of $\tau s(t)/s(t)$ i.e. $\delta(t)$ for $t$ is real.
\end{proof}

\begin{prop}\label{prop:image_pi_real_loci_torus}
	The image of $T(\R)$ by $\pi:T\rightarrow T{(c)}$ is a close and open subgroup of finite index of the real locus of $T(c)$. 
\end{prop}

\begin{proof}
	Lemma~\ref{lem:cennecting_morphism} ensures that $\df:T{(c)}(\R)\rightarrow H^1(\Z/2;T_c(\C))$ is continuous. Since the image of $\pi$ is the kernel of $\df$, it is a closed subgroup. Following Lemma~\ref{lem:group_cohomo_tori}, it has finite index. Closed subgroups of finite index are necessarily open. 
\end{proof}

	The group $T(\R)$ acts continuously on the real points of $X{(c)}$ through $\pi$. Proposition~\ref{prop:image_pi_real_loci_torus} implies that the real locus of $X{(c)}$ is the topological disjoint union of the orbits of this action. To every such orbit $\omega$ we can associate an equivalence class of real forms of $(X_c)_\C$. Let us denote it by $X_{c}^{\omega}$. For instance, if $\omega_0$ is the orbit of $x_0$ then $X_c^{\omega_0}$ is just $X_c$. The toric variety $X_c^{\omega}$ has the same fan as $X_c$ but its twist class is given by $[\delta(t)]+[\varepsilon_c]$ for any real point $t$ of $T{(c)}$ such that $t\cdot x_0$ belongs to $\omega$.
	
\begin{thm}\label{thm:fibration_affine_real_loci}
	Let $T\curvearrowright X$ be an affine real toric variety with a real point. The real part of the projection $\pi:X\rightarrow X{(c)}$ splits as the disjoint union of the following locally trivial fibrations:
	\begin{equation*}
		X_{c}^{\omega}(\R) \rightarrow \pi^{-1}(\omega) \rightarrow \omega,
	\end{equation*}
	for all $T(\R)$-orbits $\omega$ of the real locus of $X{(c)}$. Furthermore, the structure group of every such fibration is $T_c(\R)$, and the associated principal bundle is given by the following exact sequence of Lie groups:
	\begin{equation*}
		1\rightarrow T_c(\R) \rightarrow T(\R) \rightarrow \pi\big(T(\R)\big)\rightarrow 1.
	\end{equation*}
	If we further assume that $\pi:T\rightarrow T{(c)}$ induces a surjection between the real loci, then:
	\begin{equation*}
		X_c\rightarrow X\rightarrow X{(c)},
	\end{equation*}
	is an algebraic fibre bundle of structure group $T_c$ and principal bundle:
	\begin{equation*}
		1\rightarrow T_c \rightarrow T \rightarrow T{(c)}\rightarrow 1.
	\end{equation*}
\end{thm}
	
\begin{proof}
	Let $\omega$ be a $T(\R)$-orbit of the real locus of $X{(c)}$. Using the trivialisations (\ref{eq:triv_coordinates_C}) and (\ref{eq:sigmaX_coordinates}), we find the following description:
	\begin{equation*}
		\pi^{-1}(\omega)=\big\{(x;y)\in\omega\times X_c(\C)\;|\; y=\delta(t)\cdot\sigma_c(y)\big\},
	\end{equation*}
	where $t$ is the unique real point of $T{(c)}$ satisfying $x=t\cdot x_0$. By construction, every fibre is homeomorphic to $X_c^{\omega}(\R)$. Now let $s_\R:U\rightarrow T(\R)$ be a continuous section of $\pi:T(\R)\rightarrow \im(\pi)$ in a neighbourhood $U$ of the identity. We note that $u:g\mapsto s(g)/s_\R(g)$ takes its values in $T_c(\C)$. Moreover, we have:
	\begin{equation*}
		\delta|_U=\frac{\tau s}{s}=\frac{\tau s}{s}\cdot\frac{s_\R}{\tau s_\R}= \frac{\tau u}{u} = \frac{\tau_c u}{u}.
	\end{equation*}
	Let us choose an origin $x_\omega=t_\omega\cdot x_0$ in $\omega$, and consider, as a model for $X^\omega_c(\R)$, the fibre above the origin $\{y\in X_c(\C)\,|\,y=\delta(t_\omega)\cdot\sigma_c(y)\}$. The orbit $\omega$ is covered by the family of open sets $(U_t\coloneqq \pi(t)\cdot U\cdot x_\omega)_{t\in T(\R)}$. For every $t\in T(\R)$, we can find, although not continuously in general, an element $a_t\in T_c(\C)$ such that $\delta(\pi(t))$ equals $a_t/\tau_c(a_t)$. Indeed, the cohomology class of $\delta(\pi(t))$ is $\df\pi(t)$ which vanishes by exactness of the cohomological long exact sequence. Let us consider the following homeomorphism:
	\begin{equation*}
		\begin{array}{ccl}
			U_t\times X_c^\omega(\R) & \longrightarrow & \pi^{-1}(U_t) \vspace{.25cm} \\
			(x;y) & \longmapsto & \big(\,x;\tfrac{a_t}{u(g)}\cdot y\,\big),
		\end{array}
	\end{equation*}
where $g$ is the unique element of $U$ satisfying $x=\pi(t)\cdot g\cdot x_\omega$. It trivialises $\pi^{-1}(\omega)\rightarrow \omega$ above $U_t$. Moreover, a direct computation shows that the change of trivialisation from $U_{t_1}$ to $U_{t_2}$ is given by:
\begin{equation*}
	(x;y)\mapsto \left(x; \tfrac{a_{t_1}u(g_2)}{a_{t_2}u(g_1)} \cdot y\right)
\end{equation*}
where $g_i$ is the unique element of $U$ satisfying $x=\pi(t_i)\cdot g_i\cdot x_\omega$, for every $i\in\{1;2\}$. The continuous map:
\begin{equation*}
	x\in U_{t_1}\cap U_{t_2} \mapsto \frac{a_{t_1}u(g_2)}{a_{t_2}u(g_1)}\in T_c(\C), 
\end{equation*}
takes it values in $T_c(\R)$. Indeed, by construction, we have:
\begin{equation*}
	\tau_c\left( \frac{a_{t_1}u(g_2)}{a_{t_2}u(g_1)} \right) = \frac{a_{t_1}u(g_2)}{a_{t_2}u(g_1)} \cdot \frac{\delta(\pi(t_2))}{\delta(\pi(t_1))} \cdot \frac{\delta(g_1)}{\delta(g_2)}=\frac{a_{t_1}u(g_2)}{a_{t_2}u(g_1)},
\end{equation*}
for $\delta$ is a group morphism and $g_1/g_2=\pi(t_1)/\pi(t_2)$. We note that, if we replace $\omega$ by $\pi(T(\R))$ and $X_c^\omega(\R)$ by $T_c(\R)$, the exact same formulæ provide local trivialisations of the bundle $T_c(\R)\rightarrow T(\R)\rightarrow \pi(T(\R))$. If we further assume that $\pi:T\rightarrow T{(c)}$ induces a surjection on the real loci, Lemma~\ref{prop:local_section_tori} allows us to find an open neighbourhood $U$ of the identity of $T{(c)}$, whose translates cover $T{(c)}$, and a section $r:U\rightarrow T$ of $\pi$. It allows us to mimic (\ref{eq:triv_coordinates_C}) algebraically. The following isomorphism:
\begin{equation*}
		\begin{array}{ccl}
			(U\cdot x_0)\times_\R X_c & \longrightarrow & \pi^{-1}(U\cdot x_0) \vspace{.25cm} \\
			(t\cdot x_0;y) & \longmapsto & r(t)\cdot y,
		\end{array}
	\end{equation*}
is a local trivialisation of $\pi:X\rightarrow X{(c)}$ with fibre $X_c$. We can propagate this construction to a full atlas of local trivialisations via the action of $T$.
\end{proof}

\begin{rem}
	Theorem~\ref{thm:fibration_affine_real_loci} cannot really be improved when $\pi:T\rightarrow T{(c)}$ does not induce a surjection of the real loci for, in this case, its image cannot be the real locus of an algebraic subgroup of $T{(c)}$.
\end{rem}

\begin{ex}
	\begin{multicols}{2}
		Let $T$ be the three dimensional real torus defined by the following data:
		\begin{equation*}
			\left\{\begin{array}{l}
				N=\langle \partial x;\partial y;\partial z\rangle_\Z \\
				\tau(\partial x)=\partial y \textnormal{ and } \tau(\partial z)=\partial z.
			\end{array}\right.
		\end{equation*}
		We consider the affine untwisted real toric variety $X$ spanned by the cone:
		\begin{equation*}
			c= \langle\partial x-\partial y+\partial z; \partial y-\partial x+\partial z\rangle_{\R_+}.
		\end{equation*}
		The variety $T\curvearrowright X$ has isogeneous type $(2;1)$ and winding number $1$. Since $c$ is bidimensional, the base $X(c)$ has dimension 1, and the fibre $X_c$ has dimension $2$. The real loci of the tori observe the following exact sequence:
		\begin{equation*}
			\begin{tikzcd}[row sep=.05cm, column sep=.25cm]
				1 \ar[r]& T_c(\R) \ar[r]& T(\R) \ar[r]& T(c)(\R) \ar[r]& \Z/2 \ar[r]& 1.
				\end{tikzcd}
		\end{equation*}
		It can be writen as follows:
		\begin{equation*}
			\begin{tikzcd}[row sep=.05cm, column sep=.55cm]
				1 \ar[r]& \Sph^1 \ar[r,"\subset" above]& \C^\times \ar[r,"|z|^2"]& \R^\times \ar[r, "{\tfrac{x}{|x|}}"]& \Z/2 \ar[r]& 1. \\
				& \R^\times \ar[u, phantom, shift left=.5ex, "{\times}" description] \ar[r,"\id~" below]& \R^\times \ar[u, phantom, shift left=.5ex, "{\times}" description]
			\end{tikzcd}
		\end{equation*}
		The group $\pi(T(\R))$ consists only of the positive real numbers, so there is two orbits. The semi-group $c^+\cap M$ is generated by the five vectors:
		\begin{equation*}
			\pm(\df x+\df y),\; \df x+\df z,\; \df y+\df z,\textnormal{ and } \df z,
		\end{equation*}
		with the two relations:
		\begin{equation*}
			\left\{\begin{aligned}(\df x+\df z) + (\df y+\df z)  &= (\df x+\df y)+2\df z\\
			(\df x+\df y)-(\df x+\df y)&=0.
			\end{aligned}\right.
		\end{equation*}
		Hence, the algebra of functions of $X$ is given by:
		\begin{equation*}
			\smallslant{\C[t^{\pm 1};u;v;w]}{(uv-tw^2)}.
		\end{equation*}
		The projection on the $t$-coordinate is exactly the projection over the base $X(c)$, and the fibre $X_c$ above $1$ is given by the quotient algebra $t=1$:
		\begin{equation*}
			\smallslant{\C[u;v;w]}{(uv-w^2)}.
		\end{equation*}
		In these coordinates, the action of an element $(x;y;z)$ of $T(\C)$ on a point $(t;u;v;w)$ of $X(\C)$ is:
		\begin{equation*}
			(x;y;z)\cdot(t;u;v;w)=(xyt;xzu;yzu;zw).
		\end{equation*}
		If one chooses $t\mapsto(t;1;1)$ as section $s$, then, for all $t\in\R^\times$, we have:
		\begin{equation*}
			\delta(t)=(1/t;t;1).
		\end{equation*}
		The real structure of the fibre over $t$ is given by:
		\begin{equation*}
			(u;v;w)\mapsto (\bar{v}/t;t\bar{u};\bar{w}).
		\end{equation*}
		Therefore, the fibre of $X(\R)\rightarrow \R^\times$ over $t$ is $\{(u;w)\in\C\times\R\;|\;t|u|^2=w^2\}$. It consists of a single point when $t$ is negative and a quadratic cone when $t$ is positive.
	\end{multicols}   
\end{ex}

\subsection{Fibration Invariants of Smooth Affine Torus Embeddings}

Here, we want to further study smooth affine real toric varieties that admit a real point. According to Proposition~\ref{prop:real_pt_smooth}, those are necessarily untwisted. Thus, we will assume that we are given a smooth affine real \te\;$T\hookrightarrow X$. Let $c$ denote its cone. The image of $1$ by $\pi:X\rightarrow X{(c)}$ induces an isomorphism between $T{(c)}$ and $X{(c)}$. With this identification, we will assume that $\pi$ sends $X$ onto $T{(c)}$. Moreover, $X_c$ will denote the fibre of $\pi$ over $1$. The restriction of $T\hookrightarrow X$ to $T_c$ takes its values in $X_c$. Therefore, $T_c\hookrightarrow X_c$ is a real smooth affine \te. Moreover, the following diagram is commutative:
\begin{equation}
	\begin{tikzcd}
		& X_c \ar[r] & X \ar[r,"\pi"] & T{(c)} \\
		0 \ar[r] & T_c \ar[u, hookrightarrow] \ar[r] & T \ar[u, hookrightarrow] \ar[r,"\pi" below] & T{(c)} \ar[u, equal] \ar[r] & 0
	\end{tikzcd}
\end{equation}

\begin{prop}\label{prop:line_dec_affine}
	Let $T\hookrightarrow X$ be smooth affine real \te\;defined by a cone $c$. The fibre bundle $X_c\rightarrow X \rightarrow T{(c)}$ is a vector bundle. Every toric subvariety $Y$ induces a sub-vector bundle $Y\rightarrow T{(c)}$. If $Y<X$ is maximal among the toric subvarieties, then either $Y$ is a divisor and $X/Y\rightarrow T{(c)}$ is a real line bundle, or $Y$ has codimension 2 and $X/Y\rightarrow T{(c)}$ is a complex line bundle. Furthermore, the sum of the projections:
	\begin{equation}\label{eq:iso_line_bundle}
		X\longrightarrow \bigoplus_{\substack{Y\textnormal{ maximal}\\ \textnormal{toric subvariety}}} \bigslant{X}{Y},
	\end{equation}
	is an isomorphism of real vector bundles.
\end{prop}

\begin{proof}
	Let us assume that $c$ is spanned by $k$ invariant vectors and $l$ pairs of exchanged vectors. These vectors form a basis of $N_c$. Using this basis, we find that:
	\begin{equation}\label{eq:shape_Xc_Tc}
	\left\{\begin{aligned}
		T_c&\cong \G{\R}^k\times_\R\Res\G{\C}^l \\
		X_c&\cong \mathbb{A}^k_\R\times_\R\Res\mathbb{A}^l_\C,
	\end{aligned}\right.
	\end{equation}
	Since $T_c$ acts on $X_c$ via linear automorphisms, Theorem~\ref{thm:fibration_affine_real_loci} ensures that $X\rightarrow T{(c)}$ is a vector bundle. Recall that $X$ is isomorphic to the variety $(X_c)_\C\times_\C\,T{(c)}_\C$ endowed with the real structure:
	\begin{equation}
		\sigma(x;t)=\big(\delta(t)\cdot\sigma_c(x);\tau_{(c)}(t)\big),
	\end{equation}
	defined with some anti-equivariant anti-regular morphism $\delta:T{(c)}_\C\rightarrow (T_{c})_\C$. Let $Y$ be a toric subvariety of $X$. It is given by an invariant face $c'$ of $c$. The immersion $Y_\C\rightarrow (X_c)_\C\times_\C\,T{(c)}_\C$ corresponds to the vanishing of the coordinates (\ref{eq:shape_Xc_Tc}) provided by the rays of $c$ contained in $c'$. Thus, $Y$ is given by the restriction of $\sigma$ to $(X_c\cap Y)_\C\times_\C \,T{(c)}_\C$. This shows that $Y\rightarrow T{(c)}$ is a sub-vector bundle. Since the action of $T_c$ on $X_c$ preserves the decomposition given by the coordinates (\ref{eq:shape_Xc_Tc}), we see that $X$ is a sum of real and complex line bundles. Each of these line bundles corresponds to the vanishing of all but one coordinate of (\ref{eq:shape_Xc_Tc}). If $f$ is a coordinate of (\ref{eq:shape_Xc_Tc}), let us denote by $Y$ the toric subvariety of $X$ whose complexification corresponds to its vanishing. We also denote by $L$ the associated line bundle. We note that if $f$ is a complex coordinate then $L$ is a complex line bundle and $Y$ has codimension 2 for $Y_\C$ is given by the vanishing of the two induced coordinates on $\Res\A^1_\C\times_\R\Spec\,\C\cong\A^2_\C$. We have the decomposition:
	\begin{equation}
		X\cong Y\oplus L,
	\end{equation}
	so that $L$ is isomorphic to $X/Y$. Finally, since every maximal toric subvariety arises uniquely in such a way, we find the isomorphism (\ref{eq:iso_line_bundle}).
\end{proof}

\begin{prop}
	$T\hookrightarrow X$ be smooth affine real \te\;defined by a cone $c$. Every complex summand of $X\rightarrow T{(c)}$ obtained from a maximal toric subvariety is trivial.
\end{prop}

\begin{proof}
	Every such summand is isomorphic to a bundle of the form $T(c)_\C\times_\C \mathbb{A}^2_\C \rightarrow T(c)_\C$  where the real structure of the total space has the form:
	\begin{equation*}
		\sigma:(t;x;y)\mapsto\big(\tau_{(c)}(t);\delta(t)\cdot(\bar{y};\bar{x})\big),
	\end{equation*}
	for some anti-regular morphism $\delta:T(c)_\C\rightarrow\G{\C}^2$ whose components $\delta_x$ and $\delta_y$ satisfy the relation:
	\begin{equation*}
		\overline{\delta_y}\cdot(\delta_x\circ \tau_{(c)})=1.
	\end{equation*}
	The morphism $\overline{\delta_y}:T(c)_\C\rightarrow \G{\C}$ is regular, and the isomorphism of complex toric varieties: 
	\begin{equation*}
		\begin{array}{rcl}
			T(c)_{\C}\times_\C \mathbb{A}^2_\C & \longrightarrow &T(c)_{\C}\times_\C \mathbb{A}^2_\C \\[5pt]
			(t;x;y) & \longmapsto & \big(t;\overline{\delta_y}(t)\cdot x;y\big),
		\end{array}
	\end{equation*}
	becomes real once we endow the target with the real structure $(t;x;y)\mapsto(\tau\raisebox{-.05em}{$\scriptstyle(c)$}(t);\bar{y};\bar{x})$. It commutes with the projection onto $T(c)_\C$ that is real for both structures. Thus, it yields an isomorphism of our summand with the trivial bundle $T{(c)}\times_\R \Res\mathbb{A}^1_\C \rightarrow T{(c)}$.
\end{proof}

\begin{prop}\label{prop:chern_class_line_affine}
	Let $T\hookrightarrow X$ be a smooth affine \te, and $D$ be a toric prime divisor. In the first Chow group of $X$, we have:
	\begin{equation*}
		c_1\big( \pi^*X/D \big) + [D] = 0.
	\end{equation*}
\end{prop}

\begin{proof}
	Using Theorem~\ref{thm:fibration_affine_real_loci} and the fact that $T\rightarrow T(c)$ is surjective over the real loci, we consider an open neighbourhood $U$ of the identity in $T(c)$ whose translates $(U_u\coloneqq u\cdot U)_{u\in T(\R)} $ cover $T(c)$ and over which $X_c\rightarrow X\rightarrow T(c)$ is trivial. We denote by $V_u$ the inverse image $\pi^{-1}(U_u)$, and by $\chi_u:V_u\rightarrow U_u\times_\R X_c$ a trivialisation for every $u\in T(\R)$. Moreover, we consider, for every $u,v\in T(\R)$, the morphism $g_{u,v}:(U_u\cap U_v)\rightarrow T_c$ that satisties:
	\begin{equation*}
		(\chi_v\circ\chi_u^{-1})(t;x)=(t;g_{u,v}(t)\cdot x).
	\end{equation*}	
	The divisor $D\cap X_c$ corresponds to the vanishing of a real coordinate $\eta$ of (\ref{eq:shape_Xc_Tc}). We note that it yields a morphism of real toric varieties $\eta:(T_c\curvearrowright X_c)\rightarrow (\G{\R}\curvearrowright \mathbb{A}^1_\R)$. Let us consider $u\in T(\R)$ and denote by $f_u$ the function $\eta\circ\textnormal{pr}_{X_c}\circ\chi_u$ over $V_u$. The collection $(V_u\,;f_u)_{u\in T(\R)}$ is a system of local equations representing $D$. Over $V_u\cap V_v$, we have:
	\begin{equation*}
		\begin{aligned}
			f_v&=\eta\circ\textnormal{pr}_{X_c}\circ\chi_v \\
			&=\eta\circ(\textnormal{pr}_{X_c}\circ\chi_v \circ\chi_u^{-1})\circ\chi_u\\
			&=\eta\circ\big((g_{u,v}\circ\pi)\cdot\textnormal{pr}_{X_c}\big)\circ\chi_u\\
			&=(\eta\circ g_{u,v}\circ\pi)\cdot f_u.
		\end{aligned}
	\end{equation*}
Hence, the transition functions of $\O_X(-D)$ are given by $((\eta\circ g_{u,v}\circ\pi))_{u,v\in T(\R)}$. Besides, $X/D$ is also trivial above $U_u$ with a trivialisation $\chi_u'$. In the simultaneous trivialisations $\chi_u$ and $\chi_u'$, the projection $X\rightarrow X/D$ is given by:
	\begin{equation*}
		(\id_{U_u}\times \eta):U_u\times_\R X_c \rightarrow U_u\times \A^1_\R.
	\end{equation*}
	Therefore, the change of trivialisations $\chi_v'\circ (\chi_u')^{-1}$ is then given by $(x;y)\mapsto(x;\eta(g_{u,v}(x))\cdot y)$. Thus, $\pi^*X/D$ is isomorphic to $\O_X(-D)$.
\end{proof}

We note that the flat pullback by $\pi:X\rightarrow T(c)$ is an isomorphism of Chow groups for $\pi$ is a vector bundle, \textit{cf.} \cite[Theorem~3.3~(a)]{Fulton:1998aa}. To conclude with the characteristic classes of the line bundles $X/D$, let us look at the image of their Chern classes under the isomorphism (\ref{eq:can_iso_chow}).

\begin{dfn}
	Let $T\hookrightarrow X$ be a smooth affine real \te. For every prime toric divisor $D$, we denote by $v_D$ the primitive generator of its associated invariant ray in the fan of $X$. The collection of their classes is a basis of $H^2(\Z/2;N_c)$.
\end{dfn}

\begin{prop}
	Let $T\hookrightarrow X$ be a smooth affine real \te\;and $c$ be its defining cone. We denote the isomorphism of (\ref{eq:can_iso_chow}) by $f:\Ch^1(T(c))\rightarrow H^1(\Z/2;M(c))$, and  the connecting morphism in the cohomological long exact sequence by $\df:H^1(\Z/2;N(c))\rightarrow H^2(\Z/2;N_c)$. For all classes $[v]\in H^1(\Z/2;N(c))$, we have:
	\begin{equation*}
		\df [v]=\sum_{D \textnormal{ toric divisor}}\big\langle f\big(c_1(X/D)\big);[v]\big\rangle[v_D],
	\end{equation*}
	where $\langle\,;\rangle$ denotes the duality pairing (\ref{eq:nat_duality}).
\end{prop}

\begin{proof}
	We denote by $s:N(c)\rightarrow N$ a $\Z$-linear section of the projection $\pi:N\rightarrow N(c)$ and by the same symbol the corresponding morphism of complex tori $s:T(c)_{\C}\rightarrow T_\C$. We recall that we previously denoted by $\delta:T(c)_{\C}\rightarrow (T_c)_{\C}$ the anti-regular morphism of complex tori given by the expression:
	\begin{equation*}
		\delta(t)=\frac{\tau s(t)}{s\tau_{(c)}(t)}.
	\end{equation*} 
	If we denote by $d:N(c)\rightarrow N_c$ the anti-equivariant morphism given by $d\coloneqq \tau_c s-s\tau_{(c)}$ then the number $\delta(t)^\alpha$ is the complex conjugate of $t^{\,d^*\alpha}$ for all characters $\alpha\in M_c$ and all complex points $t$ of $T{(c)}$. The morphism $d$ induces $\df:H^1(\Z/2;N{(c)})\rightarrow H^2(\Z/2;N_c)$ by construction. The toric variety $X$ is isomorphic to $T(c)_{\C}\times (X_{c})_\C$ endowed with the real structure:
	\begin{equation*}
		(t,x)\longmapsto \big(\tau_{(c)}(t),\delta(t)\cdot\sigma_c(x)\big).
	\end{equation*}
	Let us assume that $c$ has $k$ invariant rays and $l$ pairs of exchanged rays. They yield a canonical isomorphism between $X_c$ and $\A^k_\R\times_\R\Res\A^l_\C$ indexed by maximal toric subvarieties of $X$. Let $D$ be a toric divisor of $X$ and $v$ be an anti-invariant vector in $N{(c)}$. Using the decomposition of the fibre, we find that the reduction modulo $2$ of $D$\textsuperscript{th} component, $m\in\Z$, of $dv$ is precisely the $D$\textsuperscript{th} component of $\df[v]$. The vector $v$ can be seen as a group morphism $v:\SO\rightarrow T{(c)}$. The pull-back $v^*X/D$ is isomorphic to $\G{\C}\times\A^1_\C$ endowed with the real structure:
	\begin{equation*}
		(t;x)\longmapsto \big(1/\bar{t}\,;\,\overline{t^m}\cdot\bar{x}\big).
	\end{equation*}
	As a line bundle over $\SO$, it is trivial if and only if $m$ is even. Thus $c_1(v^*X/D)=m(\mod 2)$. The naturality of both the pairing and the morphism $f$ implies that:
	\begin{equation*}
		\big\langle f\big(c_1(X/D)\big);[v]\big\rangle = \big\langle f\big(c_1(X/D)\big);v_*[1]\big\rangle =\big\langle f\big(v^*c_1(X/D)\big);[1]\big\rangle =	c_1(v^*X/D).
	\end{equation*}
	As a consequence, we find the desired formula. 
\end{proof}

\begin{dfn}
	Let $\mu:\Z^q\rightarrow \Z^p$ be a group morphism. We define the real torus $\G{\R}^p\times_\R^\mu\SO^q$ as the complex torus $\G{\C}^p\times_\C\G{\C}^q$ endowed with the real structure:
	\begin{equation}\label{eq:twist_prod_mu}
		(x;t)\longmapsto\big(\overline{t^\mu}\cdot \bar{x};1/\bar{t}\,\big).
	\end{equation}
	The torus real structure (\ref{eq:twist_prod_mu}) uniquely extends to a toric real structure on the complex \te\;$\G{\C}^p\times_\C\G{\C}^q\hookrightarrow \A^p_\C\times_\C\G{\C}^q$. We denote the resulting real \te\;by $\G{\R}^p\times_\R^\mu\SO^q\hookrightarrow \A^p_\R\times^\mu_\R\SO^q$. By construction it has isogeneous type $(p;q)$.
\end{dfn}

\begin{prop}\label{prop:winding_mu}
	Let $\mu:\Z^q\rightarrow \Z^p$ be a group morphism. The winding number of $\A^p_\R\times^\mu_\R\SO^q$ is the rank of the reduction of $\mu$ modulo $2$.
\end{prop}

\begin{proof}
	Let $N$ denote the cocharacter lattice of the torus of $\A^p_\R\times^\mu_\R\SO^q$. The cohomological long exact sequence associated to $0\rightarrow \Z[1]^p\rightarrow N\rightarrow \Z[-1]^q\rightarrow 0$ contains the following exact sequence:
	\begin{equation*}
		0 \rightarrow H^1(\Z/2;N) \rightarrow \big(\Z/2\big)^q \overset{\mu}{\rightarrow} \big(\Z/2\big)^p.
	\end{equation*}
	So if $r$ is the winding number of $N$, we have $q-r=q-\rk(\mu\otimes\Z/2)$.
\end{proof}

\begin{prop}\label{prop:equiv_ngbhd}
	Let $T\hookrightarrow X$ be a smooth affine real \te\;whose defining cone $c$ is made of $k$ invariant rays and $l$ pairs of exchanged rays. If we further assume that the ground torus $T{(c)}$ is of type $(p;q)_r$ then there is a matrix $\mu:\Z^{q-r}\rightarrow \Z^k$ such that $X$ is isomorphic to:
	\begin{equation}\label{eq:normal_form_affine}
		\Res\A^l_\C\times_\R\big(\A^k_\R\times^\mu_\R\SO^{q-r}\big)\times_\R\G{\R}^{p-r}\times_\R\Res\G{\C}^r.
	\end{equation} 
	It has isogeneous type $(p+k+l;q+l)$ and winding number $r+l+\rk(\mu\otimes\Z/2)$.
\end{prop}

\begin{proof}
	Let $s$ denote both a $\Z$-linear section of $\pi:N\rightarrow N(c)$ and the induced morphism of complex tori. Recall that $\delta$ is given by the expression $\tau s(t)/s(\tau_{(c)} t)$ and takes its values in $(T_c)_{\C}$. With these notations, $X$ is isomorphic to $(X_c)_{\C}\times_\C T(c)_{\C}$ endowed with the real structure:
	\begin{equation*}
		\sigma(x;t)=\big(\delta(t)\cdot\sigma_c(x);\tau_{(c)}(t)\big).
	\end{equation*}
	If $d$ denotes $\tau_cs-s\tau_{(c)}$, the involution $\tau$ of $N$ is given by:
	\begin{equation*}
		\left( \begin{array}{cc} \tau_{(c)} & 0 \\ d & \tau_c \end{array}\right) \in \left( \begin{array}{cc} \End_\Z\big(N(c)\big) & \Hom_\Z\big(N_c;N(c)\big) \\ \Hom_\Z\big(N(c);N_c\big) & \End_\Z(N_c) \end{array}\right)
	\end{equation*}
	If now $e$ is any other anti-invariant element of $\Hom_\Z\big(N(c);N_c\big)$, the element:
	\begin{equation*}
		\left( \begin{array}{cc} \tau_{(c)} & 0 \\ e & \tau_c \end{array}\right)
	\end{equation*}
	defines another real structure of $T_\C$ that can be extended to a real structure of $X_\C$. We note that the two real structures are isomorphic when $d$ and $e$ share the same cohomology class. The coordinates (\ref{eq:shape_Xc_Tc}), provided by the rays of $c$, allows to express $N_c$ as $\Z[\tau]^l\oplus\Z[1]^k$. Let us choose a decomposition $\Z[-1]^{q-r}\oplus\Z[1]^{p-r}\oplus \Z[\tau]^r$ of $N(c)$. In these coordinates, $d$ is decomposed as follows:
	\begin{equation*}
		\left( \begin{array}{ccc} d_{\tau,-1} & d_{\tau,1} & d_{\tau,\tau} \\ d_{1,-1} & d_{1,1} & d_{1,\tau} \end{array}\right) \in \left( \begin{array}{ccc} \Hom\big(\Z[-1]^{q-r};\Z[\tau]^l\big) & \Hom\big(\Z[1]^{p-r};\Z[\tau]^l\big) & \Hom\big(\Z[\tau]^r;\Z[\tau]^l\big) \\ \Hom\big(\Z[-1]^{q-r};\Z[1]^k\big) & \Hom\big(\Z[1]^{p-r};\Z[1]^k\big) & \Hom\big(\Z[\tau]^r;\Z[1]^k\big) \end{array}\right)
	\end{equation*}
	In this expression, every $\Hom$ is meant as homomorphism of Abelian groups. Therefore, the cohomology class of $d$ is:
	\begin{equation*}
		\left( \begin{array}{ccc} 0 & 0 & 0 \\ \,[d_{1,-1}] & 0 & 0 \end{array}\right).
	\end{equation*}
	This a simple consequence of the vanishing of most of the cohomology groups. Thus if $e$ denotes:
	\begin{equation*}
		\left( \begin{array}{ccc} 0 & 0 & 0 \\ d_{1,-1} & 0 & 0 \end{array}\right)
	\end{equation*}
	we have an equivalent real structure of the desired form. Here $\mu$ is simply $d_{1,-1}$. The end of the statement follows from a simple computation involving the description (\ref{eq:normal_form_affine}) and Proposition~\ref{prop:winding_mu}.
\end{proof}

\begin{rem}\label{rem:mu_ngbhd_diff}
	Proposition~\ref{prop:equiv_ngbhd} describes every possible equivariant neighbourhood of an orbit of type $(p;q)_r$. Moreover, in the expression (\ref{eq:normal_form_affine}), the decomposition of the fibre as $\Res\A_\C^l\times_\R\A^k_\R$ is canonically provided by the maximal toric divisors of the affine variety $X$, \textit{cf.} Proposition~\ref{prop:line_dec_affine}. Thus, it yields a well defined injection $\phi:\Z^k\rightarrow H^0(\Z/2;N_c)$ where each factor $\Z$ corresponds to a ray associated with a toric divisor. In practice, to find a morphism $\mu$, one chooses a decomposition $N_{(c)}\cong \Z[-1]^{q-r}\oplus\Z[1]^{p-r}\oplus\Z[\tau]^r$, and any morphism $\mu:\Z^{q-r}\rightarrow \Z^k$ that makes the following digram commutative:
	\begin{equation*}
		\begin{tikzcd}
			\Z^{q-r} \ar[r,"\mu"] \ar[dd] & \Z^k \ar[d,"\phi"] \\
			 & H^0(\Z/2;N_c) \ar[d] \\
			H^1(\Z/2;N_{(c)}) \ar[r,"\df"] & H^2(\Z/2;N_c)
		\end{tikzcd}
	\end{equation*}
\end{rem}

\section{Canonical Fibration and Isogeny}
\subsection{Canonical Fibration}

Let us consider a real torus $T$ of isogeneous type $(p;q)$ and cocharacter lattice $N$. The exact sequence of $\Z[\tau]$-modules $0\rightarrow \ker(1-\tau)\rightarrow N \rightarrow N/\ker(1-\tau)\rightarrow 0$ induces an exact sequence of real tori:
\begin{equation}\label{eq:can_fib_tori}
 1 \longrightarrow \G{\R}^p \longrightarrow T \longrightarrow \SO^q \overset{\pi}{\longrightarrow}1.
\end{equation}
According to Proposition~\ref{prop:local_section_tori}, the sequence (\ref{eq:can_fib_tori}) is an algebraic principal fibre bundle. In this section, we investigate to what extend a similar fibration holds for toric varieties.
\begin{dfn}
	Let $T\hookrightarrow X$ be a real \te. The \emph{canonical fibre} of $X$ is the closure in $X$ of the fibre over $1$ of (\ref{eq:can_fib_tori}).
\end{dfn}
We can remark that the closure of any fibre of (\ref{eq:can_fib_tori}) yields an isomorphic closed subscheme, the isomorphism being provided by the action of an element of $T$.
\begin{prop}
	Let $T\hookrightarrow X$ be a real \te. Its canonical fibre $F$ is an \te\;of the fibre torus of (\ref{eq:can_fib_tori}). If $N$ denotes the cocharacter lattice of $X$, then the cocharacter lattice of $F$ is given by $\ker(1-\tau)\subset N$. Its fan is given by the collection of cones $\{c\cap\ker(1-\tau):c\in C\}$ where $C$ denotes the fan of $X$.  
\end{prop}

\begin{proof}
	By definition, $F$ is endowed with an action of the fibre torus $\G{\R}^p$. Thus, we only need to show that its complexification is a toric variety. We denote by $M$ the character lattice of $X$, and by $\pi:M\rightarrow \Hom\big(\ker(1-\tau);\Z\big)$ the adjoint projection of the inclusion $\ker(1-\tau)\subset N$. Let $c$ be a cone of the fan of $X$. The inclusion $F_\C\cap (U_c)_{\C} \subset (U_c)_{\C}$ is given by the following surjective morphism of complex algebras:
	\begin{equation*}
		\begin{array}{rcl}
			\C[c^+\cap M] & \longrightarrow & \C\big[\pi(c^+)\cap \Hom\big(\ker(1-\tau);\Z\big)\big] \\[5pt]
			\x^\alpha & \longmapsto & \x^{\pi(\alpha)}.
		\end{array}
	\end{equation*}
	One can easily check that $\pi(c^+)$ is $(c\cap\ker(1-\tau))^+$. Let us show that the following collection of cones $\{c\cap\ker(1-\tau):c\in C\}$ forms a fan. Let us assume that $c\cap\ker(1-\tau)$ is a cone of positive dimension. Let $c\cap\ker(1-\tau)\cap\ker(\alpha)$ be a face of $c\cap\ker(1-\tau)$, with $\alpha$ a non-trivial form on $\ker(1-\tau)$ that is non-negative over $c\cap\ker(1-\tau)$. There exists $\beta\in c^+$ such that $\pi(\beta)=\alpha$. To construct $\beta$, we first consider $S$ a supplementary subgroup of $\ker(1-\tau)$ in $N$ such that $c\cap S=\{0\}$. We can find $S$ as a primitive subgroup of $\ker(\gamma)$, where $\gamma\in c^+$ satisfy $\ker(\gamma)\cap c=\{0\}$. The dimensional assumption yields $\ker(1-\tau)+\ker(\gamma)=N$. Since $\ker(1-\tau)$ and $\ker(\gamma)$ are both primitive, we can find a supplementary subgroup of $\ker(1-\tau)$ in $\ker(\gamma)$. Then, $\beta$ is defined by $\beta|_S=0$ and $\beta|_{\ker(1-\tau)}=\alpha$. Therefore, $c\cap\ker(1-\tau)\cap\ker(\alpha)$ equals $c\cap\ker(\beta)\cap\ker(1-\tau)$. Hence, $\{c\cap\ker(1-\tau):c\in C\}$ is a fan. We find that $F$ is a real \te\;with cocharacter lattice $\ker(1-\tau)$ and fan $\{c\cap\ker(1-\tau):c\in C\}$.
\end{proof}

\begin{dfn}\label{dfn:proper_winding}
	Let $T\curvearrowright X$ be a real toric variety and $C$ be its fan. We say that $X$ is \emph{properly wound} if its set of invariant cones $C^\tau$ forms a fan, or equivalently if every invariant cone is point-wise fixed by $\tau$. To avoid redundancy, we say that $X$ is \emph{properly unwound} when it is both properly wound and unwound.
\end{dfn}

\begin{prop}\label{prop:fan_can_fib_prop_wound}
	Let $T\hookrightarrow X$ be a properly wound real \te. The fan of its canonical fibre coincides with the set of invariant cones of the fan of $X$.
\end{prop}

\begin{proof}
	Let us consider a cone $c$ of the fan of $X$. By definition, $c\cap\tau(c)$ belongs to the fan of $X$. It is invariant. Thus, it is contained in $\ker(1-\tau)$ by assumption. Therefore, $c\cap\ker(1-\tau)$ equals $c\cap\tau(c)$.
\end{proof}

\begin{prop}
	An unwound real toric variety $T\curvearrowright X$ that has smooth topological core is properly unwound.
\end{prop}

\begin{proof}
	Let $N$ be the cocharacter lattice of $X$. Since $X$ is unwound, $N$ splits as $\ker(1-\tau)\oplus\ker(1+\tau)$. Let $c$ be an invariant cone of the fan $C$ of $X$. By assumption, it is spanned by a basis $v_1,...,v_k$ of $N_c$ that is permuted by the action of $\tau$. Let us assume there is a pair of exchanged elements, say $v_1,v_2$. In this case, $v_2-v_1$ would be divisible by 2 for $\im(1-\tau)=2\ker(1+\tau)$. As a consequence, $v_1\wedge v_2$ would also be divisible by $2$ and $N/\langle v_1;v_2\rangle$ would have torsion. This would contradict the smoothness hypothesis.
\end{proof}

\begin{prop}
	A properly wound real toric variety $T\curvearrowright X$ has a real point if and only if it is untwisted.
\end{prop}

\begin{proof}
 Following Proposition~\ref{prop:dimension_real_locus}, $X$ has a real point if and only if there is some invariant cone $c$ of its fan for which the twist class of $X$ lies in the image of $H^1(\Z/2;N_c)\rightarrow H^1(\Z/2;N)$. Since $X$ is properly wound, the real structure $\tau$ acts trivially on every subgroup $N_c$ parametrised by an invariant cone $c$. Thus, the cohomology groups $H^1(\Z/2;N_c)$ vanish. It forces $X$ to be untwisted.
\end{proof}

\begin{thm}\label{thm:can_fib}
	Let $T\hookrightarrow X$ be a properly wound real \te, $(p,q)$ be its isogeneous type, $\G{\R}^p\hookrightarrow F$ be its \cf, and $U$ be its \tc. The quotient $U/\G{\R}^p$ is isomorphic to $\SO^q$, and $U$ is a fibre bundle:
	\begin{equation*}
		F \rightarrow U \rightarrow \SO^q,
	\end{equation*}
	with structure group $\G{\R}^p$, and associated principal bundle: $0\rightarrow\G{\R}^p \rightarrow T \rightarrow \SO^q \rightarrow 0.$
\end{thm}

\begin{proof}
	Let us denote by $C$ the fan of $X$. Following Proposition~\ref{prop:quot}, the quotient $U/\G{\R}^p$ is an \te\;of $\SO^q=T/\G{\R}^p$. Its orbital lattice is $(N/\ker(1-\tau);\{0\})$. Indeed, the image of $C^\tau$ in $N/\ker(1-\tau)$ is precisely the fan $\{0\}$. This is ensured by the properness of the winding. Therefore, $\SO^q\hookrightarrow U/\G{\R}^p$ is a principal real \te. Hence, $U/\G{\R}^p$ is isomorphic to $\SO^q$. We can apply Proposition~\ref{prop:local_section_tori} to construct a local section $s:V\rightarrow T$ of the quotient projection $\pi:T\rightarrow \SO^q$. It allows us to construct a local trivialisation of $\pi:U\rightarrow \SO^q$:
	\begin{equation*}
		\begin{array}{rcl}
			V\times_\R F & \longrightarrow & \pi^{-1}(V) \\[5pt]
			(t ; x) & \longmapsto &  s(t)\cdot x.
		\end{array}
	\end{equation*}
	It is an isomorphism for its complexification is invertible. This is stated in \cite[\S2.1]{Fulton:1993aa}. Proposition~\ref{prop:local_section_tori} allows us to choose $V$ in such a way that its translates cover $\SO^q$. This enables us to propagate this local trivialisation into an atlas of local trivialisations. 
\end{proof}

\begin{dfn}
	Let $T\hookrightarrow X$ be a properly wound real \te. We will refer to the fibre bundle $F \rightarrow U \rightarrow \SO^q$ of Theorem~\ref{thm:can_fib} as the \emph{canonical fibration} of $X$.
\end{dfn}

\begin{prop}\label{prop:can_fib_properly_unwound}
	If $T\hookrightarrow X$ is a properly unwound real \te\;then its canonical fibration is trivial.
\end{prop}

\begin{proof}
	It follows naturally from the fact that the projection $T\rightarrow \SO^q$ admits a global section which provides a global trivialisation.
\end{proof}

\begin{cor}\label{cor:path_connected_properly_wound}
	Let $T\hookrightarrow X$ be a properly wound real \te. If its real locus is compact then its it is path connected.
\end{cor}

\begin{proof}
	Following Theorem~\ref{thm:can_fib}, it suffices to show that $F(\R)$ is path connected for $X(\R)$ is a locally trivial fibration in $F(\R)$ over $(\Sph^1)\raisebox{.5em}{$\scriptstyle q$}$. From the hypothesis and Proposition~\ref{prop:compactness_real_locus}, we know that $\ker(1-\tau)$ is contained in the support of $C$ thus in the support of $C^\tau$. Hence, $F(\R)$ is compact, $F$ is even complete. Whenever an equivariant embedding of a split torus has a fixed point, which is certainly the case when it is complete, its real locus is path connected. Indeed, one can join any real point from a point of the open orbit, and, from here, join the fixed point. 
\end{proof}

\subsection{Isogeny and Unwinding}

\begin{dfn}\label{dfn:unwinding}
	Let $T\hookrightarrow X$ be a real \te~with \dl~$(N;C)$. Its \emph{unwinding} is the real \te~$\tilde{T}\hookrightarrow \tilde{X}$  associated to $(\tilde{N};C)$, together with the induced morphism of real \te s $w:\tilde{X}\rightarrow X$. Recall that $\tilde{N}$ is the sublattice $\ker(1-\tau)\oplus\ker(1+\tau)\subset N$.
\end{dfn}

\begin{exs}[Fundamental examples]\label{exs:fundamental_exs_unwinding} As we will see, the following examples are the fundamental pieces to describe the unwindings of smooth affine real toric varieties.
	\begin{enumerate}
		\item[(i)] \emph{The torus isogeny}: Let us consider the real torus whose cocharacter lattice is given by $\Z[\tau]$. The unwinding is spanned by the sublattice $\langle 1+\tau;1-\tau\rangle$. Hence, the unwinding morphism is given in natural coordinates by $u:(x;y)\mapsto(xy;x/y)$. We have the following commutative diagram with exact rows:
		\begin{equation*}
			\begin{tikzcd}
				0 \ar[r]& \big\langle (-1;-1)\ar[d,"\id"left] \big\rangle \ar[r]& (\C^\times)^2 \ar[r,"u"] \ar[d, "(x;y)\mapsto (\bar{x};1/\bar{y})" left] & (\C^\times)^2 \ar[r] \ar[d, "(x;y)\mapsto (\bar{y};\bar{x})"]& 0 \\
				0 \ar[r]& \big\langle (-1;-1) \big\rangle \ar[r]& (\C^\times)^2 \ar[r,"u"]& (\C^\times)^2 \ar[r]& 0
			\end{tikzcd}
		\end{equation*}
		It induces the exact sequence $0\to \langle \pm 1\rangle \to \R^\times\times\Sph^1 \to \C^\times \to 0$ between real tori.
		\item[(ii)] \emph{The unwinding of the Möbius strip}: Let us consider the affine real \te\; $T\hookrightarrow X$ given by the cone spanned by $(1+\tau)$ in the cocharacter lattice $\Z[\tau]$. Its real locus is the Möbius strip:
		\begin{equation*}
			X(\R)=\big\{ (\xi;z)\in\Sph^1\times\C\;|\;\xi\bar{z}= z\big\}.
		\end{equation*}
		The unwinding is then given by the double cover:
		\begin{equation*}
			\begin{array}{rcl}
				\Sph^1\times\R & \longrightarrow & \;X(\R) \\[2pt]
				(\xi;t)\; & \longmapsto & (\xi^2;t\xi).
			\end{array}
		\end{equation*}
		\item[(iii)] \emph{The quadratic cone}: Let us consider the real affine \te\;$T\hookrightarrow X$ with cocharacter lattice $\Z[\tau]$ and cone spanned by $1$ and $\tau$. This is the Weil restriction of the complex affine line. Its real locus is $\C$. Let $Q$ be its unwinding. This is the quadratic cone $\{xy=z^2\}$ endowed with the real structure $(x;y;z)\mapsto (\bar{y};\bar{x};\bar{z})$. The real locus of $Q$ is isomorphic to the real affine surface $\{x^2+y^2=z^2\}$ and the unwinding morphism is the following projection:
		\begin{equation*}
			\begin{array}{rcl}
				\big\{x^2+y^2=z^2\big\} & \longrightarrow & \;\C \\
				(x;y;z)\; & \longmapsto & x+iy.
			\end{array}
		\end{equation*}
	\end{enumerate}
\end{exs}

\begin{prop}
	The unwinding of a real toric variety $X$ is an unwound real toric variety of the same isogeneous type.
\end{prop}

\begin{proof}
	This is a direct consequence of Definitions~\ref{dfn:iso_type_real}~and~\ref{dfn:unwinding}.
\end{proof}

\begin{prop}
	The unwinding of a properly wound real toric variety is properly unwound.
\end{prop}

\begin{proof}
	It follows from the last proposition and Definition~\ref{dfn:proper_winding}.
\end{proof}

\begin{prop}\label{prop:unwind_smooth}
	The unwinding of a real \te\;that has smooth topological core retains this property if and only if it is properly wound.
\end{prop}

\begin{proof}
	Let $T\hookrightarrow X$ be a real \te\;that has smooth topological core. We assume that $X$ is properly wound, and let $c$ be an invariant cone of its fan. Since $X$ is properly wound, $c$ is spanned by invariant primitive vectors $v_1,...,v_k$. These vectors span a direct summand $N_c$ of the cocharacter lattice $N$ of $X$. Thus, $N_c$ is a direct summand of $\ker(1-\tau)$. It ensures that $c$ remains smooth when seen in $\tilde{N}=\ker(1-\tau)\oplus\ker(1+\tau)$. If, on the contrary, we assume that $X$ is improperly wound, its fan possesses a bidimensional cone $c$ spanned by a pair of exchanged rays $v_1,v_2$. The third case of Examples~\ref{exs:fundamental_exs_unwinding} illustrates that $c$ becomes singular in $\tilde{N}$. The semi-group $c\cap \tilde{N}$ is spanned by $ 2v_1$, $v_1+v_2$, and $2v_2$. 
\end{proof}

\begin{prop}\label{prop:unwinding_quot}
	Let $T\hookrightarrow X$ be a real \te. Its unwinding $w:\tilde{X}\rightarrow X$ satisfies the following properties:\begin{enumerate}
	\item[\textnormal{(i)}] $w:\tilde{X}\rightarrow X$ is the geometric quotient of $\tilde{X}$ by $\Gamma_\R$;
	\item[\textnormal{(ii)}] $w:\tilde{X}(\R)\rightarrow X(\R)$ is the topological quotient of $\tilde{X}(\R)$ by $\Gamma$;
	\item[\textnormal{(iii)}] $w$ is \emph{totally real} i.e. the set $\big\{x\in \tilde{X}(\C)~|~w(x)\in X(\R)\big\}$ equals $\tilde{X}(\R)$.
	\end{enumerate}
\end{prop}

\begin{proof}
	The first point is a direct consequence of the definition of the unwinding. The variety $\tilde{X}$ and the group $\Gamma_\R$ satisfy the requirement of Definition~\ref{dfn:quot}. Thus, Proposition~\ref{prop:quot} ensures $w$ is the quotient map. Since $\Gamma_\R$ is finite, $w:\tilde{X}\rightarrow X$ is a separated geometric quotient by \cite[Theorem~5.1]{zbMATH01361100}. Thus, $w$ induces a homeomorphism between $X(\C)$ and the quotient of $\tilde{X}(\C)$ by $\Gamma_\R(\C)$. The three groups $\Gamma_\R(\C)$, $\Gamma_\R(\R)$, and $\Gamma$ are all equal for $\Gamma_\R$ is the constant group associated to $\Gamma$. Hence, $w$ induces a homeomorphism between the quotient of $\tilde{X}(\R)$ by $\Gamma$ and $w(\tilde{X}(\R))\subset X(\R)$. Hence, the second point will be a consequence of the surjectivity of $w:\tilde{X}(\R)\rightarrow X(\R)$. Let $x$ be a real point of $X$. By Proposition~\ref{prop:density_real_loc}, we can find a real 1-parameter subgroup $v\in\ker(1-\tau)\subset N$ in the support of the fan of $X$, and $u\in T(\R)$ such that $t\in\R^\times\mapsto ut^v\in T(\R)$ converges to $x$ in $X(\R)$ as $t$ tends to $0$. By definition of $\tilde{N}$, the $1$-parameter subgroup $v$ belongs to $\tilde{N}$. Moreover, $u$ is the image by $w$ of an element $\tilde{u}$, \textit{cf.} (\ref{eq:torus_can_isog}). Since $X$ and $\tilde{X}$ have the same fan, \cite[\S2.1]{Fulton:1993aa} ensures that $t\in\R^\times\mapsto \tilde{u}t^v\in \tilde{T}(\R)$ converges to a point $\tilde{x}$ in $\tilde{X}(\C)$ as $t$ tends to $0$. Using continuity, we find that $\tilde{x}$ is a real point and that its image by $w$ is $x$. Thus, $w:\tilde{X}(\R)\rightarrow X(\R)$ is surjective. For the last part, we can note that the preimage of $x\in X(\R)$ in $\tilde{X}(\C)$ is a $\Gamma$-orbit for the quotient is geometric. Since $x$ has a preimage in $\tilde{X}(\R)$ and $\Gamma$ acts by real automorphisms, every point of the orbit is real and $w$ is totally real.
\end{proof}

Examples~\ref{exs:fundamental_exs_unwinding} are enough to describe the unwinding of a smooth affine \te. Let $X$ be a smooth affine real \te. Following Proposition~\ref{prop:equiv_ngbhd}, it is of the form:
	\begin{equation*}
		X \cong \Res\G{\C}^r\times_\R\G{\R}^{p-r}\times\Big(\SO^{q-r}\times^\mu_\R\A_\R^k\Big)\times_\R\Res\A_\C^l.
	\end{equation*}
	One can show quite easily that:
	\begin{equation*}
		\tilde{X} \cong \G{\R}^{p}\times_\R\SO^{q}\times_\R \A_\R^k\times_\R Q^l,
	\end{equation*}
	where $Q$ is the quadratic cone $\{x^2+y^2=z^2\}\subset \A^3_\R$. The unwinding replaces the factor $\Res\G{\C}$ by the product $\G{\R}\times\SO$ and trivialises the vector bundle:
	\begin{equation*}
		\SO^{q-r}\times_\R^\mu\A_\R^k\rightarrow \SO^{q-r}.
	\end{equation*}
	However, the summands $\Res\A_\C$ do not behave well under this transformation: they become singular quadratic cones. This phenomenon is a consequence of the unproperness of the winding. We address this problem in the next subsection. We finish this subsection by showing that, when the winding is proper, unwinding the embedding leads to a simpler \te\;that topologically finitly covers the one we started with.
	
\begin{prop}
	Let $T\hookrightarrow X$ be a properly wound real \te. The group $\Gamma_\R$ acts freely on the topological core of the unwinding of $X$.
\end{prop}

\begin{proof}
	Let us write $\tilde{T}=\G{\R}^p\times_\R\SO^q$ according to the splitting $\tilde{N}=\ker(1-\tau)\oplus\ker(1+\tau)$. The properness of the winding reduces the statement to $\Gamma_\R\cap\G{\R}^p=\{1\}$. Since $\Gamma_\R$ is of 2-torsion we have $\Gamma_\R\cap\G{\R}^p=\Gamma_\R\cap\G{\R}^p[2]$. Moreover, (\ref{eq:2tors}) implies that:
	\begin{equation*}
		\tilde{T}[2]=\G{\R}^p[2]\times_\R\SO^q[2]=\big(\ker(1-\tau)\otimes\F_2\big)_\R\times_\R \big(\ker(1+\tau)\otimes\F_2\big)_\R=\left(\smallslant{\tilde{N}}{2\tilde{N}}\right)_\R.
	\end{equation*}
	Further (\ref{eq:emb_gamma}) asserts that the embedding of $\Gamma_\R$ in $\tilde{T}[2]$ is induced by the map $\gamma\mapsto (\df_0\gamma;\df_1\gamma)$ where $\df_0$ and $\df_1$ are both injective. Thus, $\Gamma_\R\cap\G{\R}^p=(\ker(\df_1))_\R=\{1\}$.
\end{proof}

We can finish this section by remarking that the real locus of a properly wound \te\;of type $(p;q)_r$, $T\rightarrow X$ can be seen as a joint mapping torus of the action of $\Gamma$ on the real locus of the canonical fibre:
\begin{equation}\label{eq:joint_map_tor}
	X(\R)\approx F(\R)\times^\Gamma (\Sph^1)^q.
\end{equation}
It is not a joint mapping torus per se as $q$ would need to equal $r$ if it was. However, it is the product of $(\Sph^1)^{q-r}$ and of the mapping torus of the action of $\Gamma$.

\subsection{Resolution of the Winding}

\begin{dfn}
	Let $T\curvearrowright X$ be a real toric variety. A \emph{resolution of its winding} is a subdivision of its fan yielding a properly wound variety $T\curvearrowright X'$. For such a resolution, we have an equivariant birational proper morphism $X'\rightarrow X$.
\end{dfn}

\begin{prop}
	Let $T\curvearrowright X$ be a real toric variety. The barycentric subdivision of the fan of $X$ always yields a resolution of the winding of $X$.
\end{prop}

\begin{proof}
	Let $C$ denote the fan of $X$, and $C'$ denote its barycentric subdivision. One easily checks that $\tau$ permutes the cones of $C'$. A cone of $C'$ is represented by a flag of cones ${c_1<\cdots<c_k}$ of $C$. The image by $\tau$ of such a barycentric cone is represented by the flag ${\tau(c_1)< \cdots <\tau(c_k)}$. Hence, an invariant barycentric cone corresponds to a flag of invariant cones. Therefore, the invariant cones of $C'$ form a fan. 
\end{proof}

\begin{cor}\label{cor:untwisted_compact_connected}
	The real locus of a real \te\;that has compact real locus is path connected.
\end{cor}

\begin{proof}
	Let $T\hookrightarrow X$ be a real \te\;with compact real locus. The real toric variety $X'$ associated with the barycentric subdivision of the fan of $X$ is a properly wound real \te. Its real locus is also compact by Proposition~\ref{prop:compactness_real_locus}. The birational morphism $X'\rightarrow X$ is an isomorphism between the principal orbits. Since $X$ and $X'$ are untwisted, these orbits contain real points and $X'(\R)\rightarrow X(\R)$ has dense image, \textit{cf.} Proposition~\ref{prop:density_real_loc}. Since $X'(\R)$ is compact, $X'(\R)\rightarrow X(\R)$ is surjective. Corollary~\ref{cor:path_connected_properly_wound} ensures that $X'(\R)$ is path connected. Thus, $X(\R)$ is path connected as well.
\end{proof}

The untwistedness is essential as Example~\ref{ex:fakeP1P1} exhibits a complete real toric surface whose real locus consists of two points.

\begin{ex}\label{ex:fakeP1P1}
	Let us consider the real torus $T$ given by the cocharacter lattice $\langle \partial x;\partial y\rangle$ with real structure $\tau\partial x=\partial x$ and $\tau\partial y=-\partial y$. Its first cohomology group is $\Z/2$. We consider the twisted real toric variety $T\curvearrowright X$ whose fan is depicted in Figure~\ref{fig:fakeP1P1}. Its real locus consists of the two real toric fixed points associated with the cones $c_1$ and $c_2$. One can resolve this variety into a real form of $\Proj^1\times\Proj^1$ blown up at its toric fixed points. The real locus of the resolution is empty.
	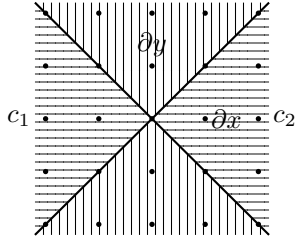
\begin{figure}[!ht]
		\centering
		\begin{tikzpicture}[scale=.7]
			\foreach \x in {-2,...,2}
				\foreach \y in {-2,...,2}
					\fill (\x,\y) circle (.05);
			\draw[thick] (-2.2,-2.2) -- (2.2,2.2);
			\draw[thick] (-2.2,2.2) -- (2.2,-2.2);
			\fill[pattern={Lines[angle=0, yshift=-.6pt , line width=.25pt]}] (0,0)--(2.2,2.2)--(2.2,-2.2);
			\fill[pattern={Lines[angle=0, yshift=-.6pt , line width=.25pt]}] (0,0)--(-2.2,-2.2)--(-2.2,2.2);
			\fill[pattern={Lines[angle=90, yshift=-.6pt , line width=.25pt]}] (0,0)--(2.2,2.2)--(-2.2,2.2);
			\fill[pattern={Lines[angle=90, yshift=-.6pt , line width=.25pt]}] (0,0)--(-2.2,-2.2)--(2.2,-2.2);
			\draw (1.4,0) node{$\partial x$};
			\draw (0,1.4) node{$\partial y$};
			\draw (2.5,0) node{$c_2$};
			\draw (-2.5,0) node{$c_1$};
		\end{tikzpicture}
		\caption{The fan of a complete real surface with disconnected real locus.}
		\label{fig:fakeP1P1}
	\end{figure}
\end{ex}

Whenever $X$ is smooth there is a less expensive way to resolve its winding. We note that Proposition~\ref{prop:unwind_smooth} implies that resolving the winding of such a variety can be seen as resolving the singularities of the unwinding in advance.

\begin{dfn}
	Let $C$ be a fan in the underlying real vector space of a lattice $N$, and $v\in N$ be a non-zero integral vector of the support of $C$. The \emph{stellar subdivision} $C(v)$ of $C$ at $v$ is the collection of rational polyhedral cones $c$ of $N\otimes\R$ such that:\begin{itemize}
	\item[(i)] either $c\in C$ and $v\notin c$;
	\item[(ii)] or $c=d+\R_+v$ where $v\notin d$ and there exists $e\in C$ such that $v\in e$ and $d$ is a face of $e$.
	\end{itemize}
	This is a fan, see \cite[Lemma~11.1.3]{Cox-Lit-Sch_tor_var} for instance. For all vectors $v_1,...,v_k\in N$, we denote the iterated subdivision $C(v_1)\cdots(v_k)$ by $C(v_1;\dots;v_k)$.
\end{dfn}

\begin{dfn}\label{dfn:toric_blow_up}
	Let $T\hookrightarrow X$ be a smooth real \te, $N$ be its cocharacter lattice, and $C$ be its fan. Let $Z$ be a toric subvariety of $X$. We denote by $v_Z$ the sum of the primitive generators of the rays of the cone $c$ associated to $Z$. Since $c$ is invariant, $v_Z$ is an invariant vector and $C(v_Z)$ is stabilised by the action of $\tau$. We denote by $T\hookrightarrow \Bl_ZX$ the smooth real \te\;associated with the orbital lattice $(N;C(v_Z))$. The \emph{toric blow-up} of $X$ along $Z$ designates the \te\;$T\hookrightarrow \Bl_ZX$ together with the morphism of \te s:
	\begin{equation*}
		\pi:\Bl_ZX \rightarrow X,
	\end{equation*}
	induced by $\id_N:(N;C(v_Z))\rightarrow (N;C)$.
\end{dfn}

\begin{prop}\label{prop:tor_blow_up_subvar}
	Let $T\hookrightarrow X$ be a smooth real \te, and $Z$ be a toric subvariety of $X$. The morphism of real varieties $\pi:\Bl_ZX \rightarrow X$ is the blow-up of $X$ along $Z$. Moreover, the exceptional divisor corresponds to the ray spanned by $v_Z$.
\end{prop}

\begin{proof}
	The complexification $\pi_\C:(\Bl_ZX)_\C \rightarrow X_\C$ is the blow-up of $X_\C$ along $Z_\C$. This folkloric fact can be derived using the “\,charts\,” provided by invariant affine open subschemes, see \cite[Definition~3.3.17]{Cox-Lit-Sch_tor_var} or \cite[Proposition~1.26]{Oda:1988aa} for instance. Now, let us denote by $f:X'\rightarrow X$ the blow-up of $X$ along $Z$. We note that $X_\C\rightarrow X$ is flat for $\Spec\,\C\rightarrow\Spec\,\R$ is flat and flatness is preserved by pullbacks. Thus, $f_\C$ is the blow-up of $X_\C$ along $Z_\C$ since blow-ups commute with flat pull-backs. Therefore, by the universal property of blow-ups, there is a unique isomorphism $g$ that makes the following diagram commute:
	\begin{equation*}
		\begin{tikzcd}
			\big(\Bl_ZX\big)_\C \ar[r,"\pi"] & X_\C \\
			X'_\C \ar[u,"g"] \ar[ru,"f_\C" below]
		\end{tikzcd}
	\end{equation*}
	As a consequence, $(\Bl_ZX)_\C$ is endowed with two real structures. A toric one $\sigma$ that comes from Definition~\ref{dfn:toric_blow_up}, and another one $\sigma'$ that is the push-forward of the real structure of $X'_\C$ by $g$. A~priori, $\sigma'$ is not toric. Let us denote by $\varphi\in\Aut_\C(\Bl_ZX)_\C$ the automorphism defined by $\sigma\sigma'$. On the one hand, it reduces to the identity outside of the exceptional locus by the property of the blow-up. Indeed, over this open subscheme, both real structures must be induced by the real structure of $X_\C$. On the other hand, the equaliser of $\varphi$ and the identity must be a closed subscheme for $(\Bl_ZX)_\C$ is reduced and separated. Hence, $\varphi$ and the identity must agree everywhere, and $\sigma$ equals $\sigma'$. The exceptional divisor of $(\Bl_ZX)_\C$ corresponds to the ray spanned by $v_Z$. Since $Z$ corresponds to a real subvariety, the cone is invariant. Thus, the divisor is real, and the exceptional locus of $\Bl_ZX$ corresponds to this divisor with induced real structure.
\end{proof}

\begin{prop}\label{prop:subdiv_commut}
	Let $C$ be a smooth fan in the underlying real vector space of a lattice $N$, and $v_1,v_2\in N$ be two non-zero integral vectors of the support of $C$. For all $i\in\{1;2\}$, we denote by $c_i$ the minimal cone of $C$ that contains $v_i$. If $c_1\cap c_2=0$ then $C(v_1;v_2)$ equals $C(v_2;v_1)$.
\end{prop}

\begin{proof}
	From Definition~\ref{dfn:toric_blow_up}, we find that the cones of $C(v_i)$ are either of the form $c$, where $c\in C$ does not contain $v_i$, or of the form $c+\R_+v$, where $c\in C$ does not contain $v$ but satisfies $c+c_i\in C$. These two kinds of cones of $C(v_i)$ are differentiated by answering the question: \emph{Does $v_i$ belong to $c$?} Thus we understand that in $C(v_1;v_2)$ and $C(v_2;v_1)$ there are four kinds of cones. Let us denote by $c_{12}$ the minimal cone of $C(v_1)$ that contains $v_2$.
	\begin{itemize}
		\item[1\textsuperscript{st} Kind:] The cones of $C(v_1;v_2)$ that do not contain neither $v_1$ nor $v_2$ are of the form $c\in C$ where $c$ does not contain neither of them;
		\item[2\textsuperscript{nd} Kind:] The cones of $C(v_1;v_2)$ that contain $v_1$ but not $v_2$ are of the form $c+\R_+v_1$ where $c\in C$ does not contain neither of them but satisfies $c+c_1\in C$;
		\item[3\textsuperscript{rd} Kind:] The cones of $C(v_1;v_2)$ that contain $v_2$ but not $v_1$ are of the form $c+\R_+v_2$ where $c\in C(v_1)$ does not contain neither $v_1$ nor $v_2$ and satisfies $c+c_{12}\in C(v_1)$. This implies that $c$ belongs to $C$ and satisfies $c+c_2\in C$. Reciprocally, if $c$ does not contain neither $v_1$ nor $v_2$ but satisfies $c+c_2\in C$, then $c+c_{12}\in C(v_1)$. It follows that $c$ is not subdivided from $C$ to $C(v_1)$, and that there is a cone $d$ of $C$ containing both $c$ and $v_2$. In this case, $c$ is a face of $d$ and will be a face of all the maximal pieces of the subdivision of $d$ in $C(v_1)$. One of them $d'$ must contain $v_2$. In this case, $c$ and $c_{12}$ are faces of $d'$. Thus $c+c_{12}$ is also a face of $d'$ and thus belongs to $C(v_1)$;
		\item[4\textsuperscript{th} Kind:] The cones of $C(v_1;v_2)$ that contain both $v_1$ and $v_2$ are of the form $c+\R_+v_1+\R_+v_2$ where $c\in C$ does not contain neither $v_1$ nor $v_2$ and satisfies $c+\R_+v_1+c_{12}\in C(v_1)$ and $c+c_1\in C$. Form the assumption $c_1\cap c_2=0$, we derive that $v_1$ does not belong to $c_{12}$. Moreover, $c+c_{12}\in C(v_1)$ since it is a face of $c+\R_+v_1+c_{12}$. Therefore, $c+c_{2}\in C$. Reciprocally, if $c\in C$ does not contain neither $v_1$ nor $v_2$ and satisfies $c+c_1\in C$ and $c+c_2\in C$ then $c+\R_+v_1+c_{12}\in C(v_1)$. Indeed, if for any $d\in C$ we have $d+c_2\in C$ and $v_1\in d+c_2$ then $c_1$ is a face of $d+c_2$ and thus a face of $d$ by assumption and $v_1\in d$. Hence, we find that $v_1\notin c+c_2$.
	\end{itemize}
	So we find that, under the assumption $c_1\cap c_2=0$, the definitions of the four kinds of cones of $C(v_1;v_2)$ can be made symmetric in $v_1,v_2$. Thus $C(v_1;v_2)=C(v_2;v_1)$ where the second and third kind are exchanged. 
\end{proof}

\begin{dfn}\label{dfn:subdiv_multiple}
	Let $C$ be a smooth fan in the underlying real vector space of a lattice $N$, and $V$ be a collection of non-zero vectors of $N$ contained in the support of $C$. Let us denote by $c_v$ the minimal cone of $C$ containing $v\in V$. If for all distinct $v,v'\in V$, $c_v\cap c_{v'}=0$, we denote by $C(V)$ the subdivision $C(v_1;...;v_k)$ for any enumeration $v_1,...,v_k$ of the elements of $V$. It is well defined by Proposition~\ref{prop:subdiv_commut}.
\end{dfn}

\begin{dfn}
	Let $T\hookrightarrow X$ be a smooth real \te. We denote by $W$ the union of its maximal toric subvarieties\footnote{$Z$ is maximal if the only other toric subvariety that contains it is $X$.} of codimension $2$.  
\end{dfn}

\begin{prop}\label{prop:transverse_toric_subvar}
	Let $Z_1$ and $Z_2$ be two distinct irreducible components of $W$. They are either disjoint or meet transversally. Thus, the collection $V_W\coloneqq \{v_Z:Z \textnormal{ irreducible component of }W\}$ satisfies the requirement of Definition~\ref{dfn:subdiv_multiple}.
\end{prop}

\begin{proof}
	Let $c_1$ and $c_2$ be the respective cones corresponding to $Z_1$ and $Z_2$. By hypothesis, they are bidimensional and spanned by two exchanged rays. Following (\ref{eq:bij_subvar_cone_C}), $Z_1$ meets $Z_2$ if and only if there exists a cone containing both $c_1$ and $c_2$ as faces. It is the case if and only if $c_1+c_2$ is a cone of $C$. Let us assume that they meet, so that $Z_1\cap Z_2$ is associated to the cone $c_1+c_2$ which is invariant. By smoothness, the set of rays of $c_1+c_2$ is the union of the set of rays of $c_1$ and $c_2$. Since $\tau$ permutes these rays and $c_1$ is distinct from $c_2$, there is exactly four rays. Thus, $c_1+c_2$ has dimension 4. Again, this is ensured by smoothness. Following \cite[\S5.1]{Fulton:1993aa}, we have that $(Z_1)_\C$ and $(Z_2)_\C$ meet transversally\footnote{One can easily check the transversality in affine charts.}. Since transversality is a geometric property, the same holds for $Z_1$ and $Z_2$.  
\end{proof}

Prpopsition~\ref{prop:transverse_toric_subvar} states that the set of irreducible components of $W$ meets the requirements of \cite[Definition~2.2]{Li:2009aa}. In that regard, \cite[Theorem~1.3]{Li:2009aa} applies and we have the following proposition.

\begin{prop}\label{prop:tor_blow_up_winding}
	The blow-up of $X$ along $W$ is a smooth variety. Moreover, it is isomorphic to the real equivariant embedding of $T$ associated with the orbital lattice $(N;C(V_W))$. We denote it by $\Bl_WX$ and implicitly see it as an equivariant embedding of $T$. In this representation, the blow-up morphism $\Bl_W X\rightarrow X$ is the morphism of \te s corresponding to the subdivision $\id_N:(N;C(V_W))\rightarrow (N;C)$.
\end{prop}

\begin{proof}
	Let $Z_1,...,Z_k$ be an enumeration of the irreducible components of $W$. They are toric subvarieties of $X$, and hence, by definition, geometrically irreducible. Thus, $(Z_1)_\C,...,(Z_k)_\C$ is an enumeration of the irreducible components of $W_\C$. Proposition~\ref{prop:transverse_toric_subvar} asserts that, in the terminology of L. Li, they form a “\,building set\,”, \textit{cf.} \cite[Definition~2.2]{Li:2009aa}. Therefore, \cite[Theorem~1.3]{Li:2009aa} provides us with two statements:\begin{enumerate}
	\item[(i)] \emph{The blow-up of $X_\C$ along $W_\C$ is smooth};
	\item[(ii)] \emph{The blow-up of $X_\C$ along $W_\C$ is isomorphic to the iterated blow-up of $X_\C$ along the proper transforms of the $(Z_l)_\C$'s.} 
	\end{enumerate}
	Thus, the blow-up $Y_0$ of $X$ along $W$ is smooth for smoothness is a geometric property and blow-ups commute with flat pull-backs. Now let us consider $Y_1$ the equivariant embedding of $T$ associated with $(N;C(V_W))$. It is endowed with a proper morphism of \te s $Y_1\rightarrow X$ associated with the subdivision $\id_N:(N;C(V_W))\rightarrow (N;C)$. Proposition~\ref{prop:tor_blow_up_subvar} and Proposition~\ref{prop:subdiv_commut} assert that it is the iterated blow-up of $X$ along the proper transforms of the $Z_l$'s. From (ii) we derive that $Y_0$ and $Y_1$ are, a priori, two real forms of the blow-up $Y_\C$ of $X_\C$ along $W_\C$. Thus, we have two corresponding real structures $\sigma_0$ and $\sigma_1$ on $Y_\C$ both for which $\pi_\C:Y_\C\rightarrow  X_\C$ is real (i.e. equivariant). We note that $\pi_\C$ is an isomorphism above $T_\C\subset X_\C$ and that the equivariance implies that $\pi^{-1}_\C(T_\C)$ is invariant under $\sigma_0$ and $\sigma_1$. Thus, $\sigma_0\sigma_1$ is a complex automorphism of $Y_\C$ that reduces to the identity over $\pi^{-1}_\C(T_\C)$. Since $Y_\C$ is an integral scheme $\sigma_0\sigma_1=\id_{Y_\C}$, and $\sigma_0=\sigma_1$. The remaining of the proposition follows from this observation.
\end{proof}

\begin{rem}\label{rem:iterative_blow_up} Let $Z_1,...,Z_k$ be an enumeration of the irreducible components of $W$. For all integers $1\leq l\leq k-1$, let us denote by $B_{l+1}\coloneqq \Bl_{Z_l'}B_l$, and by $Z_{l+1}'$ the proper transform of $Z_{l+1}$ in $B_{l+1}$, where $B_1\coloneqq X$ and $Z_1'\coloneqq Z_1$. Proposition~\ref{prop:tor_blow_up_winding} and \cite[Theorem~1.3]{Li:2009aa} imply that:
\begin{equation*}
	\Bl_W X\cong \Bl_{Z_k'}\Bl_{Z_{k-1}'}\cdots\Bl_{Z_1'}X.
\end{equation*}
	Proposition~\ref{prop:tor_blow_up_winding} is essential to garante that the equivariant embedding of $T$ in $\Bl_W X$ does not depend on the particular enumeration of the irreducible components of $W$.
\end{rem}

\begin{prop}\label{prop:blow_up_W}
	Let $T\hookrightarrow X$ be a real \te. The variety $\Bl_W X\rightarrow X$ is a resolution of the winding of $X$. Moreover, $\Bl_W X\rightarrow X$ restricts to an isomorphism of the canonical fibres.
\end{prop}

\begin{proof}
	Let us first prove that $T\hookrightarrow \Bl_W X$ is properly wound. Let $c\in C(V_W)$ be a cone. Extrapolating on the proof of Proposition~\ref{prop:subdiv_commut}, it has the form $d+\langle v_Z:Z\in A\rangle_+$ where $A$ is a subset of $V_W$, and $d$ is a cone of $C$ that does not contain any of the vectors of $V_W$ but is a face of a cone $e\in C$ that contains $\{v_Z:Z\in A\}$. If $c$ is invariant so has to be $d$. Indeed, $\tau$ permutes the rays of $c$, and since the vectors of $V_W$ are invariants it has to permute the rays of $d$ (this is ensured by the smoothness and the fact that $d$ is a face of $c$). Now we claim that every ray of $d$ has to be invariant. If two of them, $\R_+v_1$ and $\R_+v_2$, were exchanged then $\langle v_1;v_2\rangle_+$ would yield a maximal toric subvariety of codimension 2 of $X$. However, by definition of $W$, $v_1+v_2$ would belong to $V_W$. Since we assumed that $d\cap V_W$ is empty this cannot happen. Therefore, we find that if $c$ is invariant then $\tau$ reduces to the identity over it. $\Bl_W X$ is properly wound. According to Proposition~\ref{prop:fan_can_fib_prop_wound}, the remaining of the proposition is proved by showing that every invariant cone of $C(V_W)$ is the intersection of $\ker(1-\tau)$ with a cone of $C$. Let $c=d+\langle v_Z:Z\in A\rangle_+$ be an invariant cone of $C(V_W)$ in the notations previously introduced. We showed that $d$ is fully contained in $\ker(1-\tau)$. For every $Z\in A$ let us denote by $c_Z\in C$ the corresponding bidimensional cone. The cone $c_Z$ is the smallest cone of $C$ that contains $v_Z$. Recall that $d$ is a face of a cone $e\in C$ that contains $\{v_Z:Z\in A\}$. Then $e$ necessarily contains $c_Z$ as a face for every $Z\in A$. Therefore, $d+\sum_{Z\in A}c_Z$ is a face of $e$ and thus a cone of $C$. Now we can notice that $c$ is precisely the intersection of $d+\sum_{Z\in A}c_Z$ and $\ker(1-\tau)$. 
\end{proof}

\section{Cycles and Cohomology}
	In this section, we investigate the cohomology of smooth real \te s with compact real loci. We will rely on the well studied cohomology of smooth complete equivariant embeddings of split tori. Let us start by summarising the cohomological properties of such objects. We consider a smooth \te\;$\G{\R}^p\hookrightarrow F$ with compact real locus. It is complete by Proposition~\ref{prop:compactness_real_locus}. The cycle class map of the complexification is an isomorphism:
	\begin{equation*}
		\cl_{F_\C}:\Ch^*(F_\C) \overset{\cong}{\longrightarrow} H^*(F(\C);\Z).
	\end{equation*} 
	In particular, the cohomology of odd degree of $F(\C)$ vanishes. Furthermore, these rings are torsion free. They have a classical presentation given by a quotient of the Stanley-Reisner ring. The rings are spanned by the classes of toric subvarieties. All these assertions are contained in \cite[Theorem~10.8]{Danilov:1978aa}. The cohomology of the real locus is also totally algebraic, i.e. the cycle class map is onto.
	\begin{equation}\label{eq:alg_coh_split}
		\cl_{F}:\Ch^*(F) \twoheadrightarrow H^*(F(\R);\F_2).
	\end{equation}
	To prove it, one can adapt the arguments of \cite[Proposition~10.4]{Danilov:1978aa} to the real case. It relies on the algebraic cellular decomposition provided by a shelling of the fan in the projective case, and a version of Lemma~\ref{lem:alg_cohomo_blow_up}. The same holds over $\R$. In addition, for all integers $k\geq 0$, the complexification morphism:
	\begin{equation*}
		\begin{array}{rcl} 
			H^k(F(\R);\F_2) & \longrightarrow & H^{2k}(F(\C);\F_2) \\ 
			\cl_F(Z) & \longmapsto & \cl_{F_\C}(Z_\C) \;(\mod 2),
		\end{array}
	\end{equation*}
	is a well defined isomorphism, \textit{cf.} \cite[Proposition~5.14]{Borel:1961aa}. It implies the following lemma.
	\begin{lem}\label{lem:coh_split}
		Let $\G{\R}^p\hookrightarrow F$ be a smooth and complete \te. The cohomology of $F(\R)$ is generated in degree 1. Moreover, we have the following presentation of the first cohomology group of its real locus:
		\begin{equation*}
				0\rightarrow \big\langle {\textstyle \sum_{D}} \alpha(v_D)D : \alpha \in M\otimes \F_2 \big\rangle_{\F_2} \rightarrow \big\langle D : D\textnormal{ toric divisor}\big\rangle_{\F_2} \rightarrow H^1(F(\R);\F_2)\rightarrow 0,
		\end{equation*}
		where the projection sends a divisor $D$ to its class $\cl_F(D)$, and $v_D$ denotes the primitive generator of the ray associated to $D$.
	\end{lem}
\subsection{Betti Numbers}
\begin{dfn}
	Let $X$ be the support of a cellular complex of finite dimension. Its \emph{Poincaré polynomial} is the generating function of its $\F_2$-Betti numbers:
	\begin{equation*}
		b[X]\coloneqq \sum_{k\geq 0} b_k(X)t^k=\sum_{k\geq 0} \dim\,H^k(X;\F_2)\,t^k.
	\end{equation*}
\end{dfn}
\begin{dfn}[Théorème~0.2 in \cite{McCroryClint2003VBno}]
	Let $X$ be a real variety. We denote its virtual Poincaré polynomial by $\beta[X]$. It has the following properties:
	\begin{enumerate}
		\item[(i)] If $Y$ is a closed subvariety of $X$ then $\beta[X]=\beta[X\setminus Y]-\beta[Y]$;
		\item[(ii)] If $Y$ is another real variety then $\beta[X\times_\R Y]=\beta[X]\beta[Y]$;
		\item[(iii)] If $X$ is smooth and have compact real locus then $\beta[X]=b\big[X(\R)\big]$.
	\end{enumerate}
	A key feature of this polynomial is that whenever $X$ has no real points then $\beta[X]$ vanishes.
\end{dfn}

\begin{lem}\label{lem:virt_poinc_torus}
	Let $T$ be a real torus of isogeneous type $(p,q)$. Its virtual Poincaré polynomial is given by:
	\begin{equation*}
		\VP[T]=(t-1)^p(t+1)^q.
	\end{equation*}
\end{lem}

\begin{proof}
	We have the following formul\ae:
	\begin{enumerate}
		\item[(i)] $\VP[\mathbb{G}_{\textnormal{m},\R}]=\VP[\mathbb{P}^1_{\R}\setminus\{0;\infty\}]=t+1-2=t-1$; \hspace{1cm} (ii) $\VP[\SO]=\VP[\Sph^1]=t+1$;
		\item[(iii)] $\VP[\textnormal{Res}_{\C/\R}\,\mathbb{G}_{\textnormal{m},\C}]=\VP\big[\mathbb{A}^2_{\R}\setminus\{x^2+y^2=0\}\big]=\VP\big[\mathbb{A}^2_{\R}\big]-\VP\big[\Spec\,\R\big]-\VP\big[\{x^2+y^2=0;\,x\neq 0\}\big]$.
	\end{enumerate}
	Since $\{x^2+y^2=0;\,x\neq 0\}$ does not have real points we find that $\VP[\textnormal{Res}_{\C/\R}\,\mathbb{G}_{\textnormal{m},\C}]=t^2-1$. If $r$ denotes the winding number of $T$, our torus is isomorphic to the following product:
	\begin{equation*}
		\mathbb{G}_{\textnormal{m},\R}^{p-r}\times_\R \SO^{q-r}\times_\R \textnormal{Res}_{\C/\R}\,\mathbb{G}_{\textnormal{m},\C}^r.
	\end{equation*}
	Hence, its virtual Poincaré polynomial is given by the following formula:
	\begin{equation*}
		\VP[T]=(t-1)^{p-r}(t+1)^{q-r}(t^2-1)^r=(t-1)^p(t+1)^q.
	\end{equation*}
\end{proof}

\begin{dfn} Let $T\hookrightarrow X$ be a real torus embedding and $k,l$ be two non-negative integers. We denote by $e_{k,l}(X)$ the number of real toric orbits of $X$ of isogeneous type $(k;l)$. The \emph{isogeneous type polynomial} of $X$ is defined as follows:
	\begin{equation*}
		e[X]\coloneqq \sum_{k,l\geq 0}e_{k,l}(X)x^ky^l.
	\end{equation*}
	If $(p;q)$ denotes the isogeneous type of $X$ then $e[X]$ has degree $p$ in $x$ and $q$ in $y$.
\end{dfn}

\begin{prop}\label{prop:virt_bett_numb}
	Let $T\hookrightarrow X$ be a real torus embedding. The virtual Poincaré polynomial of $X$ is given by the following formula:
	\begin{equation*}
		\beta[X]=e[X](t-1;t+1).
	\end{equation*}
	Hence, whenever the topological core of $X$ is smooth and have compact real locus, the Poincaré polynomial of $X(\R)$ is given by $b[X(\R)]=e[X](t-1;t+1)$.
\end{prop}

\begin{proof}
	This proposition is a consequence of Lemma~\ref{lem:virt_poinc_torus} and the filtration (\ref{eq:orb_filt_R}) of $X$.
\end{proof}

\begin{dfn}\label{dfn:a_polynomial}
	Let $T\hookrightarrow X$ be a smooth real torus embedding. Let $k,l$ be two non-negative integers, we denote by $a_{k,l}(X)$ the number of real cones of its fan that are made of $k$ real rays and $l$ pairs of exchanged rays. We define $a[X]$ to be the following polynomial:
	\begin{equation*}
		a[X]=\sum_{k,l\geq 0} a_{k,l}(X)x^{k} y^{l}.
	\end{equation*}
	We can remark that $a[X](t,t^2)$ is the generating function of the number of invariant cones of the fan of $X$, and that $a[X](x;x)$ equals $a[F](x;y)$ where $F$ denotes the canonical fibre of $X$.
\end{dfn}

\begin{prop}\label{prop:a_to_e}
	Let $T\hookrightarrow X$ be a smooth real torus embedding of isogeneous type $(p;q)$. We have the following identity:
	\begin{equation*}
		a[X]=x^p\big(\tfrac{y}{x}\big)^qe[X]\big(\tfrac{1}{x};\tfrac{x}{y}\big).
	\end{equation*}
\end{prop}

\begin{proof}
	There is a bijection between the invariant cones of the fan of the torus embedding $T\hookrightarrow X$ and the real toric orbits of $X$. Let $c$ be an invariant cone made of $k$ invariant rays and $l$ pairs of exchanged rays. By definition, the orbit associated with $c$ is of the same isogeneous type as $N(c)$. This isogeneous type is $(p-k-l;q-l)$ for $N_c$ is of isogeneous type $(k+l;l)$. Therefore, we conclude that $a_{k,l}(X)$ equals $e_{p-k-l;q-l}(X)$ for all integers $k,l$, and the identity follows.
\end{proof}

\begin{cor}\label{cor:e_poly_can_fib}
	Let $T\hookrightarrow X$ be a smooth real torus embedding, and $F$ denote the \cf\;of $X$. We have the following identity:
	\begin{equation*}
		e[F](x;y)=e[X](x;1).
	\end{equation*}
\end{cor}

\begin{proof} Using Definition~\ref{dfn:a_polynomial} and Proposition~\ref{prop:a_to_e}, we have the following computation:
	\begin{equation*}
		\begin{aligned}
			e[F](x;y)&= x^pa[F]\left(\tfrac{1}{x};\tfrac{1}{xy}\right)= x^pa[X]\left(\tfrac{1}{x};\tfrac{1}{x}\right)=x^p \tfrac{1}{x^p} e[X]\left(x;1\right).
		\end{aligned}
	\end{equation*}
\end{proof}

\begin{prop}
	Let $T\hookrightarrow X$ be a real equivariant torus embedding of type $(p;q)_r$ with compact real locus. The total virtual Betti number of $X$ is at least $2^{q-r}(p+1)$.
\end{prop}

\begin{proof}
	The total virtual Betti number of $X$ is given by:
	\begin{equation*}
			\beta_*(X) = e[X](0;2) = \sum_{l=0}^q e_{0,l}(X)2^l.
	\end{equation*}
	Since $e_{0,l}(X)$ equals $a_{p-(q-l),q-l}(X)$, and that the latter number vanishes as soon as $q-l$ is bigger than $r$, and we have:
	\begin{equation*}
		\beta_*(X) = \sum_{l=q-r}^q e_{0,l}(X)2^l \; \geq \; 2^{q-r}e[X](0;1)\; = \;2^{q-r}\beta_*(F).
	\end{equation*}
	Where $F$ denotes the \cf\;of $X$. It is of isogeneous type $(p;0)$, so its total virtual Betti number equals the number of maximal cones of its fan. Indeed, both $a[F]$ and $e[F]$ do not depend on $y$ and they satisfy $e[F](x)=x^pa[F](1/x)$. Since this fan is complete of dimension $p$, there is a least $(p+1)$ such cones.
\end{proof}

\begin{cor}
	Let $k\geq 1$ be an integer. The $k$-dimensional sphere is the real locus of a smooth and complete real toric variety if and only if $k$ is at most $2$.
\end{cor}

\begin{proof}
	Let $p,q,r$ be three non-negative integers with $r\leq \min(p;q)$. If $2^{q-r}(p+1)$ is at most $2$ then either $p=0$ and $r\leq q\leq r+1$, or $p=1$ and $q=r$. Thus, the condition $r\leq \min(p;q)$ only allows the following triple: $(0;0;0)$, $(0;1;0)$, $(1;0;0)$, and $(1;1;1)$. So if the real locus of a smooth and complete toric variety is a sphere it can only be $1$ or $2$ dimensional. The sufficiency comes from $\Proj^1_\R$ and $\textnormal{Res}_{\C/\R}\,\mathbb{P}_{\C}^1$.
\end{proof}

\begin{prop}\label{eq:dehn_sommerville}
	Let $T\hookrightarrow X$ be a real \te\;of type $(p;q)_r$ that have compact real locus. We have:
	\begin{equation}\label{eq:generalisation_dehn_sommerville1}
		\beta[X](0)=e[X](-1;1)=1.
	\end{equation}
	Further, if $X$ has smooth topological core, then:
	\begin{equation}\label{eq:generalisation_dehn_sommerville2}
		(p-r)e_{0,q-r}(X)=2e_{1,q-r}(X).
	\end{equation}
\end{prop}

\begin{proof}
	Let us prove the first formula. We can note that, given Proposition~\ref{prop:a_to_e}, we have:
	\begin{equation*}
		e[X](-1;1)=(-1)^pa[X](-1;1)=(-1)^pa[F](-1;1),
	\end{equation*}
	where $F$ denotes the canonical fibre of $X$. Now $a[F]$ does not depend on the $y$ coordinate. The coefficient of $x^k$ is simply the number of $k$-dimensional cones of the fan $C$ of $F$. Let $B^p$ denote a closed ball centred at the origin in the vector space spanned by the cocharacter lattice of $F$. The fan $C$ induces a cellular decomposition of $B^p$. In this setting, we find that $a[F](-1;1)$ is the relative Euler characteristic $\chi(B^p;\partial B^p)$. This is $1-(1+(-1)^{p-1})=(-1)^p$. Thus (\ref{eq:generalisation_dehn_sommerville1}) holds. Let us now assume that $X$ has smooth topological core. Following Definition~\ref{dfn:a_polynomial} and Proposition~\ref{prop:a_to_e}, the number $e_{0,q-r}(X)$ equals the number of invariant cones $c$ in the fan of $X$ made of $(p-r)$ invariant rays and $r$ pairs of exchanged rays. Likewise, $e_{1,q-r}(X)$ equals the number of invariant cones $d$ in the fan of $X$ made of $(p-r-1)$ invariant rays and $r$ pairs of exchanged rays. Thus, $d\cap\ker(1-\tau)$ is a cone of codimension $1$ in $C$. Since $X$ has compact real locus, $C$ is complete, and $d\cap\ker(1-\tau)$ is contained in exactly two maximal cones of $C$. Therefore, $d$ is contained in exactly two invariant cones of the fan of $X$. An invariant cone of $X$ that contains $d$ as a face can only be of the kind of $c$. Moreover, every cone of the kind of $c$ contains exactly $(p-r)$ cones of the kind of $d$. The second formula follows this observation. 
\end{proof}

We can remark that (\ref{eq:generalisation_dehn_sommerville1}) and (\ref{eq:generalisation_dehn_sommerville2}) are generalisations of some of the Dehn-Somerville relations. One recovers the classical relations when $X$ is type $(n;0)_0$.

\subsection{Algebraicity of the Cohomology}

\begin{lem}\label{lem:image_cl}
	Let $X$ be a real variety and $U$ be an open neighbourhood of $X(\R)$. The maps $\cl^X$ and $\cl^U$ have the same image. In particular, if $U$ is smooth, the same is true for $\cl_X$ and $\cl_U$.
\end{lem}

\begin{proof}
	This is a simple consequence of the localisation exact sequence of Chow groups and of the functoriality. Since $X\setminus U$ does not have any real point we have the following commutative diagram with exact top row:
	\begin{equation*}
		\begin{tikzcd}
			\Ch_k(X\setminus U) \ar[r,"i_*"]\ar[d,"\cl^{X\setminus U}" left] & \Ch_k(X) \ar[d,"\cl^{X}"]\ar[r,"j^*"] & \Ch_k(U) \ar[d,"\cl^{U}"]\ar[r] & 0 \\
			0 \ar[r] & H^\textnormal{BM}_k(X(\R);\F_2) \ar[r,equal] & H^\textnormal{BM}_k(U(\R);\F_2)
		\end{tikzcd}
	\end{equation*}
	Where $i:X\setminus U \hookrightarrow X$ is the closed immersion and $j:U \hookrightarrow X$ is the open immersion. 
\end{proof}

\begin{lem}\label{lem:surj_deg_1_can_fib}
	Let $T\hookrightarrow X$ be a smooth and properly wound \te. Let $i:F\to X$ denote the embedding of its canonical fibre. If $X(\R)$ is compact then:
	 \begin{equation*}
	 	i^*:H^1\big(X(\R);\F_2\big)\to H^1\big(F(\R);\F_2\big),
	\end{equation*}
	is surjective.
\end{lem}

\begin{proof}
	We have the commutative square:
	\begin{equation*}
		\begin{tikzcd}
			\Ch^1(X) \ar[d,"i^*" left] \ar[r,"\cl_X"] & H^1\big(X(\R);\F_2\big)\ar[d,"i^*" right] \\
			\Ch^1(F) \ar[r,"\cl_F" below] & H^1\big(F(\R);\F_2\big)
		\end{tikzcd}
	\end{equation*}
	Since $F$ is smooth, complete, and has type $(p;0)$ the cohomology of its real locus is spanned by the fundamental classes of its toric subvarieties, i.e. $\cl_F$ is onto, \textit{cf.} (\ref{eq:alg_coh_split}). Let $D$ be a real toric divisor of $F$. It is associated to a real ray of its fan. Now if $D'$ denotes the toric divisor of $X$ associated with the same ray ($X$ is properly wound) then $i^*$ takes $D'$ to $D$. Thus $\cl_F\circ i^*$ is surjective and so is ${i^*:H^1\big(X(\R);\F_2\big)\to H^1\big(F(\R);\F_2\big)}$.
\end{proof}

\begin{prop}\label{prop:cohomo_pe_deg1}
	Let $T\hookrightarrow X$ be a properly wound real \te\;that has smooth topological core and compact real locus. The cohomology of its real locus is generated by its classes of degree 1. 
\end{prop}

\begin{proof}
	Let $F\overset{i}{\rightarrow} U \rightarrow \SO^q$ be the canonical fibration of $X$. The cohomology of $F(\R)$ is generated in degree $1$, \textit{cf.} Lemma~\ref{lem:coh_split}. Thus, Lemma~\ref{lem:surj_deg_1_can_fib} ensures that $i^*$ is surjective in every degree. Thereafter, the Leray-Hirsch Theorem, see \cite[Theorem 4D.1]{Hat_alg_top} for instance, states that the cohomology of $X(\R)=U(\R)$ is a free module over the cohomology of $(\Sph^1)^q$. Since the latter is generated in degree $1$, so is the cohomology of $X(\R)$.
\end{proof}

\begin{rem}\label{rem:leray-hirsh}
	In the proof of Proposition~\ref{prop:cohomo_pe_deg1}, we found that the Leray-Hirsch Theorem applies to the canonical fibration of the real locus $F(\R)\rightarrow X(\R)\rightarrow (\Sph^1)^q$. We denote the fibre embedding, and the projection by $i$ and $\pi$ respectively. Let $s:(\Sph^1)^q\rightarrow X(\R)$ be a section of $\pi$. Such a section can always be constructed using a toric fixed point of the fibre $F(\R)$. It exists for $F$ is a complete equivariant embedding of a split real torus. In particular, the Leray-Hirsch Theorem asserts that $\alpha\in H^1(X(\R);\F_2)$ vanishes if and only if both $i^*\alpha$ and $s^*\alpha$ vanish. We should note that the theorem does not describe the ring structure of the cohomology of $X(\R)$, only the structure of the bigraded algebra associated to the Leray-Serre filtration. For instance, let $X$ be the toric blow-up of $\Proj^1\times\Proj^1$ at every of the four toric fixed points. We endow it with the wound toric real structure of type $(1;1)_1$, $\sigma(x;y)=(\bar{y};\bar{x})$. Its real locus is Klein's Bottle whose cohomology ring is different from that of a product of two circles. In the Klein's Bottle the first Steenrod square does not vanish, in contrast with a product of two circles. Indeed, in surfaces, the first Steenrod square corresponds, for 1-dimensional cohomology classes, to the cup product with the first Wu class, which is itself the first Steifel-Whitney class. Nonetheless, if $X$ is properly wound with compact real locus, then $b[X(\R)]=b[F(\R)](t+1)^q$ which was already implied by the relation $e[X]=e[F]y^q$.
\end{rem}

\begin{prop}\label{prp:alg_cohomo_pe}
	Let $T\hookrightarrow X$ be a smooth properly wound real \te\;that has smooth topological core and compact real locus. The cohomology of its real locus is totally algebraic.
\end{prop}

\begin{proof}
	Let $U$ denote the topological core of $X$. Following Lemma~\ref{lem:image_cl}, we only need to show that $\cl_U$ is onto. By Proposition~\ref{prop:cohomo_pe_deg1}, we can only prove that $\cl_U:\Ch^1(U)\rightarrow H^1(U(\R);\F_2)$ is surjective for $\cl_U:\Ch^*(U)\rightarrow H^*(U(\R);\F_2)$ is a morphism of algebras. Using Theorem~\ref{thm:can_fib} and Corollary~\ref{cor:divisors_of_SO2}, the canonical fibration:
	\begin{equation*}
		\begin{tikzcd}
			F \ar[r,"i"] & U \ar[r,"\pi"] & \SO^q
		\end{tikzcd}
	\end{equation*}
	induces the following commutative diagram with exact bottom row:
	\begin{equation*}
		\begin{tikzcd}
		& \Ch^1(\SO^q) \ar[r,"\pi^*_\Ch"] \ar[d,"\cl_{\SO^q}","\cong" left] & \Ch^1(U) \ar[d,"\cl_U"]  \ar[r,"i^*_\Ch"] & \Ch^1(F) \ar[d,"\cl_F"] \\
		0 \ar[r] & H^1\big(\SO^q(\R);\F_2\big) \ar[r,"\pi^*_H"] & H^1\big(U(\R);\F_2\big) \ar[r,"i^*_H"] & H^1\big(F(\R);\F_2\big) \ar[r] & 0 
		\end{tikzcd}
	\end{equation*}
	We note that $\cl_F$ is onto for $F$ is a complete equivariant embedding of a split torus. The Chow groups of $F$ are spanned by its toric cycles. Hence, the morphism $i^*_\Ch$ is surjective as well. Now let $\alpha$ be a cohomology class of degree 1 on $U(\R)$. We can find a divisor $D$ on $U$ whose class restricts to the same class as $\alpha$ on $F(\R)$. Therefore, $\alpha-\cl_U(D)$ belongs to the image of $\pi^*_H$. Let $D'$ be the corresponding divisor of $\SO^q$. Then $\alpha$ is the class of $D+\pi^*_\Ch D'$.
\end{proof}

\begin{lem}\label{lem:alg_cohomo_blow_up}
	Let $X$ be a smooth and complete real variety, and $Z\subset X$ be a smooth closed subvariety of codimension $r\geq 2$. If the cohomology of $\Bl_Z X$ is totally algebraic then so is the cohomology of $X$.
\end{lem}

\begin{proof}
	\begin{multicols}{2}
	Let $E$ be the exceptional divisor of the blow-up, and $\pi$ be the blow-up morphism. Let $\xi$ denote the first Chern class of the tautological line bundle on $E$ (it is isomorphic to the normal bundle of $E$ in the blow-up).\columnbreak
	\begin{equation*}
		\begin{tikzcd}
			E \ar[r, hookrightarrow, "j"] \ar[d,"\pi" left] & \Bl_Z X \ar[d,"\pi"] \\
			Z \ar[r,hookrightarrow, "i" below ] & X
		\end{tikzcd}
	\end{equation*}
	\end{multicols} 
\noindent We have the following commutative diagram with exact rows:
	\begin{equation}\label{eq:chow_cohomo_blowup}
		\begin{tikzcd}[column sep=small]
			0 \ar[r] & \Ch^{k-r}(Z) \ar[d,"\cl_Z"] \ar[r,"\Phi"] & \Ch^k(X)\times \Ch^{k-1}(E) \ar[d,"\cl_X\times\cl_E"] \ar[r,"\Psi"] & \Ch^k(\Bl_ZX) \ar[d,"\cl_{\Bl_ZX}"] \ar[r] & 0\\
			0 \ar[r] & H^{k-r}\big(Z(\R);\F_2\big) \ar[r,"\varphi"] & H^k\big(X(\R);\F_2\big)\times H^{k-1}\big(E(\R);\F_2\big) \ar[r,"\psi"] & H^k\big(\Bl_ZX(\R);\F_2\big) \ar[r] & 0
		\end{tikzcd}
	\end{equation}
	where $\Phi$, $\Psi$, $\varphi$, and $\psi$ are given by the following formulæ:
	\begin{equation*}
		\left\{\begin{aligned}
			\Phi(Y)&=\big(i_*(Y);-\pi^*(Y)\cdot c_{r-1}(\mathcal{E})\big) \\
			\varphi(\alpha)&=\big(i_!(\alpha);-\pi^*(\alpha)\cup w_{r-1}(\mathcal{E})\big) 
			\end{aligned}\right.
		\quad\textnormal{and}\quad
		\left\{\begin{aligned}
			\Psi(Y_1;Y_2)&= \pi^*(Y_1)+j_*(Y_2) \\
			\psi(\alpha_1;\alpha_2)&=\pi^*(\alpha_1)+j_!(\alpha_2),
		\end{aligned}\right.
	\end{equation*}
	where $\mathcal{E}$ denotes the excess normal bundle. It is defined by the following exact sequence of vector bundles on $E$:
	\begin{equation*}
		0\rightarrow \O_E(-1) \rightarrow \pi^*\mathcal{N}_{Z/X} \rightarrow \mathcal{E} \rightarrow 0. 
	\end{equation*}
	In (\ref{eq:chow_cohomo_blowup}), the exactness of the algebraic sequence is ensured by \cite[Proposition~6.7]{Fulton:1998aa}. The exactness of the topological sequence is the adaptation of \cite[Theorem~7.31]{Voisin:2002aa} to the real case. Hence, we have:
	\begin{equation*}
		(1-\xi)c_*(\mathcal{E})=\pi^*c_*(\mathcal{N}_{Z/X}) \quad\textnormal{which implies that}\quad c_*(\mathcal{E})=\Big(\sum_{k\geq 0}\xi^k \Big)\pi^*c_*(\mathcal{N}_{Z/X}).
	\end{equation*}
	We can recall that $\pi^*$ makes $\Ch^*(E)$ into a $\Ch^*(Z)$-algebra given by:
	\begin{equation*}
		\Ch^*(E)\cong \bigslant{\Ch^*(Z)[\xi]}{\left(\,\displaystyle\sum_{p+q=r} c_p(\mathcal{N}_{Z/X})\xi^q\right)}.
	\end{equation*}
	Likewise, we have:
	\begin{equation*}
		H^*\big(E(\R);\F_2\big)\cong \bigslant{H^*\big(Z(\R);\F_2\big)[\cl_E(\xi)]}{\left(\,\displaystyle\sum_{p+q=r} w_p(\mathcal{N}_{Z/X})\cup\cl_E(\xi)^q\right)}.
	\end{equation*}
	Hence, in these “\,blow-up coordinates\,” we find that:
	\begin{equation*}
		\left\{\begin{aligned}
			c_{r-1}(\mathcal{E})&=\sum_{p+q=r-1} c_p(\mathcal{N}_{Z/X})\xi^q=\xi^{r-1}+(\textit{\footnotesize terms of lower degree in }\xi) \\
			w_{r-1}(\mathcal{E})&=\sum_{p+q=r-1} w_p(\mathcal{N}_{Z/X})\cup\cl_E(\xi)^q=\cl_E(\xi)^{r-1}+(\textit{\footnotesize terms of lower degree in }\cl_E(\xi))
		\end{aligned}\right.
	\end{equation*}
	Moreover, in blow-up coordinates, $\pi_!$ is just a projection:
	\begin{equation*}
		\pi_!\Big(\sum_{\substack{p+q=k \\ q\leq r-1}}\beta_p\cup\cl_E(\xi)^q\Big)=\beta_{k-r+1}.
	\end{equation*}
	Let $\alpha$ be a cohomology class of degree $k$ on $X(\R)$. By hypothesis, we can find a pair of classes $(Y_1;Y_2)$ in $\Ch^k(X)\times\Ch^{k-1}(E)$ such that $\psi(\alpha-\cl_X(Y_1);-\cl_E(Y_2))$ vanishes. By exactness, it means that there is $\beta\in H^{k-r}\big(Z(\R);\F_2\big)$ such that:
	\begin{equation*}
		\left\{\begin{aligned}
			\alpha&=\cl_X(Y_1)+i_!(\beta) \\
			\cl_E(Y_2)&=\beta\cup w_{r-1}(\mathcal{E}).
		\end{aligned}\right.
	\end{equation*}
	The second equation yields $\beta=\pi_!(\cl_E(Y_2))$. Hence, $\alpha=\cl_X( Y_1 +(i\pi)_*(Y_2))$.
\end{proof}
	
\begin{thm}\label{thm:alg_coho}
	Let $T\hookrightarrow X$ be a smooth and complete real \te, the cohomology of its real locus is totally algebraic.
\end{thm}

\begin{proof}
	Following Remark~\ref{rem:iterative_blow_up} we can successively blow-up subvarieties of codimension $2$:
	\begin{equation*}
	X \leftarrow \Bl_{Z_1} X \leftarrow \cdots \leftarrow \Bl_{Z_m}\cdots \Bl_{Z_1}X=\Bl_W X,
	\end{equation*}
	to end up with a properly wound smooth and complete real \te. Following this observation, Lemma~\ref{lem:alg_cohomo_blow_up} and Proposition~\ref{prp:alg_cohomo_pe} conclude the proof.
\end{proof}

\begin{cor}\label{cor:alg_coho}
	Let $T\hookrightarrow X$ be a real \te\;with smooth topological core and compact real locus. Its cohomology is totally algebraic.
\end{cor}

\begin{proof}
	Let $U$ denote the topological core of $X$. Following Lemma~\ref{lem:image_cl}, this is equivalent to the surjectivity of $\cl_U$. Using Proposition~\ref{prop:smooth_equiv_completion}, we can find a smooth toric completion $X'$ of $U$ with identical topological core. Thus, $\cl_U$ is onto if and only if $\cl_{X'}$ is onto. The latter surjectivity is ensured by Theorem~\ref{thm:alg_coho}. 
\end{proof}

\subsection{Divisors}
	Let $T\hookrightarrow X$ be a smooth real or complex \te. We denote by $Z^1_T(X)$ the subgroup of divisors spanned by the prime $T$-invariant divisors of $X$. Moreover, we denote by $Z^1_\textnormal{tor}(X)$ the subgroup spanned by its prime toric divisors. If $X$ is complex, the two groups are the same. Furthermore, if $X$ is real, the group $H^0(\Z/2;Z^1_{T_\C}(X_\C))$ is naturally isomorphic to $Z^1_T(X)$, while $H^2(\Z/2;Z^1_{T_\C}(X_\C))$ is naturally isomorphic to $Z^1_\textnormal{tor}(X)\otimes\F_2$. In both cases, we denote by $\Ch^1_T(X)$ the image of $
Z^1_T(X)$ in $\Ch^1(X)$. When $X$ is real, we denote by $H^1_\textnormal{tor}(X(\R);\F_2)$ the subgroup of $H^1(X(\R);\F_2)$ spanned by the classes of toric divisors of $X$.

\begin{prop}
	Let $T\hookrightarrow X$ be a smooth and complete real \te. We have the following exact sequences:
	\begin{equation*}
	\left\{\begin{tikzcd}[sep=tiny]
		0 \ar[r] &H^0(\Z/2;M) \ar[r] &Z^1_T(X) \ar[r] &\Ch^1_T(X) \ar[r] &0 \\ 
		0 \ar[r] &\Ch^1_T(X) \ar[r] &\Ch^1(X) \ar[r] &H^1(\Z/2;M) \ar[r] & 0.
	\end{tikzcd}\right.
	\end{equation*}
\end{prop}

\begin{proof}
	The groups $\Ch^1_{T_\C}(X_\C)$ and $\Ch^1(X_\C)$ are the same, \textit{cf.} \cite[Proposition~10.3]{Danilov:1978aa}. Furthermore, we have the following exact sequence:
	\begin{equation}\label{seq:chow1_cpx}
		0\rightarrow M \rightarrow Z^1_{T_\C}(X_\C) \rightarrow \Ch^1(X_\C) \rightarrow 0,
	\end{equation}
	where the inclusion sends a character $\alpha\in M$ to $\sum_{D}\alpha(v_D)D$. Recall that $v_D$ denotes the primitive generator of the ray defining the prime toric divisor $D$. This a consequence of the Stanley–Reisner presentation, \textit{cf.} \cite[Theorem~10.8]{Danilov:1978aa}. The sequence (\ref{seq:chow1_cpx}) is $\Z/2$-equivariant. Thus, it implies the following exact sequence:
	\begin{equation}\label{seq:chow1_real}
		0\rightarrow H^0(\Z/2;M) \rightarrow Z^1_{T}(X) \rightarrow H^0\big(\Z/2;\Ch^1(X_\C)\big) \rightarrow H^1(\Z/2;M)\rightarrow 0.
	\end{equation}
	Indeed, as a $\Z[\tau]$-module, $Z^1_{T_\C}(X_\C)$ is isomorphic to a direct sum of $\Z[1]$ and $\Z[\tau]$, and thus has a trivial first cohomology group. Following \cite[Theorem~2.6]{Hamel:2000aa}, $H^0(\Z/2;\Ch^1(X_\C))$ is naturally isomorphic to $\Ch^1(X)$. Therefore, (\ref{seq:chow1_real}) can be split into the two desired short exact sequences. 
\end{proof}

\begin{prop}\label{prop:presentation_H1tor}
	Let $T\hookrightarrow X$ be a real \te\;with compact real locus and smooth topological core. Let $(p;q)$ be its isogeneous type, $r$ be its winding number, and $\Gamma$ be its winding group. We have an exact sequence:
	\begin{equation*}
		\begin{tikzcd}[row sep=tiny]
			0 \ar[r] & \Gamma^\perp \ar[r] & Z^1_\textnormal{tor}(X)\otimes\F_2 \ar[r] & H^1_\textnormal{tor}\big(X(\R);\F_2\big) \ar[r] & 0 \\
			& \alpha \ar[r,mapsto] & \sum_D \alpha(v_D)D
		\end{tikzcd}
	\end{equation*}
	where $\Gamma^\perp$ denotes the subspace of linear forms of $\ker(1-\tau)\otimes\F_2$ whose restriction to $\Gamma$ vanishes, and $v_D$ is the primitive generator of the ray of the divisor $D$. The group $\Gamma$ is embedded in $\ker(1-\tau)\otimes\F_2$ via (\ref{eq:emb_gamma}).
\end{prop}

\begin{proof}
	Let $b:\Bl_W X\rightarrow X$ denote the resolution of the winding of $X$. Let:
	\begin{equation*}
		F\overset{i}{\longrightarrow} U\overset{\pi}{\longrightarrow} \SO^q,
	\end{equation*}
	be the canonical fibration of $\Bl_W X$, and $s:\SO^q\rightarrow U$ be a section of $\pi$ given by a toric fixed point of $F$. We consider a toric divisor $D\in Z^1_\textnormal{tor}(X)\otimes\F_2$. Since a prime toric divisor can only meet the irreducible components of $W$ transversally, the pull-back $b^*D$ coincide with its strict transform. We note that the morphisms $b^*:H^1(X(\R);\F_2)\rightarrow H^1(\Bl_W X(\R);\F_2)$ and $b^*:Z^1_\textnormal{tor}(X)\otimes\F_2\rightarrow Z^1_\textnormal{tor}(\Bl_WX)\otimes\F_2$ are injective. Thus, the class of $D$ vanishes if and only if the class of $b^*D$ vanishes. Following Remark~\ref{rem:leray-hirsh}, the class of $b^*D$ vanishes if and only if both its pull-backs by $i$ and $s$ vanishes. Let $D_1$,...,$D_k$ denote the prime toric divisors of $X$, and $v_1,...,v_k$ denote the primitive generators of the rays of the fan of $X$. A prime toric divisor of $F$ is either the restriction of one of the $b^*D_i$'s or the restriction of an irreducible component $E_1,...,E_l$ of the exceptional divisor of $b$. Let us denote the generator of their associated rays by $u_1,...,u_l$. Following Lemma~\ref{lem:coh_split}, $\cl(b^*D)|_{F(\R)}$ vanishes if and only there is a linear function $\alpha\in \ker(1-\tau)\rightarrow \F_2$ such that:
	\begin{equation*}
		b^*D|_F=\sum_i \alpha(v_i)b^*D_i|_F + \sum_j\alpha(u_j)E_j|_F.
	\end{equation*}
	The restriction $Z^1_\textnormal{tor}(\Bl_WX)\otimes\F_2\rightarrow Z^1_\textnormal{tor}(F)\otimes\F_2$ is an isomorphism, thus:
	\begin{equation*}
		b^*D=\sum_i \alpha(v_i)b^*D_i + \sum_j\alpha(u_j)E_j.
	\end{equation*}
	In particular, $\alpha$ needs to vanish on every vector $u_j$ for this identity to hold. Now that $b^*D$ has this special form we can pull back its class by $s$. Let $V$ be the affine open set of $U$ associated to the fixed point of $F$ used to define $s$. We denote by $c$ the associated cone of the fan of $U$. The section $s$ takes its values in $V$ and $\pi|_V:V\rightarrow \SO^q$ is a vector bundle, \textit{cf.} Proposition~\ref{prop:line_dec_affine}. Hence, for every cohomology class $\beta$ on $U(\R)$, $s^*\beta$ vanishes if and only if $\beta|_{V(\R)}$ vanishes. Let us assume that $v_1,...,v_{k'},u_1,...,u_{l'}$ are the generators of the cone of $V$. We find that:
	\begin{equation*}
		\begin{aligned}
		\cl(b^*D)|_{V(\R)}&=\sum_{i=1}^{k} \alpha(v_i)\cl(b^*D_i)|_{V(\R)} + \sum_{j=1}^{l}\alpha(u_j)\cl(E_j)|_{V(\R)}\\
		&= \sum_{i=1}^{k'} \alpha(v_i)w_1\big(V/b^*D_i(\R)\big) + \sum_{j=1}^{l'}\alpha(u_j)w_1\big(V/E_j(\R)\big).
		\end{aligned}
	\end{equation*}
	Using now Proposition~\ref{prop:relation_of_groups} and its notations, we find that $\cl(b^*D)|_{V(\R)}$ is given, as a group morphism $\pi_1((\Sph^1)^q;1)\rightarrow \F_2$,  by $\alpha\circ\df\circ g$. Thus, it vanishes if and only if $\alpha$ vanishes on the image of $\df$ for $g$ is onto. We can note that this condition is compatible with $\alpha$ vanishing on the vectors $u_j$ for their reduction modulo 2 lie in the image of $\df$. By definition, $\df$ is the connecting morphism of the group cohomology long exact sequence of $0\rightarrow N_c \rightarrow N \rightarrow N(c) \rightarrow 0$. Since $c$ is the cone associated to a fixed point of $F$, $N_c$ is the group $\ker(1-\tau)$. Therefore, $\df$ is given by:
	\begin{equation*}
		\begin{array}{rcl}
			\big(N/\ker(1-\tau)\big)\otimes\F_2 & \longrightarrow & \ker(1-\tau)\otimes\F_2 \\ 
			~[v] & \longmapsto & [v+\tau(v)]. 
		\end{array}
	\end{equation*}
	Since $\ker(1-\tau)+2N$ is $\ker(1-\tau)+\ker(1+\tau)$, the image of $\df$ is precisely the image of $\Gamma$ in $\ker(1-\tau)\otimes\F_2$ by (\ref{eq:emb_gamma}).
\end{proof}

\begin{prop}\label{prop:non_tor_classes}
	Let $T\rightarrow X$ be a real \te\;of type $(p;q)_r$ with compact real locus and smooth topological core. The subgroup $H^1_\textnormal{tor}(X(\R);\F_2)$ is of codimension $(q-r)$ in $H^1(X(\R);\F_2)$. 
\end{prop}

\begin{proof}
	Let us denote the codimension of $H^1_\textnormal{tor}(X(\R);\F_2)$ in $H^1(X(\R);\F_2)$ by $k$. We denote the canonical fibre of $X$ by $F$, and the resolution of its winding by $\Bl_WX$. Proposition~\ref{prop:presentation_H1tor} implies that:
	\begin{equation*}
		\dim\,H^1_\textnormal{tor}(X(\R);\F_2)=a_{1,0}(X)-p+r.
	\end{equation*}
	Furthermore, Lemma~\ref{lem:coh_split} computes the first Betti number of $F$:
	\begin{equation*}
		b_1\big(F(\R)\big)=a_{1,0}(F)-p=a_{1,0}(X)+a_{0,1}(X)-p.
	\end{equation*}
	Then, we have two ways of computing the first Betti number of $\Bl_WX$: by the Leray-Hirsch Theorem, \textit{cf.} Remark~\ref{rem:leray-hirsh}, and by a simple blow-up formula, \textit{cf.} Remark~\ref{rem:iterative_blow_up}. It leads to the following identity:
	\begin{equation*}
		b_1\big(F(\R)\big)+q=b_1\big(X(\R)\big)+a_{0,1}(X),
	\end{equation*}
	since we successively blew-up $a_{0,1}(X)$ subvarieties. Thus, we have:
	\begin{equation*}
		\begin{aligned}
			\big(a_{1,0}(X)+a_{0,1}(X)-p\big)+q&=\Big(k+\big(a_{1,0}(X)-p+r\big)\Big)+a_{0,1}(X)\\
			q-r&=k
		\end{aligned}
	\end{equation*}
\end{proof}

\subsection{Orientability}

\begin{prop}\label{prop:steifel_whitney}
	Let $T\hookrightarrow X$ be a real \te\;with compact real locus and smooth topological core. Its first Steifel-Whitney class is given by the following formula:
	\begin{equation*}
		w_1\big(X(\R)\big)=\sum_{\substack{D\textnormal{ toric}\\\textnormal{divisor}}}\cl_X(D).
	\end{equation*}
\end{prop}

\begin{proof}
	Let us consider a smooth and complete variety $X'$ obtained by Proposition~\ref{prop:smooth_equiv_completion}. We have a birational map $X\rightarrow X'$ inducing an isomorphism between their topological core. Thus, there is a canonical bijection $D\leftrightarrow D'$ between the toric divisors of $X$ and $X'$ that satisfies $\cl_X(D)=\cl_{X'}(D')$. Following \cite[Corollary~11.5]{Danilov:1978aa}, the first Chern class of the tangent bundle of $X'_\C$ is represented by the following algebraic cycle:
	\begin{equation*}
	c_1(X'_\C)=\sum_{\substack{D'_\C\textnormal{ toric}\\\textnormal{divisor}}}D'_\C.
	\end{equation*}
	We note that it is real for if a toric divisor $D'_\C$ is not defined over $\R$ then its conjugate is in the sum. We note that $\Ch^1(X')=H^0(\Z/2;\Ch^1(X_\C))$, \textit{cf.} \cite[Theorem~2.6]{Hamel:2000aa}, and that $c_1(X_\C')$ is given by $\pi^*c_1(X)$ where $\pi$ denotes $X'_\C\rightarrow X'$. Thus, we have:
	\begin{equation*}
	c_1(X')=\sum_{\substack{D'\;T\textnormal{-invariant}\\\textnormal{divisor}}}D'.
	\end{equation*}
	Furthermore, if $D'$ is $T$-invariant but not a toric divisor $\cl_{X'}(D')$ vanishes. Thus, by the proposition in \cite[\S5.18]{Borel:1961aa}, we find that:
	\begin{equation*}
		w_1\big(X(\R)\big)=w_1\big(X'(\R)\big)=\sum_{\substack{D'\textnormal{ toric}\\\textnormal{divisor of }X'}}\cl_{X'}(D')=\sum_{\substack{D\textnormal{ toric}\\\textnormal{divisor of }X}}\cl_{X}(D).
	\end{equation*}
	We should note that the cited proposition, in the form stated by A. Borel and A. Haefliger, only applies to non-singular quasi-projective varieties. However, given \cite{Grothendieck:1958aa}, on which it partially relies, we find the quasi-projectivity assumption to be superfluous. 
\end{proof}

\begin{thm}\label{thm:orientability}
	Let $T\hookrightarrow X$ be a real \te\;with smooth topological core and compact real locus. Its real locus is orientable if and only if there exists a linear map:
	\begin{equation*}
		j:\ker(1-\tau)\otimes\F_2 \rightarrow \F_2,
	\end{equation*}
	that vanishes on $\Gamma$ and whose value is one on every primitive generator of the invariant rays of $X$.
\end{thm}

\begin{proof}
	Following Proposition~\ref{prop:steifel_whitney}, the first Steifel-Whitney class of $X(\R)$ is given by:
	\begin{equation*}
		w_1\big( X(\R)\big)=\sum_{\substack{D\textnormal{ real toric}\\\textnormal{divisor}}}\cl_X(D).
	\end{equation*}
	Proposition~\ref{prop:presentation_H1tor} asserts that it vanishes if and only if there exists $j:\ker(1-\tau)\otimes\F_2\rightarrow \F_2$ such that $j(v_D)=1$ for all toric divisors $D$ and $j$ vanishes on $\Gamma$.
\end{proof}

\section{Topological Types in Low Dimension}

In this last section, we discuss the topological types that can be realised as real loci of a smooth and complete real \te s of small dimension. We note that given Proposition~\ref{prop:smooth_equiv_completion}, this question is equivalent to the same question extended to real \te s with compact real loci and smooth topological core. We will use several tables to summarise our discussions. Their columns will be indexed by isogeneous types, and their rows by winding number. The symbol “\,n.a.” is the abbreviation of “\,not applicable\,”. We will use it to signify that there cannot be a variety of the given type, for instance of type $(2;0)_1$.  

\begin{dfn}\label{dfn:connected_sum}
	Let $M,N$ be two connected differentiable manifolds of equal dimension $n$, we denote by $M+N$ the \emph{connected sum} of $M$ and $N$. Furthermore, we denote by $k\cdot M$ the connected sum of $k$ copies of $M$ with the convention that $0\cdot M$ is $\Sph^n$, the unit of the connected sum.
\end{dfn}

\begin{rem}
	When we are not restricted to oriented manifolds, there is an ambiguity in the definition of “\,the connected sum\,”. Indeed, there is a priori two different ways to glue the manifolds together along the sphere. For instance, the two different connected sums of $\Proj^2(\C)$ with itself are not homeomorphic. The ambiguity is lifted for oriented manifolds as there is only one way to glue the two manifolds that induces a compatible orientation on the sum. In general, the ambiguity disappears as soon as one of the summands is non-orientable or possesses an orientation reversing automorphism. None of the summands we will use here is orientable without orientation reversing automorphism.
\end{rem}

\subsection{Curves}
\InsertBoxR{0}{\begin{minipage}{0.4\linewidth}\centering
	{\setlength{\extrarowheight}{5pt}
	\begin{tabular}{c||c|c}
		$(p;q)_r$  & (1;0) & (0;1) \\[5pt]
		\hline \hline 
		0 & $\Sph^1$ & $\Sph^1$ ; $\varnothing$
	\end{tabular}}
	\captionsetup{hypcap=false}
	\captionsetup{width=0.8\linewidth}
	\captionof{table}{Topological types of real toric curves.}
	\end{minipage}
}[1]
There is only one complete complex toric curve, $\Proj_\C^1$. It admits three toric real structures: the untwisted type $(1;0)_0$, and the untwisted and twisted types $(0;1)_0$. In both untwisted types, the real locus is a circle. In the case $(1;0)_0$, the real locus of the torus is $\R^\times$ and acts by homographies. In the cases $(0;1)_0$, it is the circle acting by rotations.
\subsection{Surfaces}
{\InsertBoxL{0}{\begin{minipage}{0.45\linewidth}\centering
	{\setlength{\extrarowheight}{5pt}
	\begin{tabular}{c||c}
		I & $(2;0)_0$ \\[5pt] \hline
		II & improperly wound $(1;1)_1$  \\[5pt] \hline
		III & $(1;1)_0$ or properly wound $(1;1)_1$ \\[5pt] \hline
		IV & $(0;2)_0$ 
	\end{tabular}}
	\captionsetup{hypcap=false}
	\captionsetup{width=0.85\linewidth}
	\captionof{table}{The types of bidimensional torus real structures according to the conventions of \cite{delau2004}.}
	\label{tab:dict_Del_2}
\end{minipage}}[5]
	The topological types of all smooth real toric surfaces with compact real locus had previously been determined by C. Delaunay, \textit{cf.} \cite[Theorem~5.4.1]{delau2004}. She distinguishes four types of bidimensional toric real structures: I, II, III, and IV. Our classification has as many types: $(2;0)_0$, $(1;1)_0$, $(1;1)_1$, and $(0;2)_0$. However, the two classifications do not only differ by names. Table~\ref{tab:dict_Del_2} is a dictionary from her vocabulary to ours. We can determine the topological types of all surfaces of type $(2;0)_0$ by using the classification of toric surfaces under the action of split tori, \textit{cf.} \cite[\S2.5, Proposition]{Fulton:1993aa}. The only possibilities are blow-ups of Hirzebruch surfaces, and of the projective plane. As a consequence, the only realisable differentiable surfaces are: the product of two circles, and all non-orientable surfaces. In the two other unwound cases $(1;1)_0$ and $(0;2)_0$, we can only have the empty surface or the product of two circles. This is a simple consequence of Proposition~\ref{prop:can_fib_properly_unwound}, the fact that $\Proj^1_\R$ is the only complete toric curve of type $(1;0)_0$, and that in both cases $H^1(\Z/2;N)$ does not vanish. Only remains the case $(1;1)_1$. Those varieties are necessarily untwisted. Their cocharacter lattice is $\Z[\tau]$. Since we are interested in varieties with compact real loci, the fan of the canonical fibre is necessarily $\{\R_+(1+\tau);\{0\};\R_+(-1-\tau)\}$. The smoothness ensures that each half line is either a cone of the fan of the variety or the bisectrix of a bidimensional cone. Thus, we find three different real loci: $\Sph^2$, $\Proj^2(\R)$, or Klein's bottle. Figure~\ref{fig:111_surf} depicts the three possible sets of invariant cones of equivariant embeddings of $\Res\G{\C}$ with compact real loci. Table~\ref{tab:tpo_type_surfaces} depicts the possible real loci of smooth and complete real toric surfaces.}
\begin{figure}[H]
	\centering
	\begin{subfigure}[t]{0.3\textwidth}
		\centering
		\begin{tikzpicture}[scale=1.5]
			\foreach \x in {-2,...,2}
				\foreach \y in {-2,...,2}
					\fill (\x/2,\y/2) circle (.025);
			\draw[very thick] (-1,0) -- (1,0);
			\draw[very thick] (0,-1) -- (0,1);
			\fill[pattern={Lines[angle=-45, yshift=4pt , line width=.25pt]}, thick] (0,0) -- (1,0) -- (1,1) -- (0,1);
			\fill[pattern={Lines[angle=-45, yshift=4pt , line width=.25pt]}, thick] (0,0) -- (-1,0) -- (-1,-1) -- (0,-1);
			\fill (0,0) circle (.05);
		\end{tikzpicture}
		\caption{The Sphere.}
		\label{subfig:111_surf_S2}
	\end{subfigure}
	\begin{subfigure}[t]{0.3\textwidth}
		\centering
		\begin{tikzpicture}[scale=1.5]
			\foreach \x in {-2,...,2}
				\foreach \y in {-2,...,2}
					\fill (\x/2,\y/2) circle (.025);
			\draw (-1,0) -- (0,0) -- (0,-1);
			\draw[very thick] (1,0) -- (0,0) -- (0,1);
			\fill[pattern={Lines[angle=-45, yshift=4pt , line width=.25pt]}, thick] (0,0) -- (1,0) -- (1,1) -- (0,1);
			\fill (0,0) circle (.05);
			\draw[very thick] (-1,-1) -- (0,0);
		\end{tikzpicture}
		\caption{The Projective Plane.}
		\label{subfig:111_surf_P2}
	\end{subfigure}
	\begin{subfigure}[t]{0.3\textwidth}
		\centering
		\begin{tikzpicture}[scale=1.5]
			\foreach \x in {-2,...,2}
				\foreach \y in {-2,...,2}
					\fill (\x/2,\y/2) circle (.025);
			\draw (-1,0) -- (1,0);
			\draw (0,-1) -- (0,1);
			\fill (0,0) circle (.05);
			\draw[very thick] (-1,-1) -- (1,1);
		\end{tikzpicture}
		\caption{Klein's Bottle.}
		\label{subfig:111_surf_klein}
	\end{subfigure}
	\caption{The Possible Sets of Invariant Cones of Equivariant Embeddings of $\Res\G{\C}$ with Compact Real Locus.}
	\label{fig:111_surf}
\end{figure}
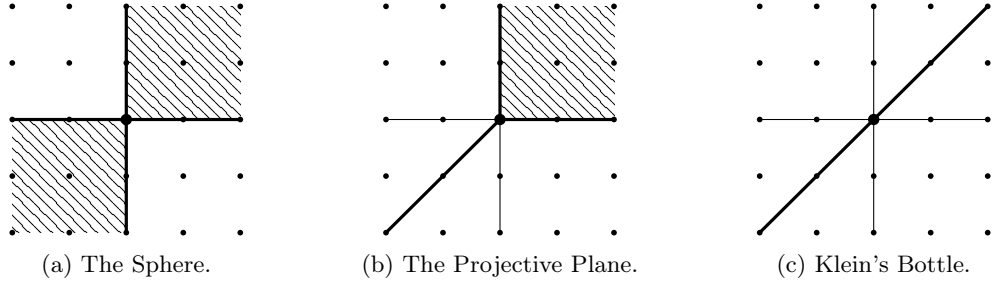

\begin{table}[!ht]
	\centering
	{\setlength{\extrarowheight}{5pt}
	\begin{tabular}{c||c|c|c}
	$(p;q)_r$  & (2;0) & (1;1) & (0;2) \\[5pt]
	\hline \hline 
	0 & $\big(h\cdot \Proj^2(\R)\big)_{h\geq 1}$ ; $\big(\Sph^1\big)^2$ & $\big(\Sph^1\big)^2$ ; $\varnothing$ & $\big(\Sph^1\big)^2$ ; $\varnothing$ \\[5pt]
	\hline
	1  & n.a. & $\Sph^2$ ; $\Proj^2(\R)$ ; $\big(2\cdot\Proj^2(\R)\big)$ & n.a. \\
	\end{tabular}}
	\caption{The topological types of real toric surfaces.}
	\label{tab:tpo_type_surfaces}
\end{table}
\subsection{Threefolds}
There are six types of tridimensional real tori and we sort the real toric threefolds accordingly. As in the case of surfaces, our classification differs from the classification of C. Delaunay not only by names. The signification of her six types is depicted in Table~\ref{tab:top_types_threefolds} (“\,improp. w.” means improperly wound, and “\,prop. w.” means properly wound).
\begin{table}[H]
	\centering
	{\setlength{\extrarowheight}{5pt}
	\begin{tabular}{c|c|c|c|c|c}
	I & II & III & IV & V & VI \\[5pt]
	\hline \hline
	$(3;0)_0$ & \begin{tabular}{c} improp. w. \\ $(2;1)_1$ \end{tabular} & \begin{tabular}{c} prop. w. $(2;1)_1$ \\ and $(2;1)_0$ \end{tabular} & \begin{tabular}{c} improp. w. \\ $(1;2)_1$ \end{tabular} & \begin{tabular}{c} prop. w. $(1;2)_1$ \\ and $(1;2)_0$ \end{tabular} & $(3;0)_0$
	\end{tabular}}
	\caption{The Types of Tridimensional Toric Real Structures According to the Conventions of \cite{delau2004}.}
	\label{tab:top_types_threefolds}
\end{table}
Every threefold can be uniquely decomposed into a connected sum of \emph{prime threefolds}, those which only allow themselves and $\Sph^3$ as connected summand, \textit{cf.} \cite{Hempel:1976aa}. Our goal here will be to provide this \emph{prime decomposition} of all smooth and complete toric threefolds whose types differ from $(3;0)_0$. We will also provide a way to topologically distinguish them. The prime decomposition and the JSJ-decomposition of orientable equivariant embeddings of $\G{\R}^3$ have been described in \cite[Theorem~3.12 and Theorem~4.11]{ERO22}. 

\paragraph{Type (1;2)\textsubscript{1}.}
Determining the topological types of such toric threefolds follows a direct study of the different sets of invariant cones that can occur.

\begin{dfn}[Lens spaces]
	Let $p$ be a non-negative integer and $q$ be an integer coprime to $p$ (if $p$ is null we allow $q$ to be $1$ and if $p$ is $1$ we allow $q$ to be $0$). The \emph{lens space} $L(0;1)$ is the product $\Sph^2\times\Sph^1$. All other \emph{lens spaces} $L(p;q)$ are obtained as quotients of $\Sph^3$, respectively by the free action of $\Z/p$ given by $\xi^k\cdot(x;y)\coloneqq (\xi^k x;\xi^{qk}y)$ where $\xi$ denotes the exponential of $2i\pi/p$, and $\Sph^3$ is endowed with the complex coordinates of $\C^2$. 
\end{dfn}

\begin{prop}\label{prop:topo_121}
	Let $T\hookrightarrow X$ be a real \te\;of type $(1;2)_1$ that has compact real locus and smooth topological core. The real locus of $X$ is homeomorphic to either $\Proj^2(\R)\times\Sph^1$, $(2\cdot\Proj^2(\R))\times\Sph^1$, or a lens space $L(2p;q)$ with $2p$ and $q$ coprime. All these threefolds occur as the real locus of such a variety.
\end{prop}

\begin{proof}
The fibre of the resolution of the winding $\Bl_WX$ is necessarily $\Proj^1_\R$. Hence, following Definition~\ref{dfn:a_polynomial}, $a[X](x;x)=1+2x$. Thus $a[X]$ can either be $(1+2x)$, $(1+x+y)$, or $(1+2y)$. Figure~\ref{fig:invariant_cones_(2;1)1} represents the different aspects of the set of invariant cones for each possible value of $a[X]$.
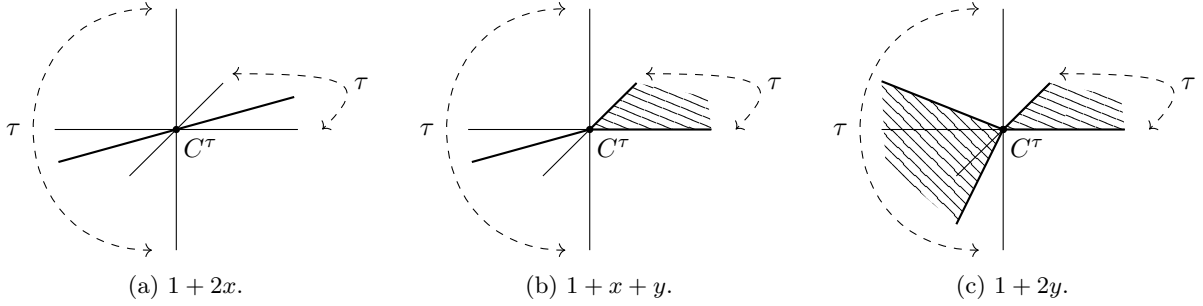
\begin{figure}[H]
	\centering
	\begin{subfigure}[t]{0.32\textwidth}
		\centering
		\begin{tikzpicture}[scale=1.6]
			\draw (0,0,-1) -- (0,0,1);
			\draw (-1,0,0) -- (1,0,0);
			\draw (0,-1,0) -- (0,1,0);
			\draw[thick] (.7,0,-.7) -- (-.7,0,.7);
			\fill (0,0,0) circle (.03);
			\draw[<->,dashed] (-.2,-1,0) .. controls (-1.5,-1,0) and (-1.5,1,0) .. (-.2,1,0);
			\draw[<->,dashed] (0,0,-1.2) .. controls (1,0,-1.2) and (1.2,0,-1) .. (1.2,0,0);
			\draw (1,0,-1) node[right]{$\tau$};
			\draw (-1.2,0,0) node[left]{$\tau$};
			\draw (.2,-.15,0) node{$C^\tau$};
		\end{tikzpicture}
		\caption{$1+2x$.}
	\end{subfigure}
	\hfill
	\begin{subfigure}[t]{0.32\textwidth}
		\centering
		\begin{tikzpicture}[scale=1.6]
			\draw (0,0,-1) -- (0,0,1);
			\draw (-1,0,0) -- (1,0,0);
			\draw (0,-1,0) -- (0,1,0);
			\draw[thick] (-.7,0,.7) -- (0,0,0);
			\draw[thick] (0,0,-1) -- (0,0,0);
			\draw[thick] (1,0,0) -- (0,0,0);
			\fill[pattern={Lines[angle=-22, yshift=4pt , line width=.25pt]}, thick] (0,0,0) -- (0,0,-1) -- (.7,0,-.7) -- (1,0,0) -- cycle;
			\fill (0,0,0) circle (.03);
			\draw[<->,dashed] (-.2,-1,0) .. controls (-1.5,-1,0) and (-1.5,1,0) .. (-.2,1,0);
			\draw[<->,dashed] (0,0,-1.2) .. controls (1,0,-1.2) and (1.2,0,-1) .. (1.2,0,0);
			\draw (1,0,-1) node[right]{$\tau$};
			\draw (-1.2,0,0) node[left]{$\tau$};
			\draw (.2,-.15,0) node{$C^\tau$};
		\end{tikzpicture}
		\caption{$1+x+y$.}
	\end{subfigure}
	\hfill
	\begin{subfigure}[t]{0.32\textwidth}
		\centering
		\begin{tikzpicture}[scale=1.6]
			\draw (0,0,-1) -- (0,0,1);
			\draw (-1,0,0) -- (1,0,0);
			\draw (0,-1,0) -- (0,1,0);
			\draw[thick] (-1,.4,0) -- (0,0,0);
			\draw[thick] (0,-.4,1) -- (0,0,0);
			\fill[pattern={Lines[angle=-45, yshift=4pt , line width=.25pt]}, thick] (0,0,0) -- (-1,.4,0) -- (-.7,0,.7) -- (0,-.4,1) -- cycle;
			\fill (0,0,0) circle (.03);
			\draw[thick] (0,0,-1) -- (0,0,0);
			\draw[thick] (1,0,0) -- (0,0,0);
			\fill[pattern={Lines[angle=-22, yshift=4pt , line width=.25pt]}, thick] (0,0,0) -- (0,0,-1) -- (.7,0,-.7) -- (1,0,0) -- cycle;
			\fill (0,0,0) circle (.03);
			\draw[<->,dashed] (-.2,-1,0) .. controls (-1.5,-1,0) and (-1.5,1,0) .. (-.2,1,0);
			\draw[<->,dashed] (0,0,-1.2) .. controls (1,0,-1.2) and (1.2,0,-1) .. (1.2,0,0);
			\draw (1,0,-1) node[right]{$\tau$};
			\draw (-1.2,0,0) node[left]{$\tau$};
			\draw (.2,-.15,0) node{$C^\tau$};
		\end{tikzpicture}
		\caption{$1+2y$.}
	\end{subfigure}
	\caption{The three different aspects of $C^\tau$ the invariant cones of the fan $C$ of $X$ (the dashed arrows indicate the action of the involution on a basis).}
	\label{fig:invariant_cones_(2;1)1}
\end{figure}
In both $(1+2x)$ and $(1+x+y)$ cases, $C^\tau$ lies in a invariant plane of type $\Z[\tau]$ that admits a supplementary invariant line $\Z[-1]$. In this plane, we recognise the set of invariant cones of Figures~\ref{subfig:111_surf_P2} and \ref{subfig:111_surf_klein}. Hence, if $a[X]$ is $1+2x$ then $X(\R)$ is homeomorphic to $(2\cdot\Proj^2(\R))\times\Sph^1$, and if $a[X]$ is $1+x+y$ then $X(\R)$ is homeomorphic to $\Proj^2(\R)\times\Sph^1$. In the remaining case, $C^\tau$ is made of $\{0\}$ and two bidimensional cones spanned by exchanged pairs of rays: $\langle \partial x;\partial y\rangle_{\R_+}$ and $\langle \partial x';\partial y'\rangle_{\R_+}$. Since $X$ is non-singular, both these cones are spanned by part of a lattice basis. We can add an anti-invariant vector $\partial z$ to $\partial x,\partial y$ to make an equivariant basis of $N$. Let us consider the coordinate integers $p,q_1,q_2$ of $\partial x'$:
\begin{equation*}
	\partial x'=q_1\partial x +q_2\partial y+p\partial z.
\end{equation*}
Since $\langle \partial x;\partial y\rangle_{\R_+}$ and $\langle \partial x';\partial y'\rangle_{\R_+}$ intersect along $0$ and $X$ is smooth, $\partial x'+\partial y'=-(\partial x+\partial y)$. Thus, we find that $q_1+q_2=-1$, and that $q\coloneqq q_1-q_2$ is odd and coprime to $p$, i.e. coprime to $2p$. Now if $u,v$ are two integers satisfying $2up-uq=1$, the vector $\partial z'\coloneqq u(\partial x -\partial y) + v\partial z$ is an anti-invariant vector completing $\partial x',\partial y'$ into a basis of $N$. Using this basis, one finds that $X(\R)$ is obtained as the following gluing of two copies of $\C\times\Sph^1$ along $\C^\times\times\Sph^1$:
\begin{equation*}
	\begin{tikzcd}[row sep=small]
		(z;\zeta) \ar[dd,mapsto] \ar[rr,mapsto] & & \left( \frac{1}{|z|} \zeta^u\left(\frac{z}{|z|} \right)^{-q}; \zeta^v \left(\frac{z}{|z|}\right)^{2p}\right)\\
		& \C^\times\times\Sph^1 \ar[d] \ar[r] & \C\times\Sph^1 \ar[d] \\
		(z;\zeta) & \C\times\Sph^1 \ar[r] & X(\R)
	\end{tikzcd}
\end{equation*}
It is a classical description of the lens space $L(2p;q)$ obtain as the gluing of two solid tori along their boundary, \textit{cf.} \cite[\S4, Example~I]{Brody:1960aa}. We deduce the realisability from the analysis we have just conducted.
\end{proof}

\begin{rem}
	Let $T\hookrightarrow X$ be a real \te\;of type $(1;2)_1$ that has compact real locus and smooth topological core with $a[X]=1+2y$. Then, we know that $X(\R)$ is homeomorphic to the lens space $L(2p;q)$. To find $p$ and $q$ we can notice first that $p$ is the integral length in $\bigwedge^3N$ of the product of three of the primitive generators of the four rays $\R_+\partial x$, $\R_+\partial y$, $\R_+\partial x'$, and $\R_+\partial y'$ of the two invariant bidimensional cones of $X$. Then, we find that $\partial x-\partial y$ and $\partial x' -\partial y'$ span the same module in $N\otimes\Z/2p$ and that $\partial x' -\partial y'$ equals $q(\partial x-\partial y)\,(\mod 2pN)$.
\end{rem}

For the sake of completeness, we recall that two lens spaces $L(p_1;q_1)$ and $L(p_2;q_2)$ are homeomorphic if and only if $p_1=p_2$ and $q_1=\pm q_2^{\pm 1}\;(\mod p_2)$, \textit{cf.} \cite[\S4, Theorem]{Brody:1960aa}.

\paragraph{Type (2;1)\textsubscript{1}.}
In this case, we can no longer proceed by a simple analysis of the fan. Instead, we will remark that the real locus $T(\R)\cong \R^\times\times\C^\times$ contains a unique subgroup isomorphic to the circle. We will use this circle and the classification of threefolds endowed with a circle action, given in \cite{Orlik:1968aa}, to determine when two real \te s of type $(2;1)_1$ have homeomorphic real loci. First, we need to comment on \cite{Orlik:1968aa} to lay out the paragraph. The primary objects of study of this article are triples $M=(|M|;a;[M/\Sph^1])$ where:\begin{enumerate}
\item[(i)] $|M|$ is a closed threefold;
\item[(ii)] $a:\Sph^1\times |M|\rightarrow |M|$ is an effective continuous action;
\item[(iii)] $[M/\Sph^1]$ is a generator of $H_2(|M|/\Sph^1;\partial |M|/\Sph^1;\Z)$, i.e. an orientation of the quotient whenever it is orientable. Otherwise, the group is trivial and $[M/\Sph^1]$ has to be zero.
\end{enumerate}
We will call such a triple an \emph{Orlik-Raymond threefold}. \cite[Theorem~2]{Orlik:1968aa} describes the equivalence classes of such objects under the following relation:
\begin{center}
	\emph{$M\sim N$ if and only if there exists an equivariant homeomorphism $\varphi:|M|\rightarrow |N|$ whose induced homeomorphism $\varphi_{/\Sph^1}:M/\Sph^1\rightarrow N/\Sph^1$ maps $[M/\Sph^1]$ onto $[N/\Sph^1]$.
}\end{center}
To that end, they associate numerical invariants $\big\{b;(\epsilon;g;h;t);(\alpha_1;\beta_1);\dots;(\alpha_n;\beta_n)\big\}$ to every Orlik-Raymond threefold, and show that any two such objects are equivalent if and only if they have the same invariants. Further, \cite[Theorems 3, 4, and 5]{Orlik:1968aa} determine which invariants yield the same underlying threefold $|M|$. To simplify the discourse, we will assume that everything is smooth\footnote{It leads to the same classification, \textit{cf.} \cite[Theorem 2]{Orlik:1968aa}.}. To describe the invariant, we need to say a few words on the structure of the orbits of $M$. The Slice Theorem, \textit{cf.} \cite[Theorem~I.2.1]{Audin:aa}, asserts that every orbit $\Omega$ has an equivariant neighbourhood that is isomorphic to the total space of its normal bundle with induced action. Hence, every orbit fits into one of the categories of the following definition.

\begin{dfn}
Let $\Sph^1\curvearrowright M$ be an action on a threefold and $\Omega$ be an orbit of $M$. We have the following alternative:\begin{enumerate}
\item[(i)] $\Omega$ is an isotropy free orbit;
\item[(ii)] $\Omega$ is a fixed point;
\item[(iii)] $\Omega$ has an equivariant closed neighbourhood that is isomorphic to a model $\textnormal{V}_{\alpha,\beta}$. The model is defined, for all pairs of coprime integers $0<\beta<\alpha$, as the quotient:
\begin{equation*}
	\Vab{\alpha}{\beta}\coloneqq \smallslant{(D^2\times S^1)}{(\Z/\alpha)},\textnormal{ by the action } k\cdot(z;\zeta)\coloneqq \left(e^{\tfrac{2i\pi k}{\alpha}} z;e^{\tfrac{-2i\pi \beta k}{\alpha}}\zeta\right),
\end{equation*}
where $D^2$ is the unit disk in $\C$. The circle acts on $\Vab{\alpha}{\beta}$ by multiplication on the second factor. We say that $\Omega$ is an \emph{exceptional} orbit of unoriented type\footnote{We can remark that $\Vab{\alpha}{\beta}$ is equivariantly isomorphic to $\Vab{\alpha'}{\beta'}$ if and only if $\alpha'=\alpha$ and $\beta'=\pm\beta\,(\mod\alpha)$. However, an equivariant isomorphism between $\Vab{\alpha}{\beta}$ and $\Vab{\alpha}{\alpha-\beta}$ is necessarily orientation reversing relatively to the canonical orientations induced by $D^2\times\Sph^1$.} $(\alpha;\min(\beta;\alpha-\beta))$. If $M$ is oriented in a neighbourhood of $\Omega$, we refer to the specific couple, $(\alpha;\beta)$ or $(\alpha;\alpha-\beta)$, that yields the correct orientation, as the oriented type of $\Omega$;
\item[(iv)] $\Omega$ has an equivariant neighbourhood that is isomorphic to the model $\Lambda$. It is defined as the quotient:
\begin{equation*}
	\Lambda\coloneqq \smallslant{(D^2\times S^1)}{(\Z/2)},\textnormal{ by the action } k\cdot(z;\zeta)\coloneqq \left(\sigma^k(z);(-1)^k\zeta\right).
\end{equation*}
where $\sigma$ is the conjugaison of $\C$. The circle acts on $\Lambda$ by multiplication on the second factor. We say that $\Omega$ is a \emph{special exceptional} orbit.
\end{enumerate}

The set of fixed points of $M$ is denoted by $F$. Let $\textit{Ex}$ denote the union of all exceptional orbits $M$. Both $F$ and $\textit{Ex}$ are finite unions of simple closed curves. We also denote by $\textit{SE}$ the union of all special exceptional orbits. It is a finite union of bidimensional topological tori. We note that the quotient $M/\Sph^1$ is an orbifold surface. It has as many boundary components as $F\cup\textit{SE}$ has connected components. Moreover, it has as many cuspidal points as $M$ has exceptional orbits. 
\end{dfn}

Let $M$ be an Orlik-Raymond threefold. Using the canonical orientation of $\Sph^1$, an orientation of the quotient $|M|/\Sph^1$ can be uniquely lifted as an orientation of $|M|\setminus(F\cup \textit{SE})$. Hence, when the quotient of $M$ is oriented, we always endow $|M|\setminus(F\cup \textit{SE})$ with the induced orientation. 

\begin{dfn}[Orlik-Raymond Invariant] If $\big\{b;(\epsilon;g;h;t);(\alpha_1;\beta_1);\dots;(\alpha_n;\beta_n)\big\}$ is the Orlik-Raymond Invariant\footnote{We use their brackets notations but we must emphasise that the invariant should not be understood as a set but rather as an unordered tuple, i.e. the number of times each $(\alpha_i;\beta_i)$ appears matters.} of $M$, then:
\begin{itemize}
\item[(i)] $\epsilon$ belongs to $\{0;1\}$. It vanishes if and only if the surface $M/\Sph^1$ is orientable;
\item[(ii)] $g$ is the genus of $|M|/\Sph^1$. It is uniquely defined by the following equation:
\begin{equation*}
	b_0(\partial |M|/\Sph^1)+\chi(|M|/\Sph^1)=2-2^{1-\epsilon}g.
\end{equation*}
\item[(iii)] $h$ is the number of connected components of the fixed point set $F$;
\item[(iv)] $t$ is the number of connected components of the special exceptional locus $\textit{SE}$;
\item[(v)] When the quotient is oriented, the unordered tuple $\big\{(\alpha_1;\beta_1);\dots;(\alpha_n;\beta_n)\big\}$ is the unordered tuple of oriented types of exceptional orbits of $M$. In the contrary, $\big\{(\alpha_1;\beta_1);\dots;(\alpha_n;\beta_n)\big\}$ is the unordered tuple of unoriented types of exceptional orbits of $M$;
\item[(vi)] $b$ is an integer. It is rather subtle to define. Since we will not completely need it here, we will just say this: it measures if some canonical sections $D^2/(\Z/\alpha)\setminus\{0\}\rightarrow \Vab{\alpha}{\beta}$ of the quotient map can be extended to the whole quotient $|M|/\Sph^1$ punctured at its cuspidal points. For instance, whenever $\Sph^1\curvearrowright M$ is free, $b$ represents the Euler class of the principal bundle $|M|\rightarrow |M|/\Sph^1$. A useful feature is its vanishing whenever $h+t>0$.
\end{itemize}
\end{dfn}

As announced, the primary aim of the paragraph will be to determine when two real \te s of type $(2;1)_1$ have homeomorphic real loci. To do so, we will first compute the numbers $(\epsilon;g;h;t)$ as well as the unoriented types of the exceptional orbits in terms of invariants of the torus action. It will be enough for our goal. It will also provide the prime decompositions of the real loci. Finally, we will be concerned with realisability. To properly express $(\epsilon;g;h;t)$, we need a refinement of the polynomial $e[X]$.

\begin{dfn}
	Let $T\hookrightarrow X$ be a real \te. For all non-negative integers $p,q,r$ with $r\leq\min(p;q)$, we denote by $e^r_{p,q}(X)$ the number of toric orbits of type $(p,q)_r$. This is also the number of toric subvarieties of $X$ of type $(p,q)_r$. We define the following polynomial:
	\begin{equation*}
		e^*[X]\coloneqq \sum_{p,q,r} e^r_{p,q}(X)x^{p-r}y^{q-r}z^r.
	\end{equation*}
	By definition $e[X](x;y)=e^*[X](x;y;xy)$.
\end{dfn}

\begin{prop}\label{prop:circle_invariant}
	Let $T\hookrightarrow X$ be a real \te\;of type $(2;1)_1$ that has compact real locus and smooth topological core. Under the action of the circle:
	\begin{enumerate}
		\item[\textnormal{(i)}] Every component of the special exceptional surface is the real locus of a toric divisor $D$ of type $(1;1)_0$, i.e. there are $e^0_{1,1}(X)$ such components;
		\item[\textnormal{(ii)}] Every component of the curve of circular fixed points is the real locus of a codimension $2$ toric subvariety of type $(1;0)_0$. There are $e^0_{1,0}(X)$ such subvarieties;
		\item[\textnormal{(iii)}] There are $e^1_{1,1}(X)-e^0_{1,1}(X)-e^0_{1,0}(X)$ exceptional orbits, all of (un)oriented type $(2;1)$;
		\item[\textnormal{(iv)}] The quotient of $X(\R)$ by the circle has genus $0$. It is a sphere with holes, thus orientable.
	\end{enumerate}
\end{prop}

\begin{proof}
	Our first goal is to determine the nature of each circular orbit of $X(\R)$. Since the circle acts through the torus $T(\R)$, every circular orbit of $X(\R)$ is contained in a single toric orbit. Moreover, any two circular orbits contained in a single toric orbit are of the same nature for we can transport an equivariant neighbourhood of the first onto an equivariant neighbourhood of the second by the action of a real point of the torus. Since $X$ has smooth topological core, the isotropy group of a toric orbit is necessarily of the form $\G{\R}^k\times_\R\Res\G{\C}^l$ where $k$ and $l$ are non-negative integers satisfying $k+2l\leq3$. Therefore, we find that the real toric orbits are of the six different types described in Table~\ref{tab:tor-orb}. They are obtained as the possible quotients of $\G{\R}\times_\R\Res\G{\C}$ by $\G{\R}^k\times_\R\Res\G{\C}^l$. We also described the circle action on the real locus of every type of toric orbits as well as the circular isotropy group. We note that the circle action is induced by the action of $\G{\R}\times_\R\Res\G{\C}$ on the quotient.
	\begin{table}[H]
	\centering
	{\setlength{\extrarowheight}{5pt}
	\begin{tabular}{c c c c c}
	Type & Toric Orbit & Real Locus & Circle Action & Circular Isotropy  \\[3pt]
	\hline\hline
	$(2;1)_1$ & $\G{\R}\times_\R\Res\G{\C}$ & $\R^\times\times\C^\times$ & $\xi\cdot(x;z)=(x;\xi z)$ & $1$\\[3pt]
	\hline
	$(1;1)_1$ & $\Res\G{\C}$ & $\C^\times$ & $\xi \cdot z=\xi z$ & $1$\\[3pt]
	\hline
	$(1;1)_0$ & $\SO\times_\R\G{\R}$ & $\Sph^1\times\R^\times$ & $\xi\cdot(\zeta;x)=(\xi^{2}\zeta;x)$ & $\Z/2$ \\[3pt]
	\hline
	$(0;1)_0$ & $\SO$ & $\Sph^1$ & $\xi\cdot\zeta=\xi^{2}\zeta$ & $\Z/2$\\[3pt]
	\hline 
	$(1;0)_0$ & $\G{\R}$ & $\R^\times$ & $\xi\cdot x=x$ &$\Sph^1$\\[3pt]
	\hline
	$(0;0)_0$ & $\Spec\,\R$ & point & $\xi\cdot x=x$ &$\Sph^1$
	\end{tabular}}
	\caption{All possible real toric orbits with induced circle action, the circle variable is represented by $\xi$.}
	\label{tab:tor-orb}
\end{table}
	 We deduce that, besides free orbits and fixed points, we can only have special exceptional orbits, and exceptional orbits of invariant $(2;1)$. Moreover, a dimension argument makes us realise that every toric subvariety of type $(1;1)_0$ yields a component\footnote{Proposition~\ref{cor:untwisted_compact_connected} implies that the real locus of every toric subvariety of $X$ is connected.} of the surface of special exceptional orbits. Following the same idea, we see that every toric subvariety of type $(0;1)_0$ yield an exceptional circular orbit provided that it does not belong to a toric subvariety of type $(1;1)_0$. To be sure, we can use Proposition~\ref{prop:equiv_ngbhd} to study equivariant neighbourhoods of such toric orbits. A toric orbit of type $(1,1)_0$ corresponds to an invariant ray $c$ of the fan of $X$. Let $v$ be its primitive generator. An equivariant neighbourhood is then given by the following model:
	 \begin{equation*}
	 	(\A_\R^1\times^\mu_\R\SO)\times_\R\G{\R},
	 \end{equation*}
	 where $\mu:\Z\rightarrow\Z$ satisfies $\rk(\mu\otimes\Z/2)=1$. Since $(\A_\R^1\times^\mu_\R\SO)$ depends only on $\mu$ modulo $2$, up to isomorphism, we can assume $\mu$ to be $\id_\Z$. Furthermore, Remark~\ref{rem:mu_ngbhd_diff} implies that the class $[v]\in H^2(\Z/2;N)$ has to vanish. Indeed, we have an exact sequence:
	\begin{equation*}
	 	H^1\big(\Z/2:N{(c)}\big)\rightarrow H^2(\Z/2;N_c) \rightarrow H^2(\Z/2;N)\cong\Z/2,
	\end{equation*}
	where the second morphism maps the generator of $H^2(\Z/2;N_c)$ to the class $[v]$. The first map being $\mu\otimes \Z/2$ which has rank $1$, the class of $[v]$ vanishes. We find that the real locus of its neighbourhood is the product of a Möbius band:
	\begin{equation*}
		\R\times^{\id_{\Sph^1}}\Sph^1=\{(z;\zeta)\in\C\times\Sph^1\;|\;\zeta\bar{z}=z\},
	\end{equation*}
	with $\R^\times$. The circle action is provided by the following formula:
	 \begin{equation}
			\Sph^1\curvearrowright\big((\R\times^{\id_{\Sph^1}}\Sph^1)\times\R^\times\big) \quad \xi\cdot(z;\zeta;x)=(\xi z;\xi^{2}\zeta;x).
	 \end{equation}
	 Thus, a typical circular orbit $\omega$ in this toric orbit of type $(1;1)_0$ is given by $\{(0;\zeta;1) :\zeta\in\Sph^1\}$. One of its equivariant circular neighbourhoods is provided by the following equivariant embedding:
	 \begin{equation*}
		\begin{array}{ccl}\Lambda& \longrightarrow & \big(\R\times^{\id_{\Sph^1}}\Sph^1)\times\R^\times \\[5pt]
		~[z;\zeta] & \longmapsto & \left(\frac{\Im(z)\zeta}{1-|z|^2};\zeta^{2};e^{\Re(z)/1-|z|^2}\right).
		\end{array}
	\end{equation*}
	
	\vspace{5pt}
	
	 We note that the vanishing of the cohomology class of the primitive generator of an invariant ray is a way to distinguish toric orbits of type $(1;1)_1$ from toric orbits of $(1;1)_0$. The same analysis (with the same notation $c$ and $v$) for a toric orbit of type $(1,1)_1$ yields the model:
	\begin{equation*}
	 	(\A_\R^1\times^\mu_\R\SO^0)\times_\R\Res\G{\C}=\A^1_\R\times_\R\Res\G{\C},
	 \end{equation*} 
	where $\mu:0\rightarrow \Z$. We have the following exact sequence:
	\begin{equation*}
	 	0=H^1\big(\Z/2:N{(c)}\big)\rightarrow H^2(\Z/2;N_c) \rightarrow H^2(\Z/2;N),
	\end{equation*}
	so $[v]$ does not vanish in this case.
	
	\vspace{5pt}
	
	Let us now assume that we are given a toric orbit of type $(0;1)_0$ associated to an invariant cone $c$ generated by two primitive vectors $v_1$ and $v_2$. This toric orbit is actually made of a single circular orbit. We note that at most one of the classes $[v_1],[v_2]\in H^2(\Z/2;N)$ can vanish for $\{v_1;v_2\}$ is a basis of $\ker(1-\tau)$. This orbit is contained in a toric subvariety of type $(1;1)_0$ if and only if one of the two classes vanishes. In both cases (one or no vanishing), the local model of toric equivariant neighbourhood is provided by:
	\begin{equation*}
	 	\A_\R^2\times^\mu_\R\SO,
	 \end{equation*}
	 where $\mu:\Z\rightarrow \Z^2$ satisfies the exact sequence:
	 \begin{equation*}
	 	\begin{tikzcd}
			0\ar[r] & H^1\big(\Z/2:N{(c)}\big) \ar[r,"\df"] & H^2(\Z/2;N_c) \ar[r] & H^2(\Z/2;N) \\
			& \Z/2 \ar[u,"\cong"] \ar[r,"\mu\otimes\Z/2" below]& (\Z/2)^2 \ar[u,"\cong"]
		\end{tikzcd}
	\end{equation*}
	Thus, if none of the class vanishes, $\mu$ can be taken to be the diagonal $\mu_{1,1}:k\mapsto(k;k)$ and otherwise $\mu_{0,1}:k\mapsto(0;k)$. In the first case, we have an equivariant diffeomorphism:
	\begin{equation*}
		\begin{array}{ccl}V_{2,1}& \longrightarrow & \R^2\times^{\mu_{1,1}}\Sph^1 =\{(z_1;z_2;\zeta)\in\C^2\times\Sph^1\;|\; \zeta \bar{z}_i=z_i,\,\forall i \}\\[5pt]
		~[z;\zeta] & \longmapsto & \left(\tfrac{\zeta\Re(z)}{1-|z|^2};\tfrac{\zeta\Im(z)}{1-|z|^2};\zeta^2\right),
		\end{array}
	\end{equation*}
	so that we have an exceptional orbit of invariant $(2;1)$. In the other case, we have the following equivariant diffeomorphism:
	\begin{equation*}
		\begin{array}{ccl}\Lambda& \longrightarrow & \R^2\times^{\mu_{0,1}}\Sph^1 =\{(x;z;\zeta)\in\R\times\C\times\Sph^1\;|\; \zeta\bar{z}=z\} \\[5pt]
		~[z;\zeta] & \longmapsto & \left(\frac{\Re(z)}{1-|z|^2};\frac{\zeta\Im(z)}{1-|z|^2};\zeta^{2}\right).
		\end{array}
	\end{equation*}
	Hence, we have a special exceptional orbit.
	
	\vspace{5pt}
	
	To summerise the situation we found that:
	\begin{itemize}
		\item[(i)] The number $t$ of components of the surface of special exceptional orbits equals $e_{1,1}^0(X)$, the number of invariant rays such that the class of their primitive generator vanishes;
		\item[(ii)] The number of curves of fixed points equals $e_{1,0}^0(X)$;
		\item[(iii)] Every exceptional orbit has invariant $(2;1)$ and their number $u$ equals the number of cones made of two invariant rays such that none of the classes of their primitive generator vanish.
	\end{itemize}
	We want now to compute $u$. To do so, let us introduce the canonical fibre $F$ of $X$. Since $X$ has compact real locus, $F$ is complete. Recall that its fan is a complete fan of $\ker(1-\tau)$ whose rays are: either invariant rays of the fan $C$ of $X$, or spanned by the sum of the two generators of a bidimensional cone of $C$ whose rays are exchanged by $\tau$. Thus, the class of the primitive generator of a new ray of $F$ vanishes by construction. Let us consider a circular enumeration $(c_k)_{k\in\Z/m}$ of the rays of $F$. It means that there is $m$ rays in the fan of $F$. This number equals the number of maximal toric subvariety of $X$ i.e. $e^1_{1,1}(X)+e^0_{1,1}(X)+e^0_{1,0}(X)$ or more concisely $e_{1,1}(X)+e_{1,0}(X)$. We accordingly denote by $h_k\in\{0;1\}$ the number that vanishes if and only if the class of the primitive generator of the corresponding ray does. As with the fan of $X$, two rays of a bidimensional cone of $F$ cannot both have vanishing $h_k$'s. This means that for all $k\in\Z/m$, the maximum of $h_k$ and $h_{k-1}$ is always $1$. Thus we have:
\begin{equation*}
	\begin{aligned}
		u+m&=\sum_{k\in\Z/m} \min(h_k;h_{k-1})+1 = \sum_{k\in\Z/m} \min(h_k;h_{k-1})+\max(h_k;h_{k-1}) \\
		&=\sum_{k\in\Z/m} h_k+h_{k-1} =\sum_{k\in\Z/m} 2h_k \\
		&=2e^1_{1,1}(X).
	\end{aligned}
\end{equation*}
	Thus, we find that $u=e^1_{1,1}(X)-e^0_{1,1}(X)-e^0_{1,0}(X)$.
	
	\vspace{5pt}
	
To finish the proof of the proposition, let us compute the genus of the quotient $X(\R)/\Sph^1$. Let us consider $\pi:\Bl_W X\rightarrow X$ the resolution of the winding of $X$ introduced in Proposition~\ref{prop:blow_up_W}. We will first establish that the toric morphism $\pi$ induces an homeomorphism between $\Bl_W X(\R)/\Sph^1$ and $X(\R)/\Sph^1$. With all that precedes, we notice that $W(\R)$ is exactly the fixed locus of the circle action on $X(\R)$. Moreover, $\pi$ induces an homeomorphism:
\begin{equation}\label{eq:iso_pi}
	\pi/\Sph^1: \left(\bigslant{\Bl_WX(\R)}{\Sph^1}\right)-\left(\bigslant{E(\R)}{\Sph^1}\right) \overset{\approx}{\longrightarrow} \left(\bigslant{X(\R)}{\Sph^1}\right) - \left(\bigslant{W(\R)}{\Sph^1}\right),
\end{equation}
where $E$ denotes the exceptional divisor. A component of $W$ is a real projective line. Following Proposition~\ref{prop:equiv_ngbhd}, $\pi$ is given, in a neighbourhood of such a component, by the following equivariant model:
	\begin{equation*}
		\begin{array}{rcl}
			\pi:\big(\R\times^{\id_{\Z}}\Sph^1\big)\times\Proj^1(\R) & \longrightarrow & \C\times\Proj^1(\R) \\
			(z;\zeta;p) & \longmapsto & (z;p).
		\end{array}
	\end{equation*}
	It gives rise to the following commutative square:
	\begin{equation*}
		\begin{tikzcd}
			(z;\zeta;p) \ar[d,mapsto]& \big(\R\times^{\id_{\Z}}\Sph^1\big)\times\Proj^1(\R) \ar[d,"\text{quot.}/\Sph^1"] \ar[r,"\pi"]& \C\times\Proj^1(\R) \ar[d,"\text{quot.}/\Sph^1" left] & (z;p) \ar[d,mapsto]\\
			(|z|;p) & \R_+\times\Proj^1(\R) \ar[r,"\pi/\Sph^1" below,equal] & \R_+\times\Proj^1(\R) & (|z|;p)
		\end{tikzcd}
	\end{equation*}
	Hence $\pi/\Sph^1$ is a local homeomorphism between compact spaces, i.e. a finite covering. Since $E(\R)$ is the preimage of $W(\R)$, and that both are invariant under the action of the circle, the homeomorphism (\ref{eq:iso_pi}) ensures that $\pi/\Sph^1$ has degree $1$ and is an homeomorphism. Now, the description of $\Bl_WX(\R)$ as $F(\R)\times^\Gamma\Sph^1$, \textit{cf.} (\ref{eq:joint_map_tor}), asserts that $\Bl_WX(\R)/\Sph^1$ is homeomorphic to $F(\R)/\Gamma$. Thereafter, we find that:
	\begin{equation}\label{eq:riemann-hurwitz}
		2\chi(F(\R)/\Gamma)=\chi\big(F(\R)\big)+\chi\big(F(\R)^\Gamma\big).
	\end{equation}
	See \cite[Corollary~A.1.3]{Deg-Kha_top_pro}. Every bidimensional cone of the fan of $F$ represents a fixed point of the action of $\Gamma$. Likewise, every ray of the fan of $F$ spanned by a primitive generator whose class vanishes in $H^2(\Z/2;N)$ corresponds to a circle entirely made of $\Gamma$-fixed points. Thus, $\chi(F(\R)^\Gamma)$ equals the number of exceptional orbits:
	\begin{equation}\label{eq:eul_char_fixed_loc}
		\chi\big(F(\R)^\Gamma\big)=e^1_{1,1}(X)-e^0_{1,1}(X)-e^0_{1,0}(X)
	\end{equation}
	 Then, by Proposition~\ref{prop:virt_bett_numb} and Corollary~\ref{cor:e_poly_can_fib}, we have:
	\begin{equation*}
		\begin{aligned}
			\chi\big(F(\R)\big)&=e[X](-2;1)\\
			&=4e_{2,1}(X)-2\big(e_{1,1}(X)+e_{1,0}(X)\big)+e_{0,1}(X)+e_{0,0}(X).
		\end{aligned}
	\end{equation*}
	We note that in the fan of $F$ there is $1=e_{2,1}(X)$ point, $e_{1,1}(X)+e_{1,0}(X)$ rays, and $e_{0,1}(X)+e_{0,0}(X)$ bidimensional cones. This a consequence of $F$ being an \te\;of a split torus and of the formula of Proposition~\ref{prop:a_to_e}. Since $F$ is complete, so is its fan and thus there is as many rays as bidimensional cones i.e. $e_{1,1}(X)+e_{1,0}(X)=e_{0,1}(X)+e_{0,0}(X)$, \textit{cf.} (\ref{eq:generalisation_dehn_sommerville2}). Therefore: 
	\begin{equation}\label{eq:eul_char_can_fib}
		\chi\big(F(\R)\big)=4-e_{1,1}(X)-e_{1,0}(X).
	\end{equation}
	Using (\ref{eq:riemann-hurwitz}), (\ref{eq:eul_char_fixed_loc}), and (\ref{eq:eul_char_can_fib}) we find that:
	\begin{equation*}
		\chi(X(\R)/\Sph^1)=2-e^0_{1,1}(X)-e^0_{1,0}(X).
	\end{equation*}
	Since $e^0_{1,1}(X)+e^0_{1,0}(X)$ is the number of boundary components of $X(\R)/\Sph^1$, we find that the genus has to vanish. The quotient is a sphere with holes, hence orientable. We depicted an example in Figure~\ref{fig:quotient_circle}.
\end{proof}
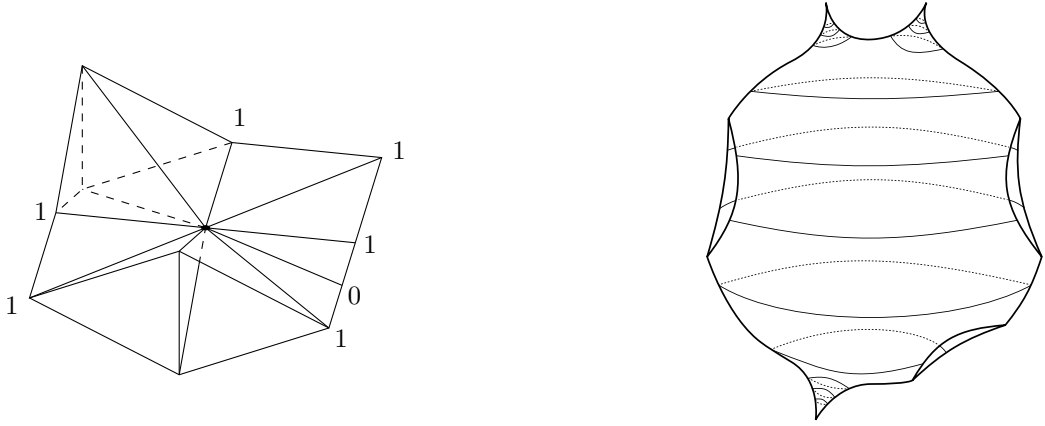
\begin{figure}[!ht]
\centering
\begin{subfigure}[t]{0.45\textwidth}
		\centering
		\begin{tikzpicture}[scale=2,3d view={10}{35}]
			\draw[white] (0,0,0) -- (0,0,-1.7);
			\fill (0,0,0) circle (.03);
			\draw (0,0,0) -- (-1,0,0);
			\draw (0,0,0) -- (1,0,0);
			\draw (0,0,0) -- (0,1,0);
			\draw (0,0,0) -- (0,-1,.5);
			\draw (0,-.25,-.125) -- (0,-1,-.5);
			\draw[dashed] (0,-.25,-.125) -- (0,0,0);
			\draw (0,0,0) -- (-1,1,.5);
			\draw[dashed] (0,0,0) -- (-1,1,-.5);
			\draw (0,0,0) -- (-1,-1,0);
			\draw (0,0,0) -- (1,-1,0);
			\draw (0,0,0) -- (1,1,0);
			\draw (0,0,0) -- (1,-.5,0);
			\draw (1,0,0) -- (1,1,0) -- (0,1,0) -- (-1,1,.5) -- (-1,0,0) -- (-1,-1,0) -- (0,-1,.5) -- (1,-1,0) -- cycle;
			\draw (-1,-1,0) -- (0,-1,-.5) -- (1,-1,0);
			\draw (0,-1,.5) -- (0,-1,-.5);
			\draw[dashed] (0,1,0) -- (-1,1,-.5) -- (-1,0,0);
			\draw[dashed] (-1,1,.5) -- (-1,1,-.5);
			\draw (0,1.3,0) node{$1$};
			\draw (-1.1,0,0) node{$1$};
			\draw (-1.1,-1.1,0) node{$1$};
			\draw (1.1,1.1,0) node{$1$};
			\draw (1.1,0,0) node{$1$};
			\draw (1.1,-1.1,0) node{$1$};
			\draw (1.1,-.6,0) node{$0$};
			\end{tikzpicture}
		\caption{The shape of the fan of $X$: \scriptsize{the “horizontal” plane corresponds to $\ker(1-\tau)$. We indicated the cohomology classes of the generators of the invariant rays by $0$s and $1$s. Here $a_{1,0}=7$, $a_{0,1}=2$, $t=1$, and $e=3$.}}
	\end{subfigure}
	\hfill
	\begin{subfigure}[t]{0.45\textwidth}
		\centering
		\includegraphics[width=\textwidth]{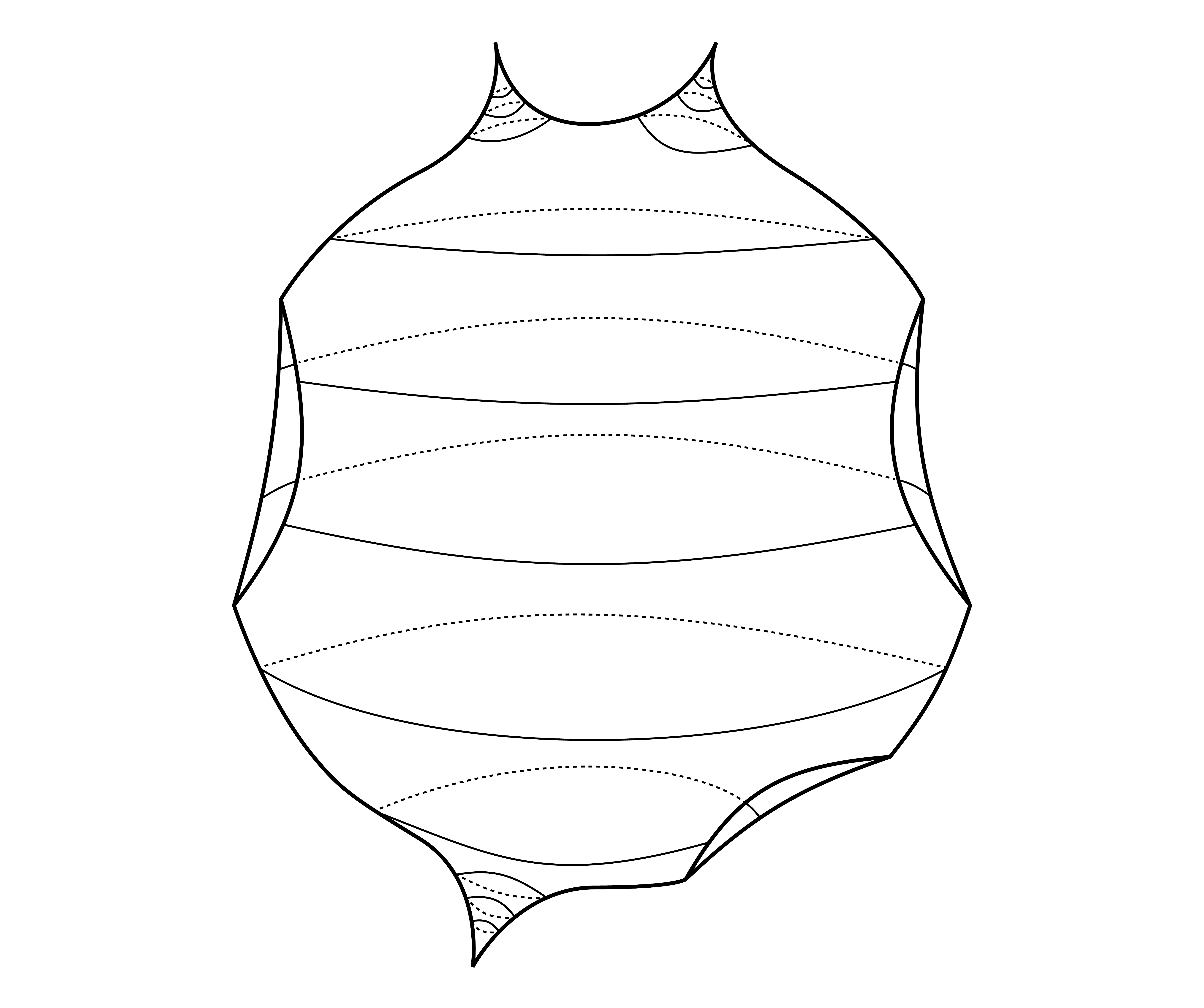}
		\caption{The quotient $X(\R)/\Sph^1$: \scriptsize{Two boundary components correspond to the two pairs of tridimensional cones, the other corresponds to the ray of cohomology class $0$. The three cusps correspond to the three bidimensional cones made of two rays of class $1$.}}
	\end{subfigure}
	\caption{The fan of an improperly wound toric variety $X$ of type $(2,1,)_1$ and the quotient of its real locus by the circle action.}
	\label{fig:quotient_circle}
\end{figure}

\begin{lem}\label{lem:hirzebruch_fixed_point}
	Let $\G{\R}^2\hookrightarrow X$ be a smooth and complete real \te. There exists a $\gamma\in \{\pm1\}^2$ that has only a finite number of fixed points in $X(\R)$ if and only if $X$ is a Hirzebruch surface of even parameter.  
\end{lem}

\begin{proof}
	 Let $\G{\R}^2\hookrightarrow X$ be a smooth and complete real \te. Let $N$ denote the cocharacter lattice of $\G{\R}^2$ and $\gamma\in N/2N$ be a non-trivial $2$-torsion element. The fixed point set of $\gamma$ in $X(\R)$ is the union of the real loci of the toric orbits $O(c)$ of $X$ such that the $2$-torsion of their isotropy group contains $\gamma$. Hence, such that $\gamma$ belongs to $N_c/2N_c$. Naturally, every toric fixed point belongs to $X(\R)^\gamma$. Let $D$ be a toric divisor whose ray is directed by a primitive vector $v$. The real locus $D(\R)$ belongs to $X(\R)^\gamma$ if and only if $v=\gamma\,(\mod 2N)$. Therefore, no primitive generator of the fan of $X$ must have the same “\,parity\,” as $\gamma$ if the fixed point set is to be finite. Let us consider $e\coloneqq e_{1,0}(X)$ the number of toric divisors of $X$. Since $X$ is complete, it is at least equal to $3$. We will show that if it is different from $4$ then one primitive generators of the rays of the fan of $X$ will be of the same parity as $\gamma$.\begin{enumerate}
	\item[$e=3$] In this case, the smoothness forces $X$ to be $\Proj^2_\R$. The three primitive generators of the rays of its fan satisfy $v_1+v_2+v_3=0$. The same relation must hold in $N/2N$. Since none of the vector is divisible by $2$, the relation imposes that $\{[v_1];[v_2];[v_3]\}=N/2N\setminus \{0\}$. Hence, one of them has the same parity as $\gamma$;
	\item[$e\geq 5$] Let $\{v_k\}_{k\in\Z/e}$ be a circular enumeration of the primitive generators of the rays of the fan of $X$. The classification of smooth equivariant embeddings of split tori imposes that, for some $k\in\Z/e$, we have $v_k=v_{k-1}+v_{k+1}$, \textit{cf.} \cite[\S2.5, third exercise]{Fulton:1993aa} or \cite[proof of Lemma~10.4.2]{Cox-Lit-Sch_tor_var}. With the same argument as in the previous case, one of the three vectors $v_{k-1},v_k,v_{k+1}$ must have the same parity as $\gamma$.
	\end{enumerate}
	Hence, if $X(\R)^\gamma$ is finite, $e=4$. Following again the classification of smooth equivariant embeddings of split tori, we know that $X$ is the Hirzebruch surface $F_{|b|}$ where the integer $b\in\Z$ satisfies $v_2+v_4=bv_1$ (under the assumption that the enumeration of the vectors satisfies $v_3+v_1=0$, \textit{cf.} \cite[\S1.1]{Fulton:1993aa}). We need to show that if $X(\R)^\gamma$ is finite then $b$ is even. Since $\{v_1;v_2\}$ and $\{v_1;v_4\}$ are bases of $N$, $v_1\neq v_2\,(\mod 2N)$ and $v_1\neq v_4\,(\mod 2N)$. Therefore, $v_4=v_2\,(\mod 2N)$ since none of $v_1,v_2,v_4$ is of the parity of $\gamma$. It follows that $bv_1$ belongs to $2N$, hence $b$ is even. Reciprocally, if $X$ is a Hirzebruch surface $F_{2m}$ then its fan is spanned by the primitive vectors $\partial x, \partial y,-\partial x,-\partial y + 2m\partial x$ in $\Z^2$ (up to isomorphism). Thus $\gamma=\partial x+\partial y\,(\mod 2\Z^2)$ satisfies the fixed point property.  
\end{proof}

\begin{lem}
	Let $T\hookrightarrow X$ be a real \te\;of type $(2;1)_1$ with compact real locus and smooth topological core. The polynomial $e^*[X]$ is uniquely determined by the numbers $e^0_{1,1}(X)$, $e^1_{1,1}(X)$, and $e^0_{1,0}(X)$.
\end{lem}

\begin{proof}
	We have:
	\begin{equation}
			e^*[X]=e^1_{2,1}(X)xz + e^1_{1,1}(X)z + e^0_{1,1}(X)xy + e^0_{1,0}(X)x+e^0_{0,1}(X)y+e^0_{0,0}(X).
	\end{equation}
	There is only one open orbit, so $e^1_{2,1}(X)=1$. Moreover, Proposition~\ref{eq:dehn_sommerville} yields the following relations:
	\begin{equation}
		\left\{\begin{aligned} e^1_{1,1}(X)+e^0_{1,1}(X)+e_{1,0}^0(X)&=e^0_{0,1}(X)+e^0_{0,0}(X) \\
		e_{0,0}^0(X)&=2e^0_{1,0}(X).
		\end{aligned}\right.
	\end{equation}
	Thus, one can express $e^*[X]$ in terms of $e^0_{1,1}(X)$, $e^1_{1,1}(X)$, and $e^0_{1,0}(X)$.
\end{proof}

\begin{thm}\label{thm:homeo_real_loc}
	Let $T\hookrightarrow X,Y$ be two real \te s of type $(2;1)_1$ with compact real loci and smooth topological cores. If $e^*[X]=e^*[Y]$, then $X(\R)$ is homeomorphic to $Y(\R)$. If $X(\R)$ is homeomorphic to $Y(\R)$, then $e^*[X]=e^*[Y]$ except when their real loci are homeomorphic to $\Proj^2(\R)\times\Sph^1$, in which case, their $e^*$-polynomials can either be $xz+2z+xy+3y$ or $xz+2z+xy+x+2y+2$.
\end{thm}

\begin{proof}
	Let $T\hookrightarrow X$ be a real \te\;of type $(2;1)_1$ with compact real locus and smooth topological core. The only possible exceptional circular orbits of $X(\R)$ are of unoriented type $(2;1)$. Thus, for any choice of orientation of the quotient $X(\R)/\Sph^1$, the Orlik-Raymond invariant of $X(\R)$ will be of the following form:
	\begin{equation*}
		\Big\{ b;\big(0;0;e^0_{1,1}(X);e^0_{1,0}(X)\big);\underset{u \textnormal{ times}}{\underbrace{(2;1);\cdots;(2;1)}}\Big\},
	\end{equation*} 
	where $u=e^1_{1,1}(X)-e^0_{1,1}(X)-e^0_{1,0}(X)$ and $b$ is an integer, \textit{cf.} Proposition~\ref{prop:circle_invariant}.
	
	\vspace{5pt}
	
	Let us first assume that $e^0_{1,0}(X)+e^0_{1,1}(X)$ vanishes. In this case, $X$ is properly wound. Indeed, the properness of the winding is equivalent to the vanishing of $e^0_{1,1}(X)$. Thus, $X(\R)$ is homeomorphic to $F(\R)\times^\Gamma \Sph^1$, \textit{cf.} (\ref{eq:joint_map_tor}). In the proof of Proposition~\ref{prop:circle_invariant}, we determined that the fixed locus $F(\R)^\Gamma$ is made of $e^0_{1,0}(X)+e^0_{1,1}(X)$ circles, and of $e^1_{1,1}(X)-e^0_{1,1}(X)-e^0_{1,0}(X)$ isolated points. Therefore, under the hypothesis $e^0_{1,0}(X)+e^0_{1,1}(X)=0$, this fixed point set is finite, and Lemma~\ref{lem:hirzebruch_fixed_point} garanties that $F$ is a Hirzebruch surface of even parameter. Thus, $F(\R)$ is homeomorphic to $\Sph^1\times\Sph^1$ and $\Gamma$ acts via an involution that have four fixed points. \cite[Corollary~5.8]{Dugger:2019aa} states that there is only one such involution, up to conjugation by a homeomorphism, namely the pillow case involution $(\zeta;\xi)\mapsto(\bar{\zeta};\bar{\xi})$. Thus, the real locus of $X$ is homeomorphic to the fibre product of two Klein bottles $\textnormal{Kl}$:
	\begin{equation}\label{eq:pillow_case}
		X(\R)\approx \textnormal{Kl}\underset{\Sph^1}{\times}\textnormal{Kl}.
	\end{equation}
	We deduce from the ninth case of \cite[Theorem~4]{Orlik:1968aa}, that the Orlik-Raymond invariant of $X(\R)$ will either be:
	\begin{equation*}
		\big\{b_0;(0;0;0;0);(2;1);(2;1);(2;1);(2;1) \big\}\textnormal{ or }\big\{-b_0-4;(0;0;0;0);(2;1);(2;1);(2;1);(2;1) \big\},
	\end{equation*}for some integer $b_0$. The ambiguity comes from the two possible orientations of the quotient. Furthermore, the same theorem asserts that no other Orlik-Raymond threefold will yield this particular threefold. Thus, given $X$, (\ref{eq:pillow_case}) holds if and only if $e^*[X]=xz+4z+4y$.
	
	\vspace{5pt}
	
	We assume now that $e^0_{1,0}(X)+e^0_{1,1}(X)$ is positive. In this case, $b$ has to vanish. Thus, regardless of a chosen orientation of the quotient, the Orlik-Raymond invariant is given by:
	\begin{equation*}
		\Big\{ 0;\big(0;0;e^0_{1,1}(X);e^0_{1,0}(X)\big);\underset{u \textnormal{ times}}{\underbrace{(2;1);\cdots;(2;1)}}\Big\},
	\end{equation*} 
	where $u=e^1_{1,1}(X)-e^0_{1,1}(X)-e^0_{1,0}(X)$. At this point, we have proved that $e^*[X]$ determines, up to a choice of orientation of the quotient, the Orlik-Raymond invariant of $\Sph^1\curvearrowright X(\R)$. \cite[Theorem 2]{Orlik:1968aa} ensures that if $e^*[X]=e^*[Y]$ then $X(\R)$ is homeomorphic to $Y(\R)$.
	
	\vspace{5pt}
	
	To prove the converse statement, we need to show that if $e^*[X]\neq e^*[Y]$ then $X(\R)$ is not homeomorphic to $Y(\R)$ except when $e^*[X]=xz+2z+xy+3y$ and $e^*[Y]=xz+2z+xy+x+2y+2$ in which case, both real loci are homeomorphic to $\Proj^2(\R)\times\Sph^1$. We have already showed that if $e^0_{1,1}(X)+e^0_{1,0}(X)$ vanishes then necessarily $e^*[X]=xz+4z+4y$, and $e^*[Y]\neq e^*[X]$ implies that $Y(\R)$ is not homeomorphic to $X(\R)$. Let us now assume that $e^0_{1,0}(X)+e^0_{1,1}(X)$ is positive. If $e^0_{1,1}(X)$ vanishes, then $e^0_{1,0}(X)$ is necessarily positive, and \cite[Theorem 5]{Orlik:1968aa} asserts that we have the following alternative:\begin{enumerate}
	\item[(i)] $e^*[X]=xz+z+xy+2y$ in which case $X(\R)$ is homeomorphic to $\Sph^1\times^{\Z/2}\Sph^2$ where $\Z/2$ acts on both factors by the antipodal involution (This is not actually possible for real \te s of type $(2;1)_1$ but we will address this in Proposition~\ref{prop:realised});
	\item[(ii)] $e^*[X]=xz+2z+xy+3y$ in which case $X(\R)$ is homeomorphic to $\Proj^2(\R)\times\Sph^1$;
	\item[(iii)] $e^*[X]=xz+2z+2xy+4y$ in which case $X(\R)$ is homeomorphic to $\textnormal{Kl}\times\Sph^1$;
	\item[(iv)] $e^*[X]$ does not belongs to the previous list. In this case, $X(\R)$ only admits the given circle action, i.e. if $e^*[Y]\neq e^*[X]$ then $Y(\R)$ is not homeomorphic to $X(\R)$. 
	\end{enumerate} 
	If $e^0_{1,1}(X)$ is positive then \cite[Theorem 3]{Orlik:1968aa} tells us that $X(\R)$ is homeomorphic to the connected sum:
	\begin{equation}\label{eq:prime_decomp}
		X(\R)\approx (e^0_{1,1}(X)-1)\cdot\big(\Sph^2\times\Sph^1\big)+e^0_{1,0}(X)\cdot\big(\Proj^2(\R)\times\Sph^1\big)+u\cdot\Proj^3(\R),
	\end{equation}
	where $u=e^1_{1,1}(X)-e^0_{1,1}(X)-e^0_{1,0}(X)$. This is the prime decomposition of $X(\R)$. The uniqueness of such decomposition, \textit{cf.} \cite{Hempel:1976aa}, ensures that the right hand side of (\ref{eq:prime_decomp}) determines $e^*[X]$ under the provision that $e^0_{1,1}(X)$ is positive. Furthermore, the only threefold arising as the right hand side of (\ref{eq:prime_decomp}) that was already listed is $\Proj^2(\R)\times\Sph^1$. The $e^*$-polynomials that leads to this real locus is $e^*[X]=xz+2z+xy+x+2y+2$. We showed that the only couple of polynomials leading to the same manifold are $xz+2z+xy+x+2y+2$ and $xz+2z+xy+3y$.
\end{proof}

As a biproduct of the proof of Theorem~\ref{thm:homeo_real_loc} we find the following proposition.

\begin{prop}
	Let $T\curvearrowright X$ be a real \te\;of type $(2;1)_1$ with compact real locus and smooth topological core. If $X$ is properly wound then $X(\R)$ is a prime threefold. If, on the contrary, $X$ is not properly wound, then $e^0_{1,1}(X)$ is positive and the prime decomposition of $X(\R)$ is the following:
	\begin{equation*}
		X(\R)\approx (e^0_{1,1}(X)-1)\cdot\big(\Sph^2\times\Sph^1\big)+e^0_{1,0}(X)\cdot\big(\Proj^2(\R)\times\Sph^1\big)+\big(e^1_{1,1}(X)-e^0_{1,1}(X)-e^0_{1,0}(X)\big)\cdot\Proj^3(\R).
	\end{equation*}
\end{prop}

\begin{prop}\label{prop:realised}
	Let $e=e^1_{2,1}xz + e^1_{1,1}z + e^0_{1,1}xy + e^0_{1,0}x+e^0_{0,1}y+e^0_{0,0}$ be a polynomial with non-negative integral coefficients. There is a real \te\;of type $(2;1)_1$ with compact real locus and smooth topological core whose $e^*$-polynomial is given by $e$ if and only if:
	\begin{equation}
		\left\{\begin{array}{ll} e^1_{2,1}=1 
		&e^1_{1,1}\geq e^0_{1,1}+e^0_{1,0}\\[5pt]
		e_{0,0}^0=2e^0_{1,0}
		&e^0_{0,1}+e^0_{0,0}\geq 3\\[5pt]
		e^1_{1,1}+e^0_{1,1}+e_{1,0}^0=e^0_{0,1}+e^0_{0,0}
		&(e^0_{1,1}+e^0_{1,0}=0 )\Rightarrow (e^1_{1,1}=4).
		\end{array}\right.
	\end{equation}
\end{prop}

\begin{proof}
	Propositions~\ref{eq:dehn_sommerville}~and~\ref{prop:circle_invariant} ensure the necessity of the statement. For the sufficiency, let us consider the lattice $N\coloneqq \Z[\tau]\oplus \Z[1]=\langle \partial x;\partial y;\partial z\rangle $, where $\tau$ acts as $\tau\partial x=\partial y$, $\tau\partial z=\partial z$. We need to exhibit a smooth equivariant fan $C$ of $N$ that contains $\langle  \partial x+\partial y;\partial z\rangle$ in its support, and that is made of:
\begin{center}
\begin{tabular}{p{.0425\textwidth} p{.4\textwidth} p{.0425\textwidth} p{.4\textwidth}}
	~~(i) & $e^1_{2,1}=1$ vertex; &
	~(ii) & $e^0_{1,0}$ pairs of exchanged rays; \\[2ex]
	(iii) & $e^1_{1,1}$ invariant rays spanned by a primitive vector with a non vanishing class in $H^2(\Z/2;N)$; &
	(iv) & $e^0_{1,1}$ invariant rays spanned by a primitive vector with a trivial class in $H^2(\Z/2;N)$; \\[7ex]
	~~(v) & $e^0_{1,0}$ bidimensional cones on a pair of exchanged rays; &
	(vi) & $e^0_{0,1}$ bidimensional cones on a pair of invariant rays; \\[4ex]
	(vii) & $e^0_{0,0}$ tridimensional cones. & \\[2ex]
\end{tabular}
\end{center}
\InsertBoxR{0}{\begin{minipage}{0.45\linewidth}\centering
\begin{tikzpicture}
	\fill (-1,-1) circle (.05cm);
	\fill (-1,0) circle (.05cm);
	\fill (-1,1) circle (.05cm);
	\draw (0,-1) circle (.05cm);
	\draw (0,0) circle (.05cm);
	\draw (0,1) circle (.05cm);
	\draw (0,1.3) node{$\partial x+\partial y$};
	\fill (1,-1) circle (.05cm);
	\fill (1,0) circle (.05cm);
	\fill (1,1) circle (.05cm);
	\draw (1,.3) node{$\partial z$};
	\draw (0,0) -- (1.2,0);
	\draw (0,0) -- (-1.2,0);
	\draw (0,0) -- (1.1,1.1);
	\draw (0,0) -- (1.1,-1.1);
\end{tikzpicture}
\captionsetup{hypcap=false}
\captionsetup{width=0.8\linewidth}

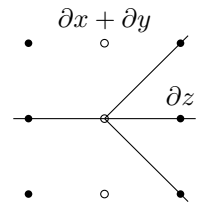
\captionof{figure}{A fan whose associated $e^*$-polynomial is given by  $xz+4z+4y$.}
\label{fig:fan_pillow_case}
\end{minipage}}[4]
In the basis $\{\partial x+\partial y;\partial z\}$ of the invariant subspace of $N$, the integer vectors with vanishing cohomology class are exactly those with an even $\partial z$-coordinate. We denote by $L$ the sublattice $\langle  \partial x+\partial y;2\partial z\rangle$ of $\langle  \partial x+\partial y;\partial z\rangle$. If $e^0_{1,1}+e^0_{1,0}$ vanishes then, the polynomial $e$ is given by $xz+4z+4y$. In this case, the fan of Figure~\ref{fig:fan_pillow_case} satisfies our requirements. From now on, we will assume that $e^0_{1,1}+e^0_{1,0}$ is positive. We will start by constructing a complete fan $C_2$ in $\langle  \partial x+\partial y;\partial z\rangle$ that we will lift to a fan of $N$. This will be the fan of the canonical fibre. We consider the following procedure:
%\InsertBoxR{0}{\begin{minipage}{0.4\linewidth}

\begin{itemize}
	\item[1\textsuperscript{st} Step:] We consider the complete fan $C_0$ of $\langle  \partial x+\partial y;\partial z\rangle$ made of the three rays $\R_+(\partial z)$, $\R_+(\partial x +\partial y)$, and $\R_+(-\partial x-\partial y-\partial z)$. Exactly one of them is spanned by a primitive vector contained in $L$;
	\item[2\textsuperscript{nd} Step:] We remark that the number $e^1_{1,1}$ is at least equal to $2$. Indeed, $2e^1_{1,1}\geq e^1_{1,1}+(e^0_{1,1}+e^0_{1,0})\geq 3$. Since $e^1_{1,1}$ is an integer, it is at least equal to $2$. Then, we construct the fan $C_1$. If the number $e^1_{1,1}$ is $2$, we set $C_1=C_2$, otherwise we subdivide $C_0$ into $C_1$ by adding the rays spanned by the vectors $\partial z+i(\partial x+\partial y)$, for all integers $1\leq i \leq e^1_{1,1}-2$. This is depicted in Figure~\ref{subfig:C0C1}. The fan $C_1$ has $1+e^1_{1,1}$ rays, only one of which is spanned by a primitive vector contained in $L$;
	\item[3\textsuperscript{rd} Step:] We assumed $e^0_{1,1}+e^0_{1,0}$ positive. If it equals 1 we set $C_2$ to be $C_1$. Otherwise, we subdivide $C_1$ by adding the rays spanned by the vectors $-(\partial x+\partial y)$ and $(2\partial z+(2i-1)(\partial x+\partial y))$ for all $1\leq i\leq e^0_{1,1}+e^0_{1,0}-2$ (if $e^0_{1,1}+e^0_{1,0}=2$ we only add the first ray). This procedure is illustrated in Figure~\ref{subfig:C1C2}. Every new ray lies exactly between two rays spanned by a primitive vector with non vanishing cohomology class for $e^1_{1,1}\geq e^0_{1,1}+e^0_{1,0}$. We can also note that $C_2$ is smooth and complete. It has $e^1_{1,1}$ rays spanned by a primitive vector of non-vanishing class and $e^0_{1,1}+e^0_{1,0}$ rays spanned by a primitive vector of vanishing class.
\end{itemize}
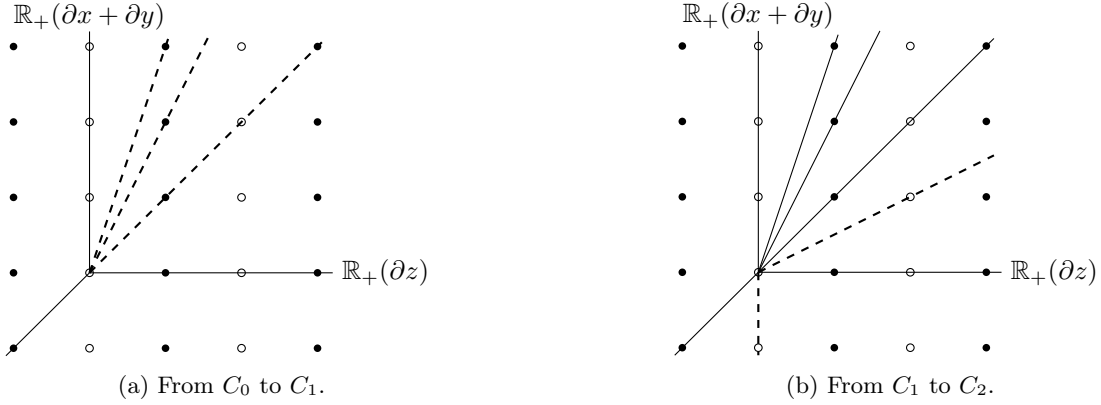
\begin{figure}[!ht]
\centering
\begin{subfigure}[t]{0.45\textwidth}
\centering
\begin{tikzpicture}
	\fill (-1,4) circle (.05cm);
	\fill (-1,3) circle (.05cm);
	\fill (-1,2) circle (.05cm);
	\fill (-1,1) circle (.05cm);
	\fill (-1,0) circle (.05cm);
	\draw (0,4) circle (.05cm);
	\draw (0,3) circle (.05cm);
	\draw (0,2) circle (.05cm);
	\draw (0,1) circle (.05cm);
	\draw (0,0) circle (.05cm);
	\fill (1,4) circle (.05cm);
	\fill (1,3) circle (.05cm);
	\fill (1,2) circle (.05cm);
	\fill (1,1) circle (.05cm);
	\fill (1,0) circle (.05cm);
	\draw (2,4) circle (.05cm);
	\draw (2,3) circle (.05cm);
	\draw (2,2) circle (.05cm);
	\draw (2,1) circle (.05cm);
	\draw (2,0) circle (.05cm);
	\fill (3,4) circle (.05cm);
	\fill (3,3) circle (.05cm);
	\fill (3,2) circle (.05cm);
	\fill (3,1) circle (.05cm);
	\fill (3,0) circle (.05cm);
	\draw (0,1) -- (0,4.2) ;
	\draw (0,4.4) node{$\R_+(\partial x+\partial y)$};
	\draw (0,1) -- (3.2,1)node[right]{$\R_+(\partial z)$};
	\draw (0,1) -- (-1.1,-0.1);
	\draw[thick, dashed] (0,1) -- (3.1,4.1);
	\draw[thick, dashed] (0,1) -- (1.6,4.2);
	\draw[thick, dashed] (0,1) -- (1.05,4.15);
\end{tikzpicture}
\caption{From $C_0$ to $C_1$.}
\label{subfig:C0C1}
\end{subfigure}
\hfill
\begin{subfigure}[t]{0.45\textwidth}
\centering
\begin{tikzpicture}
	\fill (-1,4) circle (.05cm);
	\fill (-1,3) circle (.05cm);
	\fill (-1,2) circle (.05cm);
	\fill (-1,1) circle (.05cm);
	\fill (-1,0) circle (.05cm);
	\draw (0,4) circle (.05cm);
	\draw (0,3) circle (.05cm);
	\draw (0,2) circle (.05cm);
	\draw (0,1) circle (.05cm);
	\draw (0,0) circle (.05cm);
	\fill (1,4) circle (.05cm);
	\fill (1,3) circle (.05cm);
	\fill (1,2) circle (.05cm);
	\fill (1,1) circle (.05cm);
	\fill (1,0) circle (.05cm);
	\draw (2,4) circle (.05cm);
	\draw (2,3) circle (.05cm);	
	\draw (2,2) circle (.05cm);
	\draw (2,1) circle (.05cm);
	\draw (2,0) circle (.05cm);
	\fill (3,4) circle (.05cm);
	\fill (3,3) circle (.05cm);
	\fill (3,2) circle (.05cm);
	\fill (3,1) circle (.05cm);
	\fill (3,0) circle (.05cm);
	\draw (0,1) -- (0,4.2);
	\draw (0,4.4) node{$\R_+(\partial x+\partial y)$};
	\draw (0,1) -- (3.2,1)node[right]{$\R_+(\partial z)$};
	\draw (0,1) -- (-1.1,-0.1);
	\draw (0,1) -- (3.1,4.1);
	\draw (0,1) -- (1.6,4.2);
	\draw (0,1) -- (1.05,4.15);
	\draw[thick, dashed] (0,1) -- (0,-.1);
	\draw[thick, dashed] (0,1) -- (2+1.1,2+.55);
\end{tikzpicture}
\caption{From $C_1$ to $C_2$.}
\label{subfig:C1C2}
\end{subfigure}
\caption{The construction of $C_2$ from $C_0$. The added rays at each step are dashed and the sublattice $L$ is represented by white dots.}
\end{figure}
To obtain the fan $C$, we will replace the first $e^0_{1,0}$ rays spanned by: 
\begin{equation*}
	(\partial x+\partial y),\quad -(\partial x+\partial y),\quad 2\partial z + (2i-1)(\partial x +\partial y),\;\forall 1\leq i\leq e^0_{1,1}+e^0_{1,0}-2,
\end{equation*}
by the bidimensional cones of the following list: 
\begin{equation*}
	\langle \partial x;\partial y\rangle_{\R_+},\quad\langle -\partial x;-\partial y\rangle_{\R_+},\quad \langle \partial z + (2i-1)\partial x ;\partial z+(2i-1)\partial y\rangle_{\R_+},\;\forall1\leq i\leq e^0_{1,1}+e^0_{1,0}-2.
\end{equation*}
Then, we replace the bidimensional cones of $C_2$ that were adjacent to the removed rays by tridimensional cones using the following procedure: if the ray $\rho$ is replaced by a bidimensional cone $c$ then a cone of the form $\rho+\rho'$ becomes $c+\rho'$. It is illustrated in Figure~\ref{fig:blow_down_ish}. In the end, we have the desired fan $C$.
	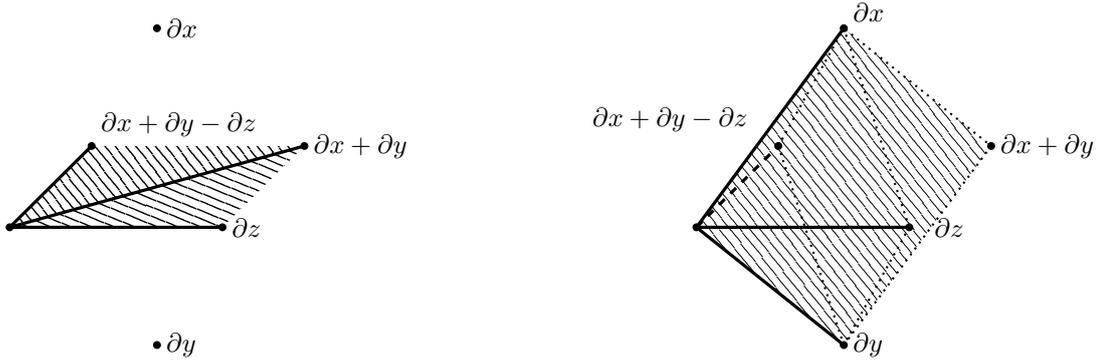
\begin{figure}[H]
		\centering
		\begin{subfigure}[t]{.48\textwidth}
		\centering
		\begin{tikzpicture}[scale=.7]
			\draw[very thick] (0,0,0) -- (4,0,0);
			\draw[very thick] (0,0,0) -- (0,0,-4);
			\draw[very thick] (0,0,0) -- (4,0,-4);
			\fill[pattern={Lines[angle=-52, yshift=4pt , line width=.5pt]}, thick] (0,0,0) -- (0,0,-4) -- (4,0,-4) -- cycle;
			\fill[pattern={Lines[angle=-22, yshift=4pt , line width=.5pt]}, thick] (0,0,0) -- (4,0,0) -- (4,0,-4) -- cycle;
			\fill (0,0,0) circle (.075cm);
			\fill (4,0,0) node[right]{$\partial z$} circle (.075cm);
			\fill (0,0,-4) circle (.075cm);
			\draw (0,0.4,-4)node[right]{$\partial x+\partial y-\partial z$};
			\fill (4,0,-4) node[right]{$\partial x+\partial y$} circle (.075cm);
			\fill (2,3,-2) node[right]{$\partial x$} circle (.075cm);
			\fill (2,-3,-2) node[right]{$\partial y$} circle (.075cm);
		\end{tikzpicture}
		\end{subfigure}
		\hfill
		\begin{subfigure}[t]{.48\textwidth}
		\centering
		\begin{tikzpicture}[scale=.7]
			\draw[very thick] (0,0,0) -- (4,0,0);
			\draw[very thick, dashed] (0,0,0) -- (0,0,-4);
			\draw[very thick] (0,0,0) -- (2,3,-2);
			\draw[very thick] (0,0,0) -- (2,-3,-2);
			\draw[thick, dotted] (4,0,0) -- (2,-3,-2);
			\draw[thick, dotted] (4,0,0) -- (2,3,-2);
			\draw[thick, dotted] (0,0,-4) -- (2,-3,-2);
			\draw[thick, dotted] (0,0,-4) -- (2,3,-2);
			\draw[thick, dotted] (2,-3,-2) -- (4,0,-4) -- (2,3,-2);
			\fill (0,0,0) circle (.075cm);
			\fill (4,0,0)  circle (.075cm);
			\draw (4.3,0,0)  node[right]{$\partial z$};
			\fill (0,0,-4) circle (.075cm);
			\draw (-.4,.5,-4)node[left]{$\partial x+\partial y-\partial z$};
			\fill (4,0,-4) node[right]{$\partial x+\partial y$} circle (.075cm);
			\fill (2,3,-2) circle (.075cm);
			\draw (2,3.3,-2) node[right]{$\partial x$};
			\fill (2,-3,-2) node[right]{$\partial y$} circle (.075cm);
			\fill[pattern={Lines[angle=-52, yshift=4pt , line width=.25pt]}, thick] (0,0,0) -- (2,-3,-2) -- (4,0,-4) -- (2,3,-2) -- cycle;
		\end{tikzpicture}
		\end{subfigure}
		\caption{Procedure to Obtain Orbits of Type $(1,0)_0$.}
		\label{fig:blow_down_ish}
	\end{figure}
\end{proof}

\begin{rem}
	We can deduce from Propositions~\ref{prop:smooth_equiv_completion}~and~\ref{prop:realised} that every threefold of the form $h\cdot\big(\Sph^2\times\Sph^1\big)+k\cdot\big(\Proj^2(\R)\times\Sph^1\big)+l\cdot\Proj^3(\R)$ where $h,k,l\geq 0$ are three integers whose sum is positive, can be realised as the real locus of a complete smooth \te\;of type $(2;1)_1$.
\end{rem}

Table~\ref{tab:top_type_3} summarises the discussion of this last section. It displays the types of topology a smooth and compact toric threefold under the action of a non-split torus can have. 

\begin{table}[H]
	\centering
	{\setlength{\extrarowheight}{5pt}
	\begin{tabular}{c||c|c|c|c}
	$(p;q)_r$  & (3;0) & (2;1) & (1;2) & (0;3) \\[5pt]
	\hline \hline 
	0 & ? & $\big((h\cdot\Proj^2(\R))\times \Sph^1\big)_{h\geq 1}$ ; $\big(\Sph^1\big)^3$ ; $\varnothing$ & $\big(\Sph^1\big)^3$ ; $\varnothing$ & $\big(\Sph^1\big)^3$ ; $\varnothing$ \\[5pt]
	\hline
	1  & n.a. & \begin{tabular}{c} $\big((h\cdot\Proj^2(\R))\hookrightarrow X \twoheadrightarrow \Sph^1\big)_{h\geq 1}$ \\ $(\Sph^1\times\Sph^1)\hookrightarrow X \twoheadrightarrow \Sph^1$ \\ $\left( \begin{array}{c}h\cdot(\Sph^2\times\Sph^1)\\+k\cdot(\Proj^2(\R)\times\Sph^1) \\+l\cdot\Proj^3(\R)\end{array}\right)_{\substack{h,k,l\geq 0 \\ h+k+l\geq 1}}$ \end{tabular} & \begin{tabular}{c} $\big(h\cdot\Proj^2(\R)\times\Sph^1\big)_{2\geq h\geq 0}$ \\ $\big( L(2k;l)\big)_{\textnormal{gcd}(2k;l)=1}$ \\ $\varnothing$ \end{tabular} & n.a.
	\end{tabular}}
	\caption{Topological Types of some Real Toric Threefolds. The upper left box is purposefully marked with a question mark as we did not determined these topological types. }
	\label{tab:top_type_3}
\end{table}

\bibliographystyle{apalike}
\bibliography{Ref_Real_Toric_Varieties_Interactions_between_their_Geometry_and_their_Topology}
\end{document}